\documentclass[12pt]{amsart}

\usepackage{amsmath}
\usepackage{amsfonts}
\usepackage{amssymb}
\usepackage{graphicx}
\usepackage{hyperref}
\usepackage{mathrsfs}
\usepackage{pb-diagram}
\usepackage{epstopdf}
\usepackage{amsmath,amsfonts,amsthm,enumerate,amscd,latexsym,curves}
\usepackage{bbm}
\usepackage[mathscr]{eucal}
\usepackage{epsfig,epsf,epic}
\usepackage{caption}
\usepackage{subcaption}
\usepackage{multirow}
\usepackage{color}
\usepackage[all]{xy}
\usepackage{fourier}
\usepackage{tikz-cd}
\usepackage{calc}
\usepackage{float}

\usepackage{comment}

\newtheorem{thm}{Theorem}[section]
\newtheorem{lemma}[thm]{Lemma}

\newtheorem{cor}[thm]{Corollary}
\newtheorem{prop}[thm]{Proposition}
\newtheorem{defn}[thm]{Definition}
\newtheorem{example}[thm]{Example}
\newtheorem{remark}[thm]{Remark}
\newtheorem{conjecture}[thm]{Conjecture}
\newtheorem{assumption}[thm]{Assumption}

\numberwithin{equation}{section}

\newcommand{\der}{\mathrm{d}}

\newcommand{\bb}{\boldsymbol{b}}

\newcommand{\AI}{A_\infty}

\newcommand{\Hom}{{\rm Hom}}

\newcommand{\OL}[1]{\overline{#1}}

\newcommand{\CG}{\mathcal{C}_\Gamma}
\newcommand{\cF}{\mathcal{F}}

\newcommand{\Z}{\mathbb{Z}}
\newcommand{\R}{\mathbb{R}}
\newcommand{\C}{\mathbb{C}}

\newcommand{\MF}{\mathrm{MF}}
\newcommand{\WF}{\mathcal{WF}}
\newcommand{\Mil}{M}
\newcommand{\Spec}{\mathrm{Spec}}

\newcommand{\bL}{\mathbb{L}}
\newcommand{\Cone}{\mathrm{Cone}}

\begin{document}

\title{Fukaya category for Landau-Ginzburg orbifolds}

\author[Cho]{Cheol-Hyun Cho}
\address{Department of Mathematical Sciences, Research Institute in Mathematics\\ Seoul National University\\ Seoul \\ South Korea}
\email{chocheol@snu.ac.kr}
\author[Choa]{Dongwook Choa}
\address{Korea Institute of Advanced Studies \\ Dongdaemoon-gu\\Seoul\\South Korea}
\email{dwchoa@kias.re.kr}
\author[Jeong]{Wonbo Jeong}
\address{Department of Mathematical Sciences, Research Institute in Mathematics\\ Seoul National University\\ Gwanak-gu\\Seoul \\ South Korea}
\email{wonbo.jeong@gmail.com}

 \begin{abstract}
For a weighted homogeneous polynomial and a choice of a diagonal symmetry group, we define a new Fukaya category for a Landau-Ginzburg orbifold (of Fano or Calabi-Yau type).
The construction is based on the wrapped Fukaya category of its Milnor fiber together with the monodromy of the singularity, and it is analogous to the variation operator in singularity theory. The new $\AI$-structure is constructed using popsicle maps with interior insertions of the monodromy orbit. This requires  new compactifications of popsicle moduli spaces  where conformal structures of some of the spheres and discs are aligned due to the popsicle structures. In particular,  codimension one popsicle sphere bubbles might exist and become obstructions to define the $\AI$-structure. For log Fano and Calabi-Yau  cases, we  show that the sphere bubbles do not arise from action and degree estimates, together with the computation of indices of twisted Reeb orbits for Milnor fiber quotients. 

	\end{abstract}

\maketitle
 \tableofcontents


\section{Introduction}
Singularity theory is a fascinating branch of mathematics with a long history and has deep relations to many branches of
mathematics, such as algebraic, complex and symplectic geometry, Lie groups and algebras, commutative algebra and 
mathematical physics.  Classifications as well as the topology and geometry of singularities have been well-established (see \cite{AGV1},\cite{AGV2}).
In commutative algebras, Cohen-Macaulay rings as coordinate rings of singularities  have been investigated in 80's and
its indecomposable maximal Cohen-Macaulay modules has been classified for ADE singularities (see \cite{KN87},\cite{Yo}). Eisenbud
has shown that maximal Cohen-Macaulay modules are equivalent to matrix factorizations of a singularity \cite{E80}.
In mathematical physics, singularities are often called Landau-Ginzburg (LG for short) models, and together with a finite group $G$ preserving the singularity, they are called LG orbifolds.

There has been much attention to the mirror symmetry of singularities, which revealed quite unexpected connections to
different branches of mathematics. Symplectic study of Picard-Lefschetz theory by Seidel is mirror to the corresponding algebraic geometry of coherent sheaves \cite{S08}, and quantum singularity theory developed by Fan, Jarvis and Ruan on Witten equation is mirror to  the corresponding integrable hierarchies \cite{FJR13} just to name a few.

%
%

 Recently, there has been a lot of research on mirror symmetry between LG orbifolds, called   Berglund-H\"ubsch mirror symmetry
 \cite{BH93}. 
 A polynomial 
   \[W(x_1, \ldots, x_n)=\sum_{i=1}^n \prod_{j=1}^n x_j^{a_{ij}}\] 
 is called \textit{invertible} if the matrix of exponents $A=(a_{ij})$ is an $n\times n$ invertible matrix. Its \textit{Berglund-H\"ubsch dual} $W^T$
 is an invertible polynomial whose exponent matrix is the transpose $A^T$. The group of diagonal symmetries
   \[G_W = \left\{(\lambda_1,\ldots,\lambda_n)\in(\C^*)^n \mid W(\lambda_1z_1,\ldots , \lambda_n z_n) = W(z_1,\ldots,z_n) \right\}.\] 
   also plays an important role. For a subgroup $G<G_W$, define its \textit{dual group} $G^T$  following \cite{BH95} (see also  \cite{Kra09}) as
   \[G^T:=\Hom(G_W/G, \C^*) \subset G_{W^T}.\]  
Berglund-H\"ubsch mirror symmetry is a duality between two Landau-Ginzburg orbifolds 
$$(W, G) \stackrel{\textrm{mirror}}{\Longleftrightarrow}  (W^T, G^T)$$ For closed string mirror symmetry, Fan-Jarvis-Ruan defined quantum singularity theory (FJRW invariants) of $(W,G)$.
It should be mirror to Saito-Givental theory of $(W^T, G^T)$ (see \cite{FJR13}, \cite{Kra09} for example).

Kontsevich's homological mirror symmetry conjecture \cite{kontsevich94} for Berglund-H\"ubsch pairs predicts a derived equivalence between the following two categories;
$$ \textrm{Fukaya category of } (W, G)    \stackrel{\textrm{mirror}}{\Longleftrightarrow}   \textrm{Matrix factorization category of } (W^T,G^T) $$
The left hand side for a nontrivial subgroup $G$ of $G_W$ has {\bf{not}} been defined.
For the case of trivial subgroup $G=\{0\}$, the left hand side is defined as a \textit{Fukaya-Seidel category} of $W$, $\overrightarrow{FS}(W)$, which is a directed $\Z$-graded $\AI$-category of vanishing cycles \cite{S08}.
The right hand side is the dg-category of $G^T$-equivariant matrix factorizations $\mathcal{MF}(W^T, G^T)$.

This BH HMS conjecture has been studied extensively for the case of the trivial subgroup $G=\{0\}$ of $G_W$. In this case $G^T = G_{W^T}$ and  the RHS (together with $\Z$-grading) is known as the category of maximally graded matrix factorization of $W^T$.
For this case of trivial $G$, HMS conjecture of this form was proposed by Takahashi \cite{T10}, and there have been many interesting works in this direction. See Seidel \cite{Sei01}, Kajiura-Saito-Takahashi \cite{KST07}, Auroux-Katzarkov-Orlov \cite{AKO08}, Futaki-Ueda \cite{FU11}, \cite{FU13}, Lekili-Ueda \cite{LU18}, \cite{LU20}, Harbermann-Smith \cite{HS19} and the references therein. We also refer readers to  a nice survey by Ebeling \cite{Ebeling}.

The candidate for the Fukaya category of the pair $(W,G)$ for a nontrivial subgroup $G$ would be an orbifold version of Fukaya-Seidel category, but the latter is currently out of reach (see \cite{FU09}, Problem 3). The main difficulty is that one needs to perturb $W$ to Morse function $W_\epsilon$ to obtain a legitimate collection of Lagrangian vanishing cycles but this procedure destroys the original symmetry $G_W$. 

The purpose of this paper is to define the Fukaya category of the pair $(W,G)$ when $W$ is a weighted homogeneous polynomial of log Fano and Calabi-Yau type and $G$ is any subgroup of the diagonal symmetry group $G_W$.

For this purpose, we propose a new approach which does not require Morsification of $W$. Namely, our approach uses wrapped Fukaya category of Milnor fiber, maximal symmetry group $G_W$ and monodromy of the singularity.
Our approach is orthogonal to that of Fukaya-Seidel category in the sense that 
we will mainly work with non-compact Lagrangians ($K$ in Figure \ref{fig:Lef}) whereas
Fukaya-Seidel category uses vanishing cycles ($L$ in Figure \ref{fig:Lef}). 

\begin{figure}[h]
\includegraphics[scale=0.5]{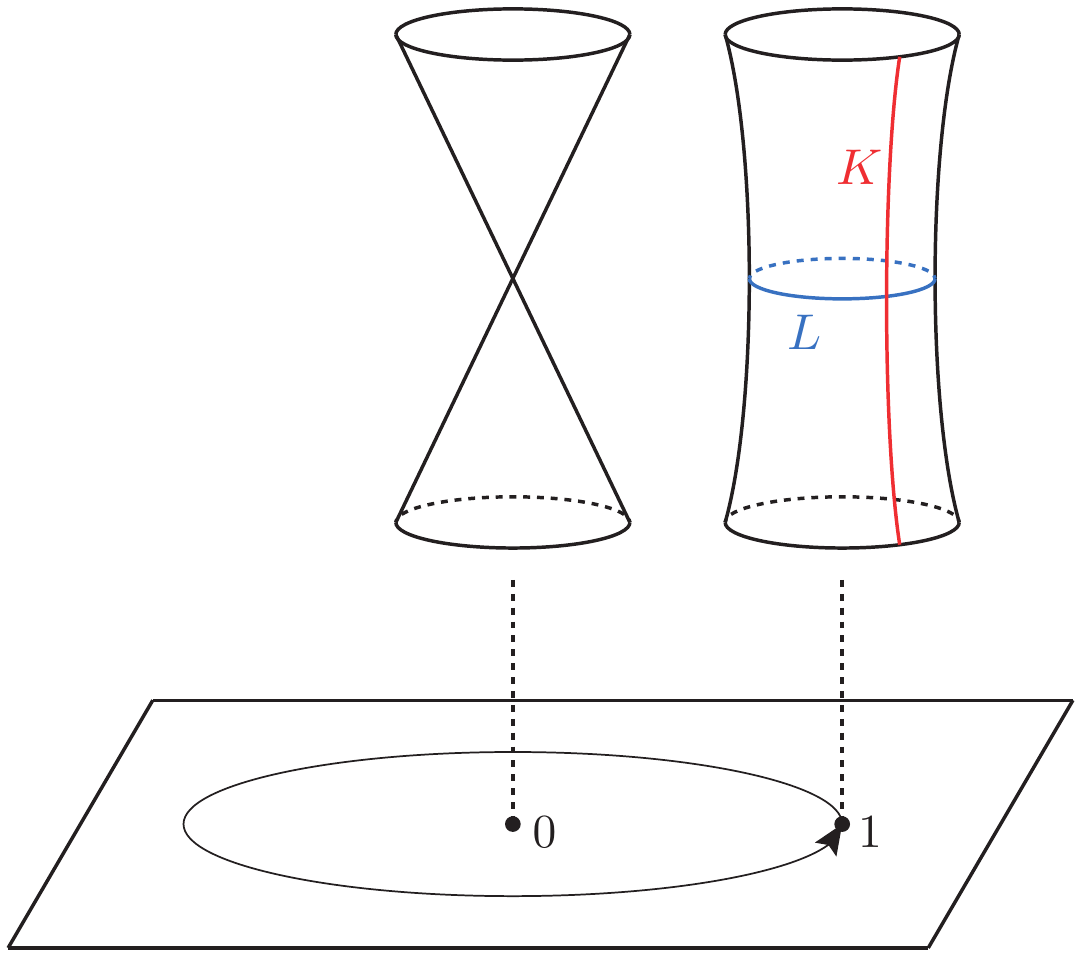}
\centering
\caption{Milnor fiber, vanishing cycle $L$ and noncompact Lagrangian $K$}
\label{fig:Lef}
\end{figure}

A topological precursor of our construction is a \textit{variation operator} 
\begin{equation}\label{eq:introv}
\mathrm{var}:H_{n-1}\left(\OL{M}_W,\partial \OL{M}_W\right) \to H_{n-1}\left(\OL{M}_W\right),
\end{equation}
where $M_W=W^{-1}(1)$ is a Milnor fiber of $W$. It is defined as a difference of a cycle itself and its monodromy image (fixing the boundary), and it provides an alternative way to describe vanishing cycles. 

Recall that a polynomial $W$ is called weighted homogeneous if 
$$W(\lambda^{w_1}z_1,\ldots, \lambda^{w_{n}}z_{n}) = \lambda^{h}W(z_1,\ldots,z_{n})$$
for $w_1,\ldots,w_{n},h \in \mathbb N$ with $\gcd(w_1, \ldots, w_n, h)=1$. 
We say $W$ has weight $(w_1,\ldots, w_{n}; h)$. 
$W$ is assumed to be a weighted homogeneous polynomial for the rest of the paper.

Monodromy homeomorphism (not  fixing the boundary) for $W$ is known to
be an action given by its weights (Milnor \cite{Milnor},  see \eqref{eq:gm}). 

We define what we call, the {\bf monodromy orbit} $\Gamma_W$ as follows.  We first observe that monodromy homeomorphism is an action of an element $J \in G_W$, hence becomes trivial for the quotient orbifold $[M_W/G_W]$. But on the link of the singularity, monodromy (not fixing the boundary) may be taken as time one Reeb flow of this contact manifold.
Therefore on the quotient orbifold  $[M_W/G_W]$, they form  principal orbits on its contact boundary.  We   define a distinguished Hamiltonian orbit $\Gamma_W$  of $[M_W/G_W]$ representing fundamental class of these principal orbits. 

\begin{figure}[h]
\begin{subfigure}{0.43\textwidth}
\includegraphics[scale=0.6]{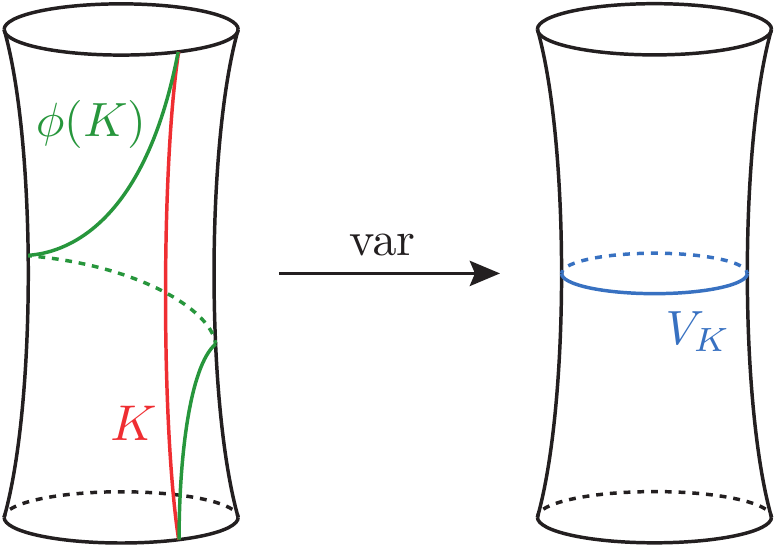}
\centering
\caption{Variation}
\end{subfigure}
\begin{subfigure}{0.43\textwidth}
\includegraphics[scale=0.6]{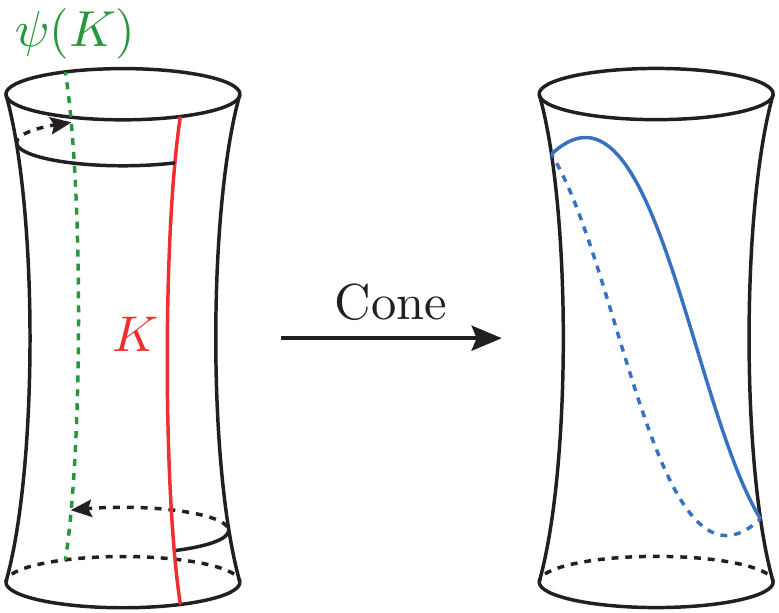}
\centering
\caption{Monodromy flow and its cone}
\end{subfigure}
\centering
\caption{Two ways of representing vanishing cycle}
\label{fig:var}
\end{figure}

Let us explain this in the case of $W= x^2+y^2$.
Monodromy $\phi$ of $W$ which fixes the boundary of the Milnor fiber, is a right-handed Dehn twist
as illustrated in  Figure \ref{fig:var} (A). Monodromy $\psi$ given by the action of weights is
nothing but $(x,y) \to (-x,-y)$, and at the boundary $p$ of the Milnor fiber, there is a Reeb flow from $p$ to $\psi(p)$ as illustrated in Figure \ref{fig:var} (B).
The quotient of the Milnor fiber by the $G_W$-action is an orbifold sphere $\mathbb{P}^1_{2,2,\infty}$ with a puncture and two $\Z/2$-orbifold points. The boundary Reeb flow from monodromy becomes a Reeb orbit winding around the puncture once. The monodromy orbit $\Gamma_W$ is the fundamental class of these principal Reeb orbits.

We remark that monodromy (fixing the boundary) on symplectic Lefschetz fibration has played important roles also in the study of Fukaya-Seidel category as well. For example, Seidel has shown that monodromy provides a natural transformation from identity to monodromy functor \cite{S08}. From this point of view, triviality of the monodromy for $[M_W/G_W]$ should provide a natural transformation from identity functor to itself which is an element of Hochschild cohomology of the wrapped Fukaya category, or  a symplectic cohomology class \cite{Se06}. This should be (conjecturally) the geometric orbit $\Gamma_W$ that we described above.

Now, let us explain our construction of the new Fukaya category using the orbit $\Gamma_W$.
First, classical variation operator may be viewed as the coequalizer sequence for monodromy $\phi$,
\[\begin{tikzcd} H_{n-1}\left(\OL{M}_W,\partial \OL{M}_W\right)  \arrow[r, shift left=1.2, "\phi"] \arrow[r, shift right=1.2, swap, "id"] & H_{n-1}\left(\OL{M}_W,\partial \OL{M}_W\right)  \arrow[r, "var"] &H_{n-1} \left(\OL{M}_W \right)  \arrow[r] &0 \end{tikzcd}\] 
To construct a symplectic analogue, we consider the quantum cap action of $\Gamma_W$ on the wrapped Fukaya category of $[M_W/G_W]$ given by a version of closed-open map. It gives an $\AI$-bimodule map $\cap \Gamma_W: \mathcal{WF}([M_W/G_W]) \to  \mathcal{WF}([M_W/G_W])$. As an analogue of coequalizer, we consider the cone of $\cap \Gamma_W$, which is again an $\AI$-bimodule over $\mathcal{WF}([M_W/G_W])$. 
We show that it also carries an $\AI$-category structure by constructing higher $\AI$-operations using $J$-holomorphic maps from popsicles with $\Gamma_W$-insertions  (see Section \ref{sec:LGFukaya}). 
\begin{thm}[Theorem \ref{thm:fulg}]   
Let $W$ be a non-degenerate  weighted homogeneous polynomial $W(z_1,\cdots,z_n)$ and $G_W$ be its group of maximal diagonal symmetry group. Suppose that the weight $(w_1,\cdots,w_n;h)$ of $W$ satisfies
\begin{equation}\label{eq:FC}
(\sum_{i=1}^n w_i)  -h \geq 0.
\end{equation}

 Then, there is an $A_\infty$-category $\cF(W, G_W)$ which fits to a distinguished diagram of bimodules;
\[\begin{tikzcd}\mathcal{WF}([M_W/G_W]) \arrow[r, "\cap \Gamma_W"] & \mathcal{WF}([M_W/G_W]) \arrow[r] & \cF(W, G_W) \arrow[r] & \phantom{a} \end{tikzcd}\]
(For a precise definition of $\Gamma_W$ and its action $\cap \Gamma_W$, see Section \ref{sec:LGFukaya}).
\end{thm}
Cohomology groups of the $\AI$-category $\cF(W, G_W)$ can be computed as a cone of $\cap \Gamma_W$-action,
but the rest of the $\AI$-structure are not constructed algebraically but rather geometrically from  $J$-holomorphic maps from popsicles.
This requires an important step of popsicle compactifications and we will explain more about it below. Also \eqref{eq:FC} is equivalent to  $\Gamma_W$ having a non-negative  Robbin-Salamon index  as we show in Corollary \ref{cor:gin} that Robbin-Salamon index of the monodromy orbit is given by the weights
$$ \mu_{RS}(\Gamma_W) =  \frac{2 \big((\sum_i w_i) - h\big)}{h}.$$  More explanation will be given in the Proposition \ref{prop:ii} below.

 The constructed $\AI$-category $\cF(W, G_W)$ is the new Fukaya category for the maximal symmetry group $G_W$. For any subgroup $G < G_W$, we can define the associated Fukaya category using semi-direct product (following Seidel \cite{Se2})
$\cF(W, G):=\cF(W, G_W) \rtimes G^T$.
This gives the desired Fukaya category for any pair $(W,G)$.

Note that for the trivial subgroup case,  the constructed Fukaya category $\cF(W, \{0\})$ is $\cF(W, G_W) \rtimes G_{W^T}$,
and it is natural to conjecture the following.
\begin{conjecture}
The Fukaya category $\cF(W, G_W) \rtimes G_{W^T}$ is derived equivalent to the Fukaya-Seidel category of $W$.
\end{conjecture}
This conjecture is a categorical version of the classical theorem that variation operator \eqref{eq:introv} is an isomorphism and
Figure \ref{fig:var} (B) illustrates how  $\cap \Gamma_W$-action recovers vanishing cycles. 

It may be surprising that Hom spaces of wrapped Fukaya category is  infinite dimensional whereas we expect
that  $\cF(W, G_W)$ have finite dimensional Hom spaces (We will show this in the case of curve singularities in the sequel).
This is because most of (but not all) wrapped generators are killed by the quantum cap action of $\Gamma_W$.

The conjectural relationship would be quite interesting. Recall that  for $\overrightarrow{FS}(W)$, one chooses a Morsification of $W$ and define the vanishing cycles of this Morsification as objects. Other vanishing cycles from  a different choice of vanishing paths or a Morsification are included only as twisted complexes of initial vanishing cycles.  On the other hand, $\cF(W,G)$ includes all (non-compact) Lagrangians as objects, and under the categorical variation operator, some of them would correspond to the vanishing cycles.
Furthermore, directness of Fukaya-Seidel category is imposed as a definition, but should be an intrinsic property of $\cF(W,G)$.
We will discuss further on the conjecture elsewhere in the future.

Now let us explain about the main ingredients for the proof of the main theorem.

A crucial step in the construction of $\AI$-category  $\cF(W, G_W)$ is to define a new compactification of  popsicles with interior marked points.  Recall that popsicle structure of a disc was introduced and developed by Abouzaid-Seidel \cite{AS} and Seidel \cite{Se18}
in a different geometric context.
Roughly speaking, given a disc $D^2$ with boundary marked points $\{z_0,\cdots,z_n\}$, one considers (hyperbolic) geodesic lines
connecting one of $z_i$ and $z_0$ in $D^2$, and  interior marked points of a popsicle disc  are only allowed to live on these lines.
In the construction of Abouzaid and Seidel, interior marked points are allowed to overlap as they were used as a place to put
the support of certain one forms. 

Unlike \cite{AS} or \cite{Se18}, we will use interior marked points to place the monodromy Hamiltonian orbit $\Gamma_W$
to define a new $\AI$-category  $\cF(W, G_W)$. Therefore, interior marked points should not overlap with each other and
 when interior marked points collide in a sequence, they will create sphere bubbles in accordance with the conventional Floer theory.
Thus, it is necessary to compactify popsicles differently from \cite{AS} and \cite{Se18}.

In fact, the convergence becomes much more delicate  compared to the cases of Abouzaid and Seidel (we thank Paul Seidel for this observation). Let us explain it in more detail.

When the interior marked points collide and create sphere bubbles, the popsicle geodesic lines also collide
and induces a popsicle structure on the sphere bubble as well.
But depending on the relative positions of marked points, the resulting popsicle structure on the sphere bubble may or may not be trivial.
Namely, geodesic lines may remain distinct  in the limit bubble or may become a single line in the limit bubble.
If popsicle lines remain distinct in the bubbled sphere, the conformal structures of this sphere bubble and a certain disc component (with the same set of non-trivial popsicle lines) should be related.  Thus the standard compactification using stable maps is not enough to
describe the compactifications.

To overcome these difficulties, we define the concept of {\em alignment data} which captures such additional relations between
conformal structures of discs and spheres.
We define a new compactification of popsicle discs using stable popsicles with alignment data.
In this new compactification, dimension of a stratum also becomes subtle because alignment might give restrictions to the possible conformal structures.
We find that stratum with   sphere bubbles can be of codimension one
when  sphere bubbles are aligned to a popsicle disc.
For example, the stratum in Figure \ref{fig:twobubble} is codimension one!

\begin{figure}[h]
\includegraphics[scale=0.5]{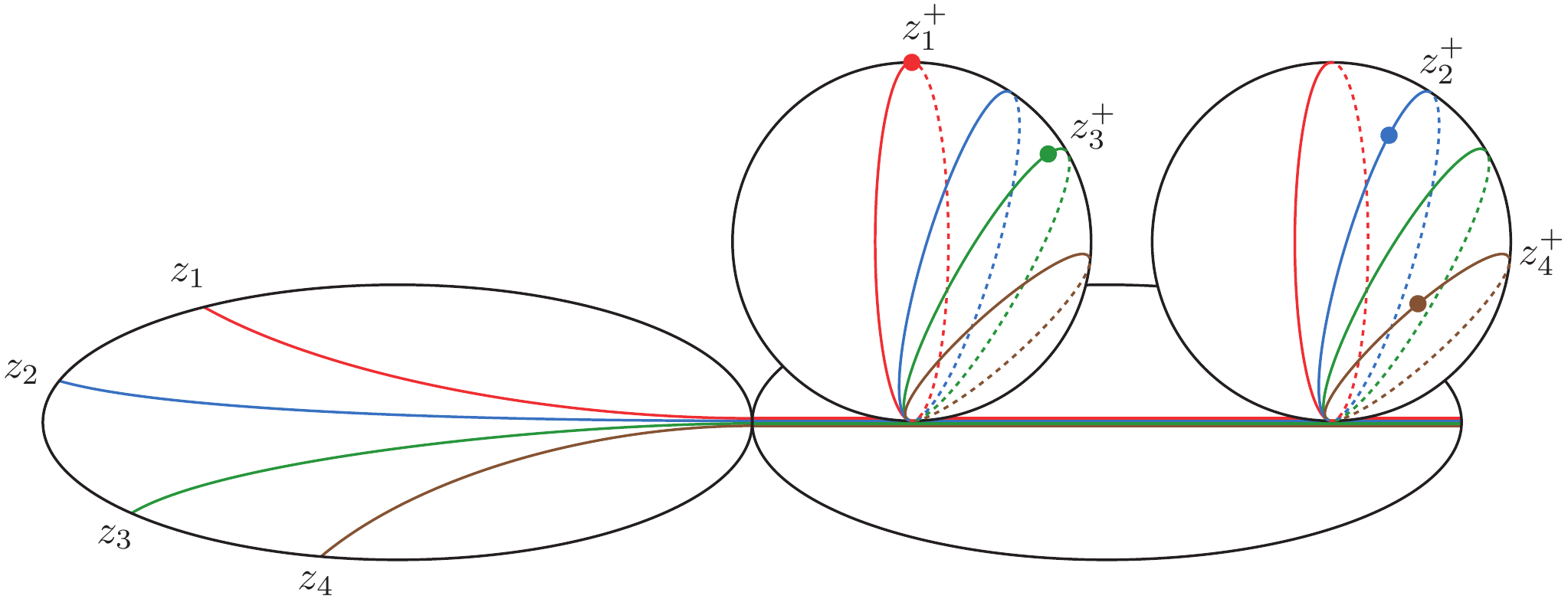}
\centering
\caption{A possible codimension one stratum in the popsicle compactification}
\label{fig:twobubble}
\end{figure}

It is interesting to note that similar phenomena were already observed in the beginning of 90's in the study of
modular operads. Fox-Neuwirth \cite{FN} introduced a cellular decomposition of $S^{2n} = \mathbb{R}^{2n} \cup \{\infty\}$,
from the configuration of $n$-points on $\mathbb{R}^{2n}$. Getzler-Jones in  \cite{GJ} developed the relationship between operads and homotopy algebras, and studied modular operad to prove a version of Deligne's conjecture. 
The above phenomenon of sphere bubbles of codimension one also appears in this context and provide obstructions to
the cellular structure of the relevant moduli spaces (see Voronov \cite{V} where Getzler-Jone's approach was completed to prove the conjecture). 

We can analogously define the compactification of the moduli space of popsicle spheres, and
we expect to obtain new homotopy algebraic structures on symplectic cohomology, such as (shifted) $L_\infty$- 
structure on symplectic cohomology. We will explore this elsewhere.

Coming back to our construction of the new Fukaya category, the desired $\AI$-equation holds if these sphere bubbles of
codimension one do not contribute. Let us first  work on the case of a symplectic cohomology class $\Gamma$ in a Liouville manifold $M$
(satisfying Assumption \ref{as:main}).
If we consider popsicle discs with $\Gamma$ insertions, the following proposition turns out to be crucial.
\begin{prop}[Prop. \ref{prop:v1}, Prop. \ref{prop:vanishing2}]\label{prop:ii}
If Robbin-Salamon index of a principle Reeb orbit component is non-negative, then  sphere bubbles of codimension one (with $\Gamma$-insertions) do not arise.
\end{prop}
We will prove the proposition using  estimates of degrees and action values.
This proposition  enables us to exclude codimension one
sphere bubbles and obtain the $\AI$-relations as in	 \cite{AS} and \cite{Se18}.
This defines an $\AI$-category $\mathcal{C}_\Gamma(M)$ (see Section \ref{sec:qc1}).

We run into more  challenges with the construction
  in the case of the Milnor fiber quotient orbifold $[M_W/G_W]$ (to define $\cF(W,G_W)$).

The first one is to define the monodromy orbit $\Gamma_W$ representing principal Reeb orbits at the boundary.
This requires a choice of a Morse-Smale function on the contact boundary. As $G_W$ acts diagonally, orbifold strata are
somewhat special in the sense that they are always given by setting some of the coordinates to be zero. This can be
used to define a desired Morse-Smale function (Proposition \ref{equivariant Morse function}) whereas it is not always possible for a general orbifold.

The next difficulty is that a general definition of orbifold symplectic cohomology is not known. Our case is somewhat simpler
in the sense that we only use $\Gamma_W$ as insertions, but the output of a possible sphere bubble could be an arbitrary orbifold symplectic cochain.  Even though we will not define the orbifold symplectic cochain complex of $[M_W/G_W]$,
we will identify their generators and compute their Robbin-Salamon indices.

Namely, we will consider the space of $G_W$-twisted Reeb orbits, i.e., $\gamma:[0,l] \to M_W$ satisfying $g\cdot \gamma(1) = \gamma(0)$
for some $g \in G_W$, and $\gamma$ being a Reeb flow. The problem can be quite complicated as the period $l$ could be a fractional number.  Their Robbin-Salamon indices are explicitly computed in Proposition \ref{prop:RS}.
We obtain an inequality relating the weights of $W$ and the Robbin-Salamon indices of twisted Reeb orbits in Proposition \ref{orbifold degree inequality 1}. 
Hence the log Fano or Calai-Yau assumption \eqref{eq:FC} combined with this inequality implies that the associated Robbin-Salamon index is non-negative.
 This enables us to prove vanishing of sphere bubbles with $\Gamma_W$ insertions (see Proposition \ref{prop:vanishing2}). This finishes the construction of the new $\AI$-category $\cF(W,G)$.
 
Our construction heavily depends on the fact that $W$ is weighted homogeneous, especially in the definition of
the monodromy orbit $\Gamma_W$. But as variational operator exists for general $W$, it is natural to expect that 
its categorification exists in general as well. 
In the case that $W$ is a polynomial of two variables,  which are not necessarily weighted homogeneous, we
can indeed define such a category (in a joint work in progress with Hanwool Bae).

In the sequel, we will apply the construction in this paper  to find a geometric understanding of  Berglund-H\"ubsch Mirror symmetry.
Given an invertible polynomial $W$ of two variables, we can find the transpose polynomial $W^T$ entirely in a geometric way.
The Milnor fiber quotient $[M_W/G_W]$  has a Landau-Ginzburg mirror $W_\bL$ given as a Lagrangian Floer potential function of a Seidel Lagrangian $\bL$.
Then, the closed open map of the monodromy orbit $\Gamma_W$ can be used to define a polynomial $g$ via Kodaira-Spencer map (\cite{FOOO_MS}, \cite{ACHL}).
We can show that quantum cap action $\cap \Gamma_W$ corresponds to the restriction of the potential $W_\bL$ to the hypersurface $g=0$
(see also Section \ref{counterpart}). Then $W_\bL$ restricted to the hypersurface $g=0$ turns out to be the transpose polynomial $W^T$!

To exclude the sphere bubbles in our main construction, we need to require  log Fano or Calabi-Yau condition.
Unfortunately, only $x^2+y^2$ satisfies this condition and all the other invertible curve singularities are of log general type.
But we can still use low dimensionality to exclude certain sphere bubbles, and more detailed discussion will be given in the sequel.

 Let us finish the introduction by working out the example of $x^2+y^2$.
  \subsection{Example: $\boldsymbol{x^2+y^2}$}
To illustrate our construction, let us explain the case of $W=x^2+y^2$. Milnor fiber $M_W =\{x^2+y^2=1\}$ is $T^*S^1$ or a cylinder. Its zero section is a vanishing cycle of $W$ and its cotangent fiber $K$ generates its wrapped Fukaya category (see Figure \ref{fig:Lef}).
The maximal diagonal symmetry group $G_{W}$ is  $\Z/2 \times \Z/2$ and
the quotient $[M_W/G_W]$ is an orbifold sphere $\mathbb{P}^1_{2,2,\infty}$ with two $\Z/2$-orbifold points, say $A,B$ and a single puncture
$C$ as in Figure \ref{fig:A2quot}.

\begin{figure}[h]
\includegraphics[scale=0.6]{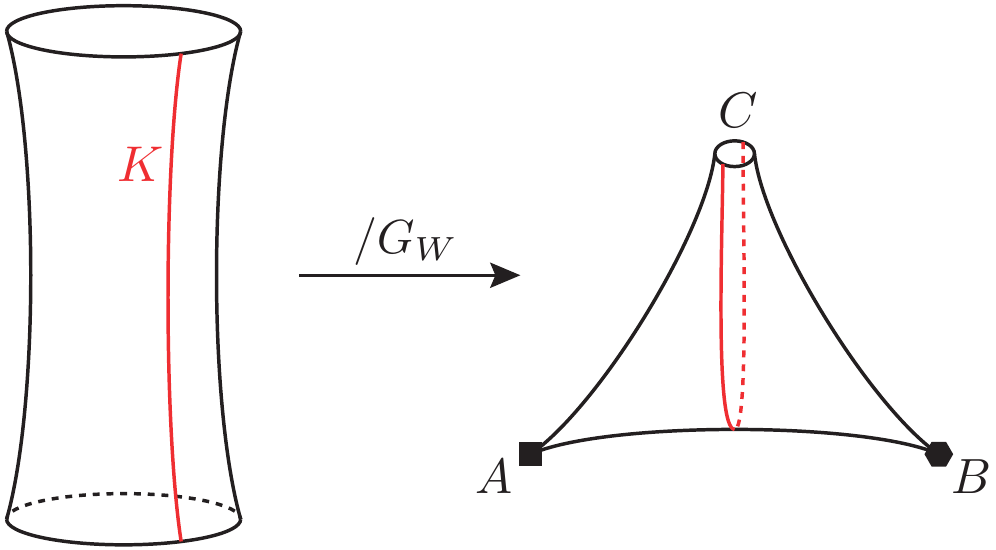}
\centering
\caption{Quotient of Milnor fiber}
\label{fig:A2quot}
\end{figure}

We claim that 
\[H^\bullet \left(\Hom_{\cF(W, G_W)}(K, K) \right) \simeq \mathrm{Cliff}_2(\mathbb C).\]
A category $\cF(W, G_W)$ is essentially the wrapped Fukaya category of $[M_W/G_W]$ $= \mathbb{P}^1_{2,2,\infty}$ with additional relation "$\Gamma_W=0$".  We realize it by considering quantum cap action of $\Gamma_W$ on wrapped Fukaya category of $\mathbb{P}^1_{2,2,\infty}$. As a complex, we have
\[\left( \Hom_{\cF(W, G_W)}(K, K) , d \right) \simeq \left(\begin{tikzcd}\left( CW^\bullet(K, K), m_1 \right)\arrow[r, "\cap \Gamma_W"] & \left( CW^\bullet(K, K), m_1 \right) \end{tikzcd}\right)\]
We view $K$ as an object of  the $\AI$-category $\mathcal{WF}([M_W/G_W])$ by considering $G_W$-orbit of $K$ consisting of $K, \psi(K)$ and their disjoint $\Z/2$-images.
Let us take two minimal wrapped generators, name them as $\alpha,\beta \in CW(K,K)$. The computation shows 
\[\left( CW^\bullet (K, K), m_1 \right) \simeq \left( \C\langle \alpha, \beta \rangle / \{\alpha^2 = \beta^2 = 1\}, m_1=0 \right) \] 
The element $\Gamma_W$ is an $S^1$ family of orbits (quotient of a Reeb flow drawn in Figure \ref{fig:var} (B)). The same flow determines an element $\alpha \beta$ and $\beta \alpha$ in $CW^\bullet(K, K)$. The map $\cap \Gamma_W$ turns out to be the multiplication by $\alpha\beta+\beta\alpha$. Therefore,
\[H^\bullet \left(\Hom_{\cF(W, G_W)}(K, K) \right) \simeq  \C\langle \alpha, \beta \rangle / \{\alpha^2 = \beta^2 = 1, \alpha\beta=-\beta\alpha\}, \]
as is claimed. 

Intuitively cone of a morphism corresponds to the Lagrangian surgery as illustrated in Figure \ref{fig:var} (B).
But we should note that the new category $\mathcal{C}_{\Gamma_W}([M_W/G_W]) =\cF(W, G_W)$ is different from the category of cones (where cones are considered as twisted complexes), and hence this surgery interpretation should be taken
only as an intuition.  The actual $\AI$-category is defined by using popsicle maps with $\Gamma_W$-insertions.

We would like to notice that this computation matches to the expectation from  Berglund-H\"ubsch version of homological mirror symmetry.
A mirror for $\mathbb{P}^1_{2,2,\infty}$ is a polynomial $x^2+y^2+xyz$. 
To see this, take Seidel's Lagrangian $\mathbb{L}$ (\cite{Se}), which is an immersed Lagrangian with three odd immersed generators $X,Y,Z$.
It is weakly unobstructed with bounding cochain $\bb = xX+yY+zZ$ and its Lagrangian
potential function is exactly $W_\bL = x^2+y^2+xyz$ (see \cite{CHL}).
This can be  seen by taking lifts of $\mathbb{L}$ to the cylinder, which gives four circles as in Figure \ref{fig:A2K}.
Pick a generic point, and count all rigid polygons passing through it. The reader can find two bigons with corners labeled by $X,X$ and $Y,Y$ together with
a minimal triangle $XYZ$.

\begin{figure}[h]
\includegraphics[scale=0.6]{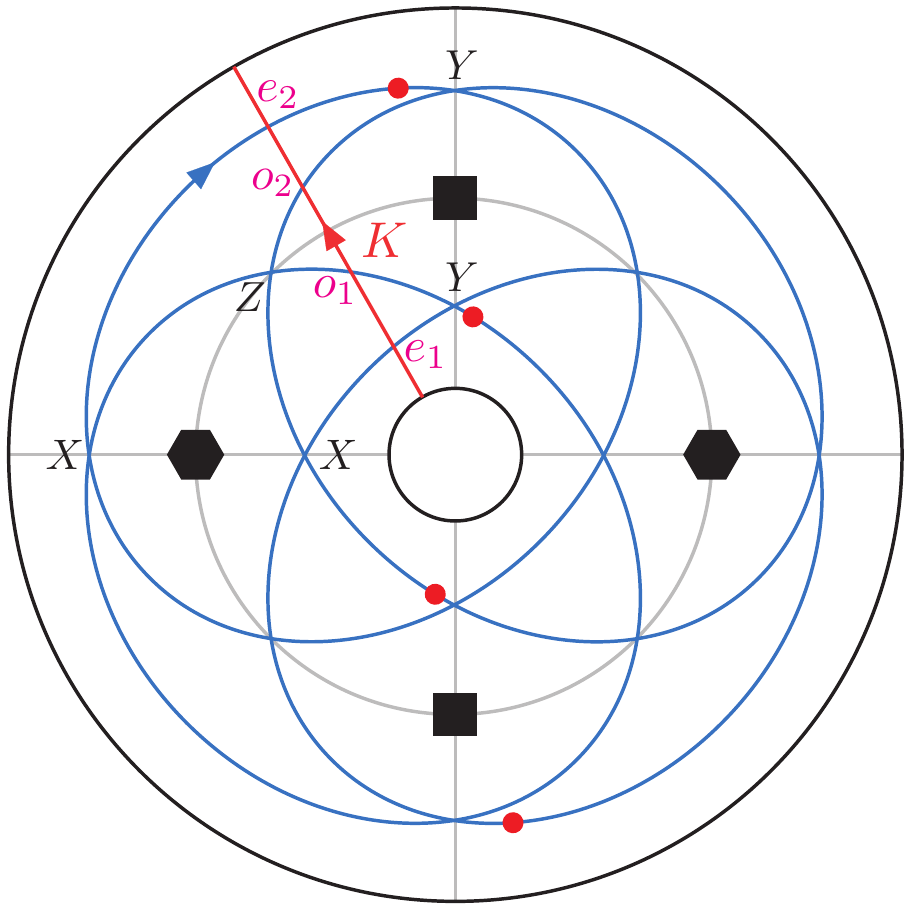}
\centering
\caption{Lifts of Seidel Lagrangian and $K$}
\label{fig:A2K}
\end{figure}

If we set $z=0$, $W_\bL$ becomes a desired dual polynomial $W^T=x^2+y^2$. 
This is related to the monodromy of $W$ as follows. Monodromy  of $W$ (fixing the boundary) is given by
a Dehn twist $\phi$ as in Figure \ref{fig:var} (A). We will consider a version of monodromy $\psi$ which does not fix the boundary
as in Figure \ref{fig:var} (B). It turns out that the monodromy $\psi$ on the Milnor fiber $M_W$ is
is the same as the action by the element $(-1,-1)$ of a maximal symmetry group $G_W$. Also, consider the Reeb flows
on the boundary $\partial \OL{M}_W$ describing the monodromy $\psi$, which give a Reeb orbit $\Gamma_W$ in the quotient $\mathbb{P}^1_{2,2,\infty}$. In this case, $\Gamma_W$ is the orbit that winds the puncture of the quotient once. Closed-open map takes this orbit to $z \cdot 1_\bL$.



We have $\AI$-functor from $\cF(W, G_W) \to \mathcal{MF}(W^T)$ which is induced by localized mirror functor \cite{CHL} (and setting $z=0$). It sends $K$ to Floer complex $\left(CW^\bullet (K,\mathbb{L}), -m_{1}^{0,\bb}\right) \bigg |_{z=0}$, which is a matrix factorization of $W^T$. We can calculate it in Figure \ref{fig:A2K} and it is given by
\begin{equation*}
-\begin{pmatrix}
y& x\\
x& -y\\
\end{pmatrix} \cdot
-\begin{pmatrix}
y& x\\
x& -y\\
\end{pmatrix}
\end{equation*}
This factorization is the compact generator of $\mathcal{MF}(W^T)$ considered in \cite{Dy},  and it's endomorphism ring is the Clifford algebra in two variables.


\subsection{Structure of the paper}
In Section \ref{sec:2}, we explain the  new phenomenon that arises in popsicle  compactifications with interior insertions, and present the
concept of  alignment data using a number of instructive examples. In Section \ref{sec:3}, we give a formal definition of
the compactified moduli space of popsicles with interior insertions.
In Section \ref{sec:qc0}, we recall the definition of quantum cap action and define $J$-holomorphic popsicles with  insertions of
a fixed symplectic cohomology cycle $\Gamma$ in a Liouville manifold $M$. 
In Section \ref{sec:qc1},  we construct an $A_\infty$-category $\CG$ from wrapped Fukaya category of $M$ on which the quantum cap action of $\Gamma$ vanishes using $J$-holomorphic popsicles with $\Gamma$-insertions.
This construction
will be generalized to orbifold quotients of Milnor fibers of weighted homogeneous polynomials in Section  \ref{sec:LGFukaya}.
A parallel construction in algebraic geometry is described in Section \ref{counterpart}, which is a restriction to a hypersurface.
In Section  \ref{sec:LGFukaya}, we explain our choice of distinguished Hamiltonian orbit $\Gamma_W$ for a Milnor fiber of a given weighted homogeneous polynomial $W$ which encodes the monodromy information. Then we define a new $\AI$-category $\cF(W,G)$ by generalizing the construction of $\CG$ to this particular class of orbifolds. In order to show its well-definedness, we classify Reeb orbits of a Milnor fiber and computes their Conley-Zehnder indices in Section \ref{sec: index computation}. Finally, we show that  $\cF(W,G)$ is well-defined for a log Fano/Calabi-Yau type polynomial in Section \ref{sec:proof2}.

In Appendix \ref{basic Floer theory}, we briefly describe the moduli spaces and perturbation scheme we use throughout the paper.  Appendix \ref{compactification} explains a compactification of popsicle moduli spaces. 


\subsection{Acknowledgement}
We would like to thank Paul Seidel for informing us a crucial error in the previous version of our compactification of popsicles and the encouragements. We would like to thank Otto van Koert, Hanwool Bae for the discussion on symplectic cohomology theory on orbifolds, and popsicle compactifications. We would like to thank Atsushi Takahashi, Philsang Yoo, Kaoru Ono for helpful discussions. The second author was partially supported by the T.J.Park science fellowship grant. 

\section{New compactification of popsicles with interior marked points: examples}\label{sec:2}
Let us first informally explain the new compactification of popsicles through several examples of increasing complexity, and
give a formal treatment later in the section.

Let us first recall the  definition of a popsicle by Abouzaid and Seidel \cite{AS} (we follow Seidel  \cite{Se18} to use rational ends).
\begin{defn}[\cite{AS},\cite{Se18}]\label{defn:popdisk}
A popsicle is a disc $D^2$ with following decorations;
\begin{enumerate}
      \item \textbf{boundary marked points}: denoted by  $z_0, z_1, \ldots, z_n$ according to their cyclic order. 
      \item \textbf{popsicle lines}: regarding the interior of $D^2$ as a hyperbolic disc, the geodesic connecting $z_i$ and $z_0$ (at infinity), denoted as $Q_i$.
      \item \textbf{flavor}: a set of flavors $F = \{1,\ldots,l \}$ and a non-decreasing map 
       \[\phi: F \to \{1,\ldots, n\}.\] 
      \item \textbf{sprinkles (interior marked points)}:   $l$ distinct interior marked points $z_1^+,\ldots,z_l^+$, such that $z_j^+$ lies on the geodesic $Q_{\phi(j)}$ for
      $j=1,\ldots,l$.
    \end{enumerate}
We say that the above popsicle has type $(n,\phi)$ (see Figure \ref{fig:popsicle}).
We call it \textit{stable} if $n+l \geq 2$. We denote a moduli space of popsicles of type $(n,\phi)$ by  $P_{n, F, \phi}$.
 \end{defn}
 
	\begin{figure}[h]
	\centering
	\def\svgwidth{5cm}
	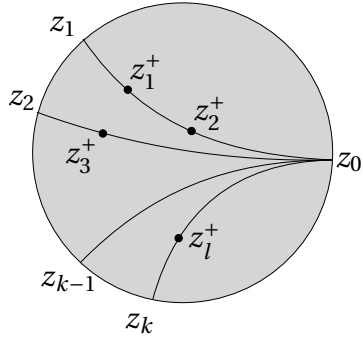
	\caption{Example of a popsicle such that $\phi(1)=\phi(2)=1$ and $\phi(l)=k$.}
	\label{fig:popsicle}
	\end{figure}

Abouzaid and Seidel defined compactifications of popsicle moduli spaces while allowing the interior marked points to coincide.
For example, $z_1^+$ and $z_2^+$ in Figure \ref{fig:popsicle} may coincide in their limit. 
But we will not allow our interior marked points to coincide, as they will be used to place inputs.

For example, let us consider a popsicle of type $(2,\phi)$, which has two sprinkles with the identity map $\phi: \{1,2\} \to \{1,2\}$.
Our compactification of the moduli space of type $(2,\phi)$ is illustrated in the Figure \ref{fig:popdeg}.
To obtain the  Abouzaid-Seidel's compactification, one has to contract the bottom edge to a point (see Figure 3 of \cite{AS}),
as sphere bubble does not arise in the limit.

Note that at the two vertices of  the bottom edge in the Figure \ref{fig:popdeg}
two popsicle lines  coincide in the limiting sphere bubble.
On the interior of the bottom edge, sphere bubbles inherit two popsicle lines from discs, and popsicle lines of the sphere and disc are aligned.

\begin{figure}[h]
\includegraphics[scale=0.45]{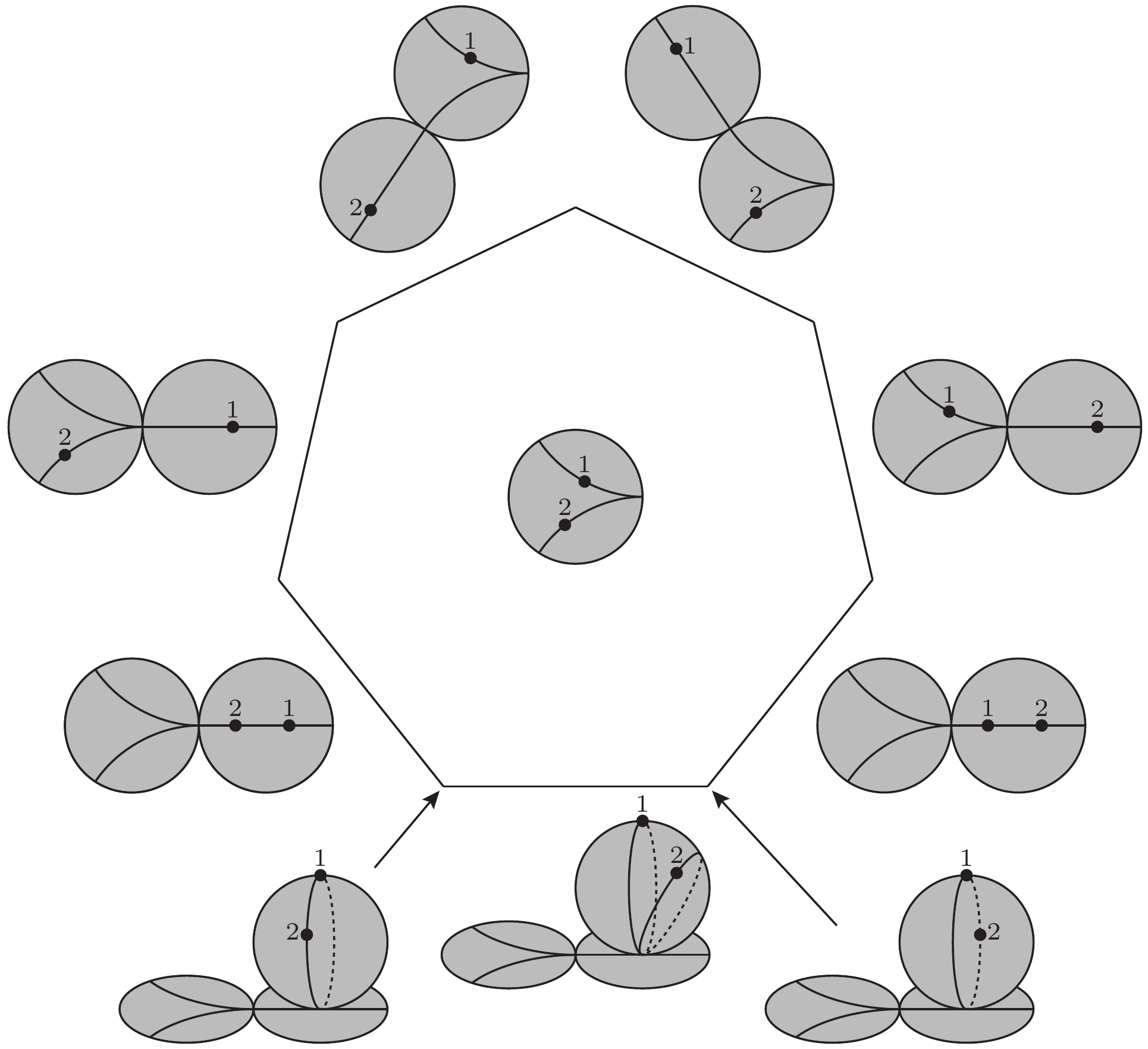}
\centering
\caption{Compactification of popsicles when sprinkles cannot overlap with each other}
\label{fig:popdeg}
\end{figure}
 
\subsection{Aligning conformal structures on a disc and a sphere}
We introduce the notion of a popsicle sphere, and explain the notion of alignment between a popsicle disc and a popsicle sphere.
\begin{defn}
A popsicle sphere is the following decoration on a rational sphere $\mathbb P^1$;
\begin{enumerate}
  \item \textbf{fixed outgoing marked point point}: denoted by $w_0^-\in \mathbb P^1$,
  \item \textbf{incoming marked points}: denoted by $\{w_{i}^+ \vert 1\leq i \leq n\}$,
  \item \textbf{popsicle lines}: regarding $\mathbb P^1\setminus w_0^-$ as a complex plane, a set of $m$ distinct oriented lines on a plane with the same direction. We can arrange them using rotation of a complex plane so that they are all oriented upward. We denote them by $Q_1, \ldots, Q_m \subset \C$ from the left to the right after this arrangement.
  \item a map $\phi: \{w_1^+, \ldots, w_n^+\} \to \{1, \ldots, m\}$ such that $w^+_{i}$ lies on $Q_{\phi(w_i^+)}$.
\end{enumerate} 
If $n \geq 2$, then we call it stable. We write such a moduli space of popsicle structure by $\mathbb P_{m, n, \phi}$.
\end{defn}

	\begin{figure}[h]
	\centering
	\def\svgwidth{10cm}
 	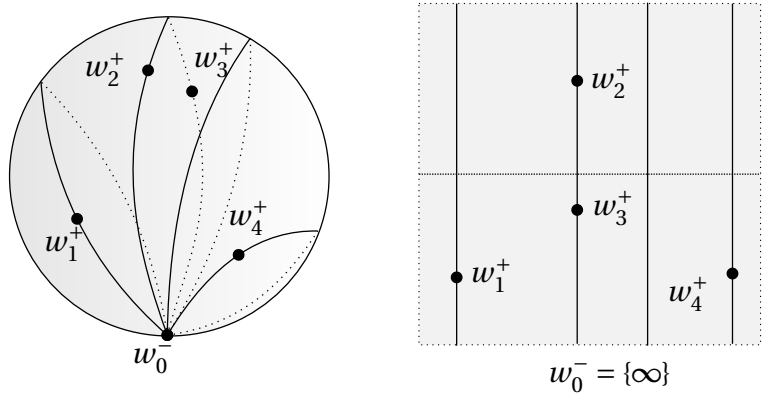
	\caption{Example of a popsicle sphere and its planar model}
	\label{fig:sphere}
	\end{figure}
 
Popsicle structure on a sphere and a disc can be compared in a natural way.
First, take out  $z_0$ (resp. $w_0^-$) from $D^2$ (resp.  $\mathbb{P}^1$) and fix an identification with the upper half plane $H$ (resp. $\C$
via stereographic projection).

In the case of a disc, boundary marked points $z_1,\cdots, z_n$ are determined by ordered real numbers $x_1, \ldots, x_n$, and 
a popsicle line between $z_0$ and $z_i$ on $D^2$ correspond to the vertical (upward) lay at $x_i$. Thus the position of sprinkle $z_i^+$
is determined by its $y$-coordinate.
\[z_i \mapsto x_i\in \R, \hskip 0.3cm Q_i \mapsto \{\mathrm{Re}z=x_i\}, \hskip 0.3cm z_j^+ \mapsto \left(x_{\phi(j)}+\sqrt{-1} y_j\right ) \in Q_{\phi(j)}.\]
The forgetful map $f: P_{n,F,\phi} \to \mathcal M_{0, n+1}$ is well-defined.  

In the case of a sphere, we may further assume that $w_1^+$ maps to $\R \in\C$ and each and every $Q_i$ head strictly upward. Then popsicle lines are determined by ordered real numbers $x_i$ and sprinkles are determined by real numbers $y_j$;
  \[Q_i \mapsto \{\mathrm{Re}z = x_i\}, \hskip 0.3cm w_1^+ \mapsto x_{\phi(w_1^+)}, \hskip 0.3cm w_j^+ \mapsto \left(x_{\phi(w_j^+)} +\sqrt{-1}y_j\right) \in Q_{\phi(j)}.\]    
We get a well-defined map $f: \mathbb P_{m,n,\phi} \to \mathcal M_{0, m+1}$, just remembering $x$-coordinates of popsicle lines.

\begin{defn}
We say a popsicle disc $\Sigma$ and a popsicle sphere $\Sigma^+$ are \textbf{aligned} if their forgetful map images coincide in $\mathcal M_{0, m+1}$. Namely, we have
\[\left(\Sigma, \Sigma^+\right) \in P_{m, F, \phi} \times_{\mathcal M_{0,m+1}} \mathbb P_{m, n, \phi'}.\] 
 When $\Sigma_1$ and $\Sigma^+$ are aligned, we will write  
$$\Sigma  \Join \Sigma^+   \;\;\; \textrm{or} \;\;\;  \Sigma^+ \Join \Sigma  $$
\end{defn}
Intuitively, if popsicle lines when seen in the upper half plane $H$ and $\C$ agree with each other, then they are aligned. Note also that a disc and a sphere with exactly two popsicle lines each, are always aligned with each other. 

A  popsicle which consists of discs and sphere bubbles  is called stable if it is stable in the usual sense.
Namely, each sphere component has at least 3 special points
(marked or nodal), and each disc component has either at least three special points (marked or nodal) or one boundary and
one interior nodal point.  Hence, stable popsicle do not contain a sphere component with several popsicle lines but with just one or zero sprinkle.
But in order to describe a neighborhood of a stable popsicle, sometimes we need to consider such unstable components for the gluing. Such unstable components will be uniquely determined from a given stable popsicle and thus they will be called auxiliary data.

 \begin{remark}
Let us remark on the alignment for the unstable case of only one popsicle line. The space of pairs $(\Sigma, \Sigma^+)$ have extra one dimension because $\mathcal M_{0, 1+1} \simeq [\bullet/ \mathbb R]$ is a stacky point. Explicitly, suppose two different sphere component $\Sigma_1^+$ and $\Sigma_2^+$ are isomorphic by a homothety automorphism. Even though $\Sigma_i^+$ represent isomorphic element in $\mathbb P_{1,F', \phi'}$, a pair $(\Sigma, \Sigma_1^+)$ and $(\Sigma, \Sigma_2^+)$ represents different element in the fiber product. This ambiguity provides extra one more dimension. 
\end{remark}

Let us discuss examples of compactifications and the alignment therein.
\subsection{Examples of new compactifications}
Let us start with the simplest example, which appears at the bottom stratum of Figure \ref {fig:popdeg}.
\begin{example}\label{ex:pop21}
Consider  $(D^2, z_0,z_1,z_2)$ with two sprinkles $z_1^+, z_2^+$, whose type is $(2, \phi)$ where $\phi:\{1,2\} \to \{1,2\}$ is an identity map.
The moduli space of popsicles of type $(2,\phi)$ has dimension two.

\begin{figure}[h] 
\begin{subfigure}[t]{0.4\textwidth}
\raisebox{7ex}{
\includegraphics[scale=0.45]{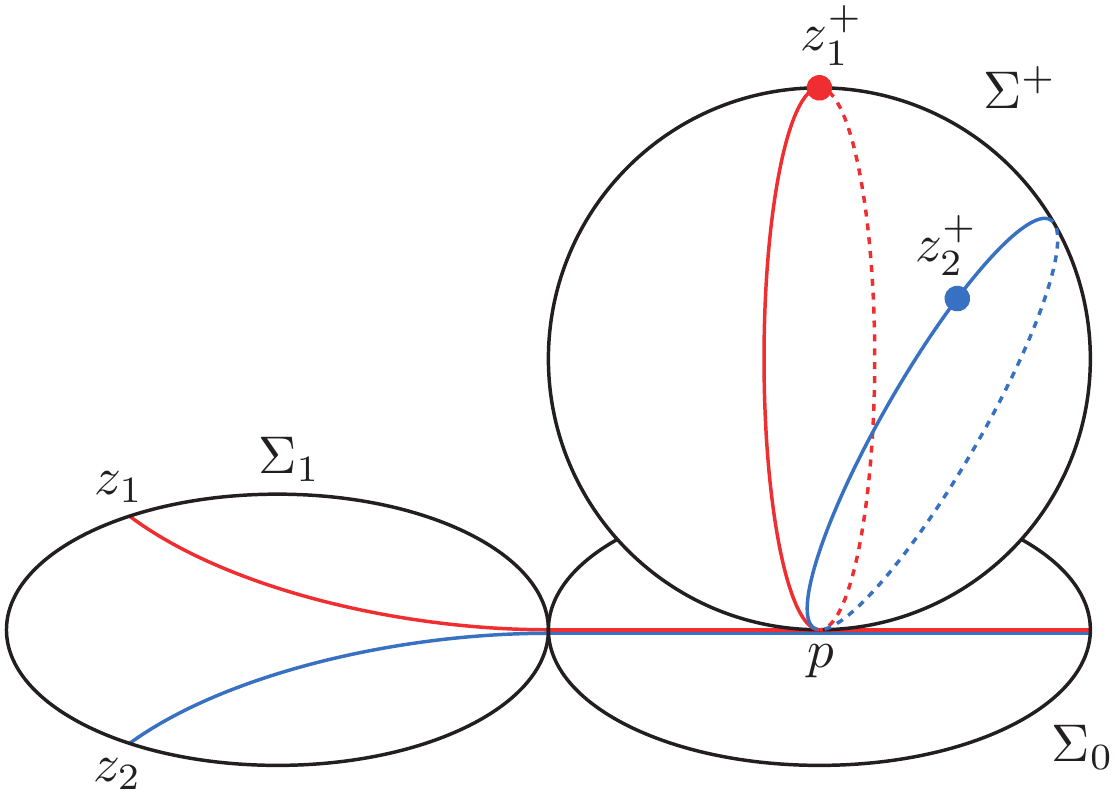}}
\centering
\caption{ }

\end{subfigure}
\qquad
\begin{subfigure}[t]{0.4\textwidth}
\includegraphics[scale=0.45]{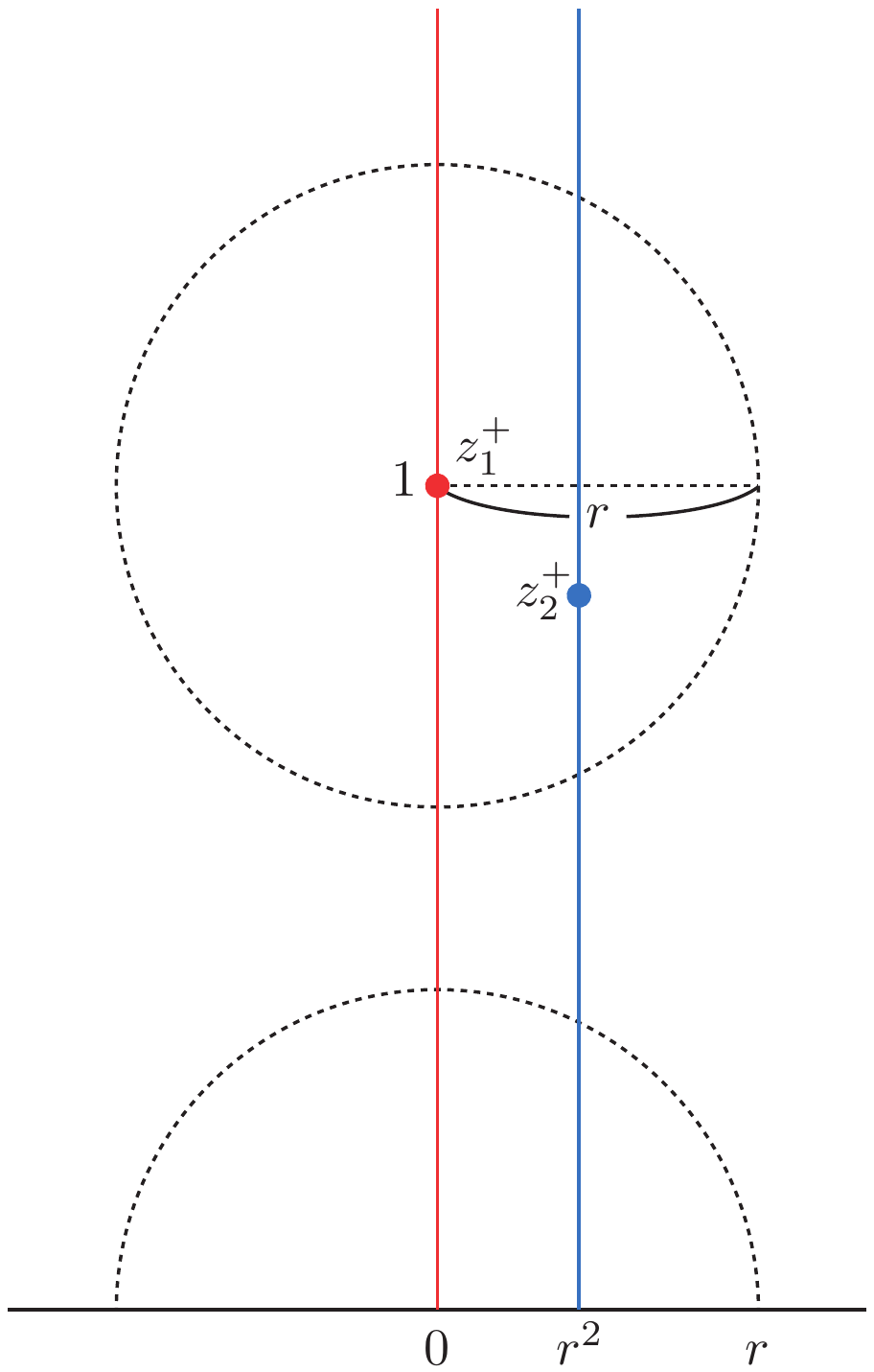}
\centering
\caption{ }
\label{fig:glue1}
\end{subfigure}
\centering
\caption{Popsicle with sphere bubble and gluing using planar model 1}
\label{fig:popdeg1}
\end{figure}

The new codimension one stratum is illustrated in Figure \ref{fig:popdeg1} (A).
Denote by  $\Sigma^+,\Sigma_0,\Sigma_1$, the
unique sphere component, the middle disc $\Sigma^+$ and the popsicle disc as in the figure, and by $\Sigma$ the whole stable popsicle.
Here, we will require that $\Sigma^+ \Join \Sigma_1$ (this condition is automatically satisified  in this case of $n=2$, but
will be non-trivial for the general case $n>2$).

This limit can appear if $z_1^+, z_2^+$ approach to each other (producing a sphere bubble) and at the same time they approach $z_0$
(producing a disc bubble).  This stratum at first might appear to be of codimension two or higher, but it is of codimension one, hence is one dimensional.

Note that $\Sigma_0,\Sigma_1$ has the unique conformal structure up to automorphism.
Since $\Sigma_0 \Join \Sigma^+$, there is one dimensional automorphisms of $\Sigma^+$
which preserve $p$ and two popsicle lines and we use it to fix the position of $z_1^+$.
Thus  there only remains the freedom to choose  $z_2^+$ on the popsicle line, which gives one dimensional family as suggested.

It is instructive to describe the explicit neighborhood  of $\Sigma$ using the gluing parameter $r$( $0< r \ll 1$) and
the parameter $y$ for the position of $z_2^+$. Here  $y$  is the vertical coordinate of $z_2^+$ in $\C \cong \Sigma^+ \setminus \{p\}$ when $z_1^+$ is at the origin of $\C$.
Given $(r,y)$, we can describe the planar model (see Figure \ref{fig:popdeg1} (B)) of the glued disc:  upper half plane is $\Sigma_0$
and $\Sigma^+$ is glued at $(0,1)$ and $\Sigma_1$ is glued at $(0,0)$.
Note that we use the same gluing parameter $r$ for the aligned disc and sphere so that popsicle lines are aligned and extend to all of $H$. 
We can place $z_2^+$ on the second vertical ray whose height is  $r^2y + 1$. This describes regular popsicles (with $r>0$) near $\Sigma$.
\end{example}
We can generalize this example to the case of $n$ sprinkles. Namely, consider $\Sigma_1$ with
$(n+1)$ marked points and $\Sigma^+$ with $n$ popsicle lines. This configuration will be
again of codimension one, with the similar planar model with several popsicle lines.

 \begin{example}
Consider the new codimension two stratum is illustrated in Figure \ref{fig:popdeg2} (which appear as
a vertex at the bottom of Figure \ref{fig:popdeg}). 
Denote the components of Figure \ref{fig:popdeg2}(A) by $\Sigma^+$, $\Sigma_0, \Sigma_1$ as before.
In this case, two popsicle lines on $\Sigma^+$ coincide and 
two sprinkles are placed on it. Note that an automorphism preserving the popsicle lines and fixing two sprinkles $z_1^+,z_2^+$ should be trivial and this stratum is 0 dimensional.

\begin{figure}[h] 
\begin{subfigure}[t]{0.4\textwidth}
\raisebox{7ex}{
\includegraphics[scale=0.45]{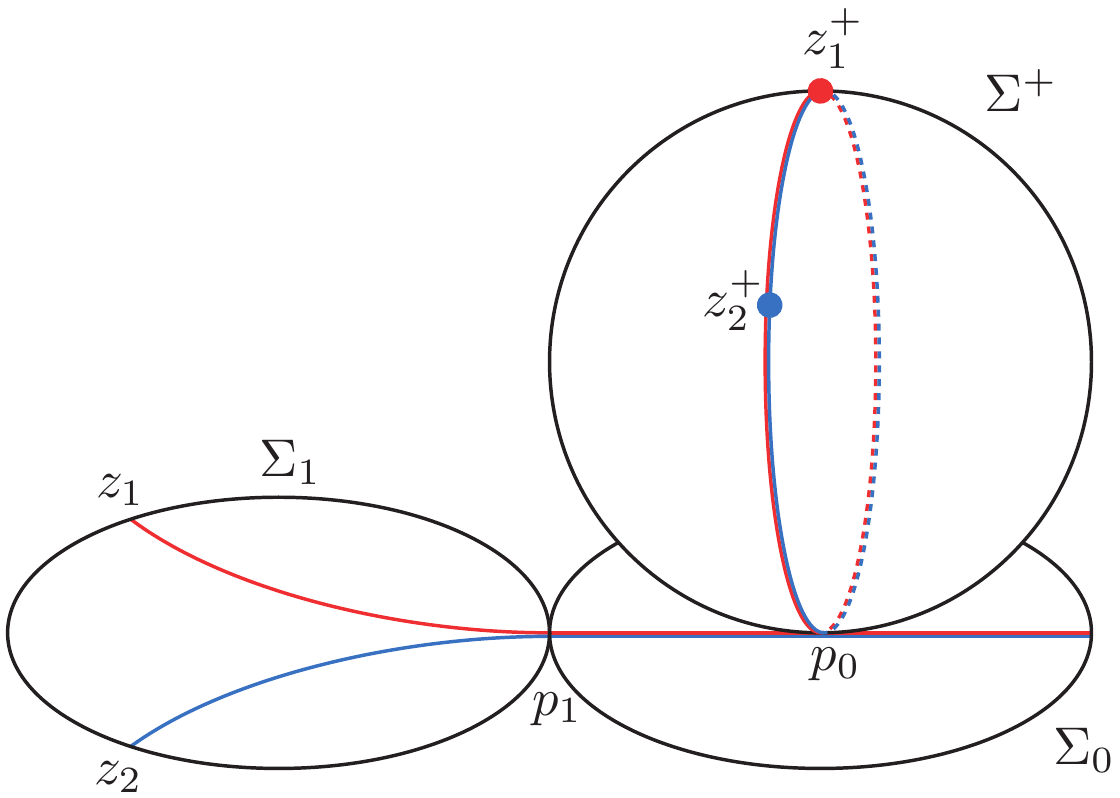}}
\centering
\caption{ }
\end{subfigure}
\qquad
\begin{subfigure}[t]{0.4\textwidth}
\includegraphics[scale=0.45]{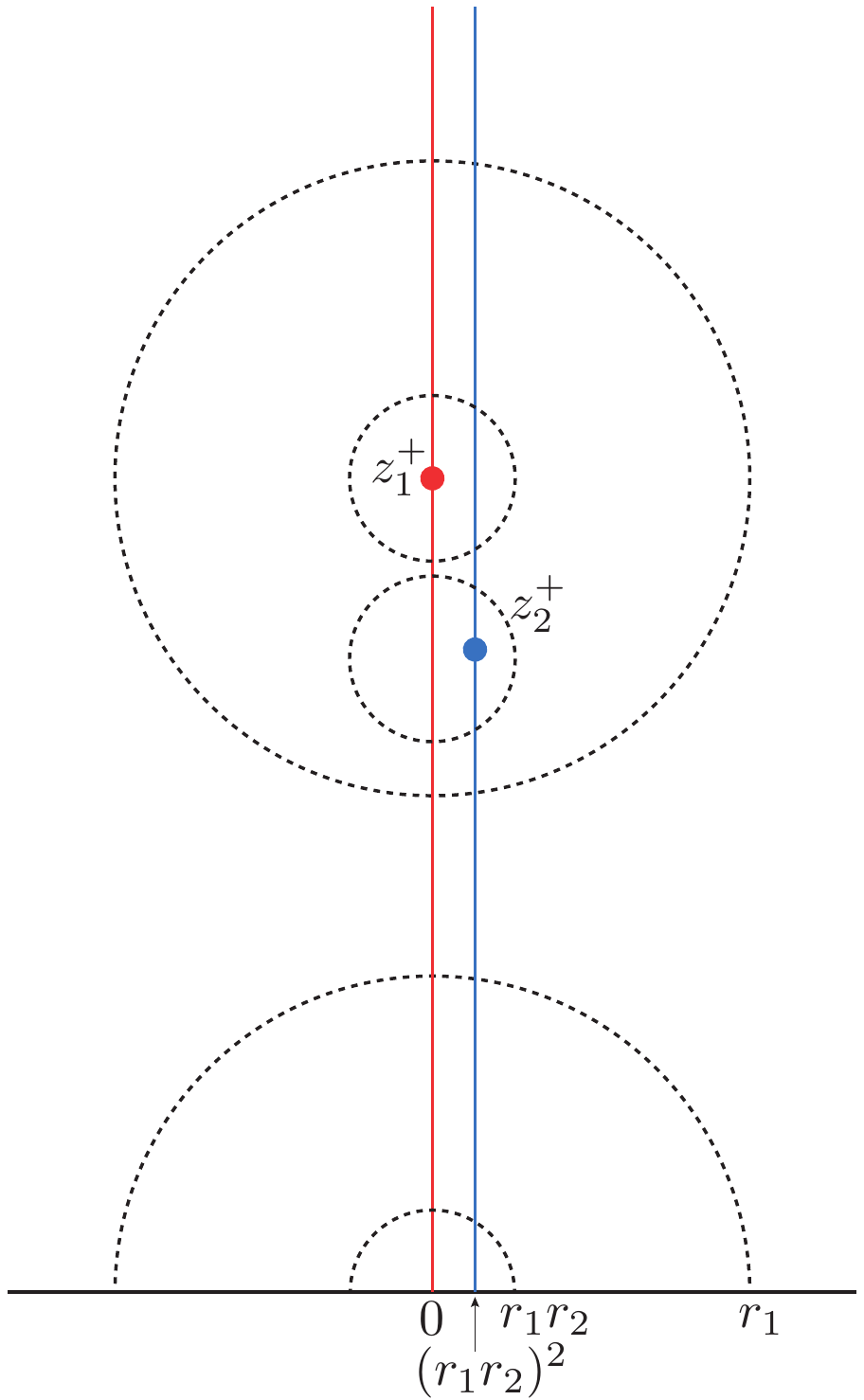}
\centering
\caption{ }
\label{fig:glue2}
\end{subfigure}
\centering
\caption{Popsicle with sphere bubble and gluing using planar model 2}
\label{fig:popdeg2}
\end{figure}

Let us explain how this stratum may arise as a limit along two different codimension one strata.

One of the limit is the usual one: consider stable popsicle with two disk components $\Sigma_0,\Sigma_1$
with $\Sigma_1$ as before, but $\Sigma_0$ has the horizontal popsicle line with $z_1^+, z_2^+$ on it.
If $z_1^+, z_2^+$ collide with each other, we obtain the sphere bubble with an induced popsicle line, and
the corresponding limit is \ref{fig:popdeg2}(A).

The other one is as a limit of Example \ref{ex:pop21} when $z_2^+$ approaches $p_0$.
It might appear that this will create another sphere bubble, but this will not happen as the resulting configuration would not be stable.
It is better seen from the planar gluing model.
In Figure \ref{fig:popdeg1} (B), suppose that $y$-coordinate of $z_2^+$ is $1 + cr$ for a small constant $c$, whereas
the popsicle line is at $x=r^2$. Note that the order of popsicle lines approaching to each other and that of $z_1^+$ and $z_2^+$
are $r^2$ and $r$ respectively, and the limit $r \to \infty$ is the desired stratum.
Note also that in the limit the alignment between $\Sigma_1$ and $\Sigma^+$  no longer exists.

Another way to see the second limit is as follows.
Identify $\Sigma^+$ with $S^2$, with south pole $p$ and north pole $z_1^+$.
Fix a popsicle line through $p$ and $z_1^+$. Use the remaining automorphism of $\Sigma^+$ to place $z_2^+$ at the equator,
which fixes the popsicle line through $z_2^+$. Then $z_2^+$ has a freedom to move along the half of the equator and
as $z_2^+$ approaches the first popsicle line, we obtain two limits corresponding to two vertices of the bottom edge of
Figure \ref{fig:popdeg}.
\end{example}

Let us explain the alignment data which we assign for this stratum.
As explained above,  $\Sigma_1$ and $\Sigma^+$
has different order of Gromov-convergence, and we record it using the following alignment data.
We first introduce the unstable popsicle disc $\Sigma_{un}$
(that is $D^2$ with $z_0, z_1$ and popsicle line(s) between them) and associate it to the nodal point $p_1$ as an auxiliary data.
Then, we align this unstable disc with $\Sigma^+$. 
As an alignment data, it will be denoted as
$$p_1 \Join \Sigma^+.$$
By aligning each sphere bubble with a disc or a nodal point, we will keep track of the order of convergence,
relate conformal structures and also use it to find the glued model.

It is helpful to describe the glued model   Figure \ref{fig:popdeg2}(B)  in the neighborhood of this codimension two stratum.
Since it is of codimension two, there are two gluing parameters, say $r_1,r_2 \ll 1$.
Here $r_1$ is the gluing parameter between $\Sigma^+$ and $\Sigma_0$.
Given another parameter $r_2$, a standard gluing of $\Sigma_0$ and $\Sigma_1$ with parameter $r_2$ is not the correct one.
Instead, we make use of the unstable disc $\Sigma_{un}$ at $p_1$ that is aligned with $\Sigma^+$, and
glue $\Sigma_0, \Sigma_{un}$ and $\Sigma_1$. 
This unstable disc (with two strip-like ends removed) corresponds to the half of annulus (centered at $0 \in H$ between two circles of radius $r_1r_2$ and $r_1$) 
after gluing in the Figure \ref{fig:popdeg2}(B). Alternatively one simply forget $\Sigma_{un}$ and glue $\Sigma_0$ and $\Sigma_1$ with gluing parameter $r_1r_2$ to get the same glued model.  This is the role of the alignment $p_1 \Join \Sigma^+$.

Another feature of  Figure \ref{fig:popdeg2}(B) is that the neighborhood of two marked points $z_1^+, z_2^+$ in $\Sigma^+$
are replaced by two discs of radius $r_1r_2$ which is the same as that of $\Sigma_1$.
In this way, popsicle lines extends consistently to the entire glued model. We may say that $z_1^+ \Join \Sigma_1$ and $z_2^+ \Join \Sigma_1$. But since it is clear how to align each sprinkle, we will not put these alignments explicitly in the alignment data.

In the above two examples, we have illustrated that how the alignment data arises from compactifications and how
we can use it to describe neighborhoods of  stable popsicles.

Now, let us consider another example, which illustrates that stable popsicle alone is not enough to 
describe the limit  in the compactification of popsicles.
In the previous two examples, the alignment data is uniquely determined from each stable popsicle. 
But a general stable popsicle might allow several compatible alignment data. For example
the order of convergence  may not be uniquely determined from the given stable popsicle.
Therefore, alignment data should be specified to describe the limit correctly.

\begin{example}
The following stable popsicle has three string of disks $\Sigma_0,\Sigma_1,\Sigma_2$ and a sphere $\Sigma^+$ attached to $\Sigma_0$ at $p_0$.
Let us denote by $p_i$ the intersection of $\Sigma_{i-1}$ and $\Sigma_{i}$ for $i=1,2$.
$\Sigma_2$ has marked points $z_1,z_2,z_3$  with three popsicle lines to $z_0$, $\Sigma_1$ has $z_3^+$ on 
the common popsicle line and on $\Sigma^+$, $z_1^+, z_2^+$ lie on the common popsicle line
(for all colors). Here we put a (distinct) color to each popsicle line for convenience.

\begin{figure}[h]
\includegraphics[scale=0.45]{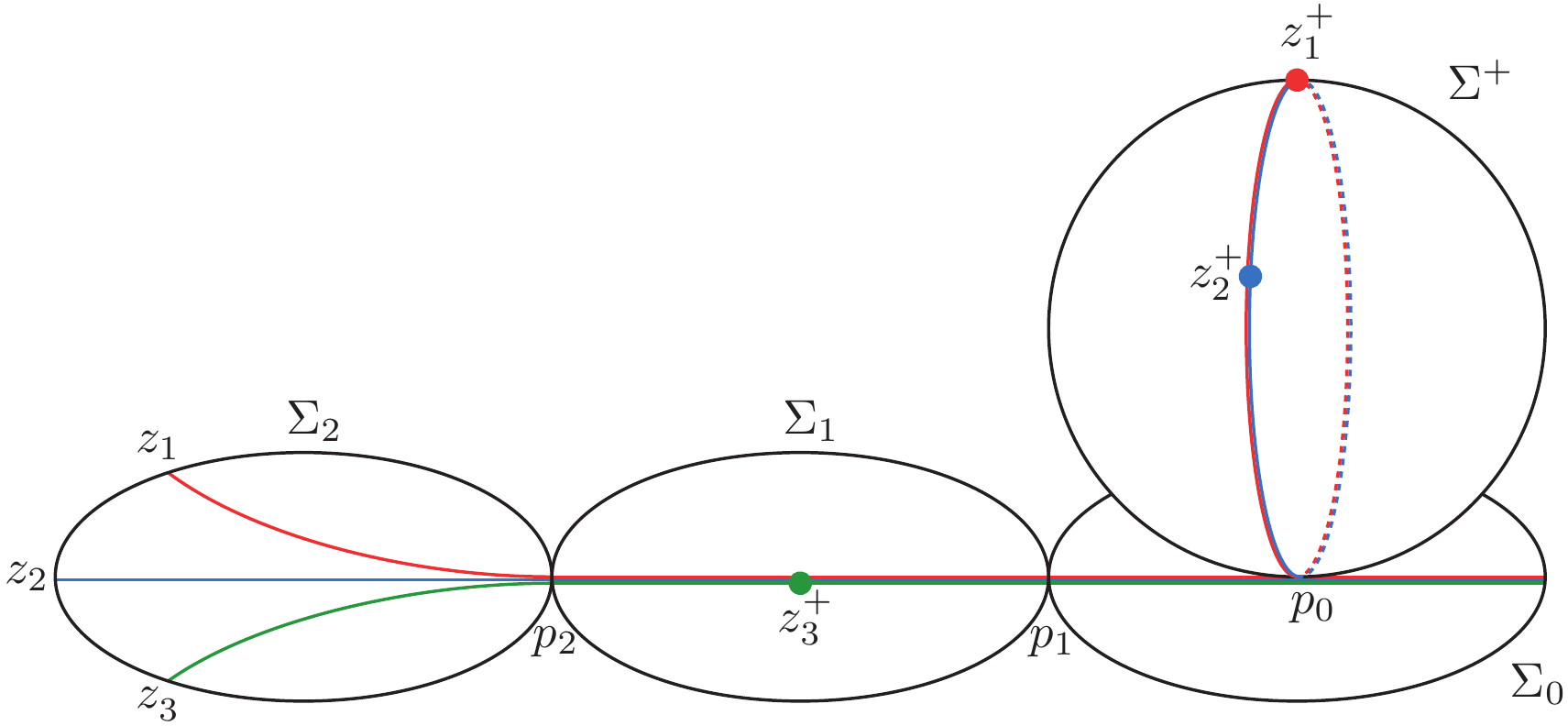}
\centering
\caption{ }
\label{fig:popdeg3}
\end{figure}

\begin{figure}[h] 
\begin{subfigure}[t]{0.3\textwidth}
\raisebox{1.4ex}{
\includegraphics[scale=0.45]{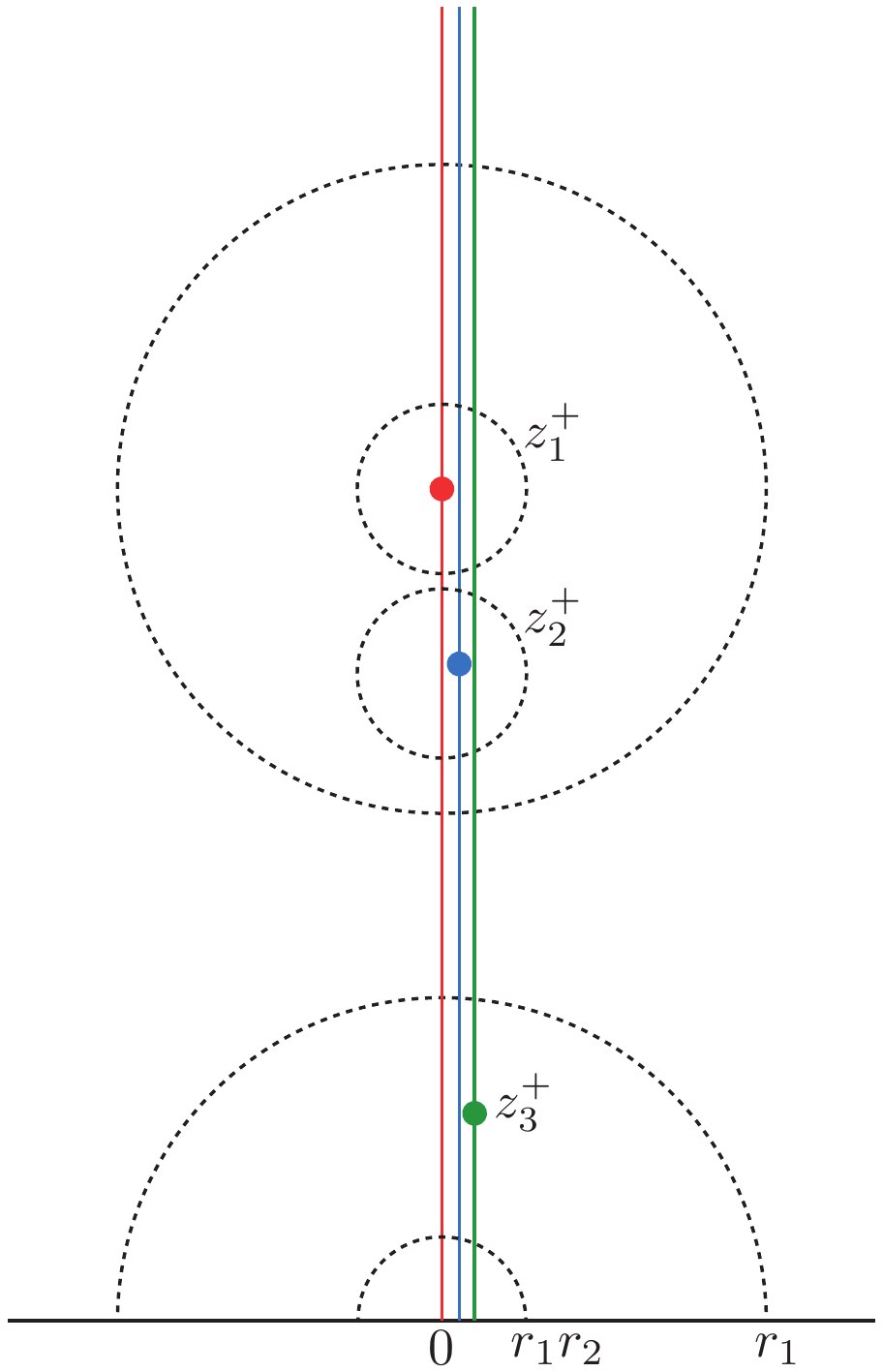}}
\centering
\caption{ }
\label{fig:glue3}
\end{subfigure}
\qquad
\begin{subfigure}[t]{0.3\textwidth}
\includegraphics[scale=0.45]{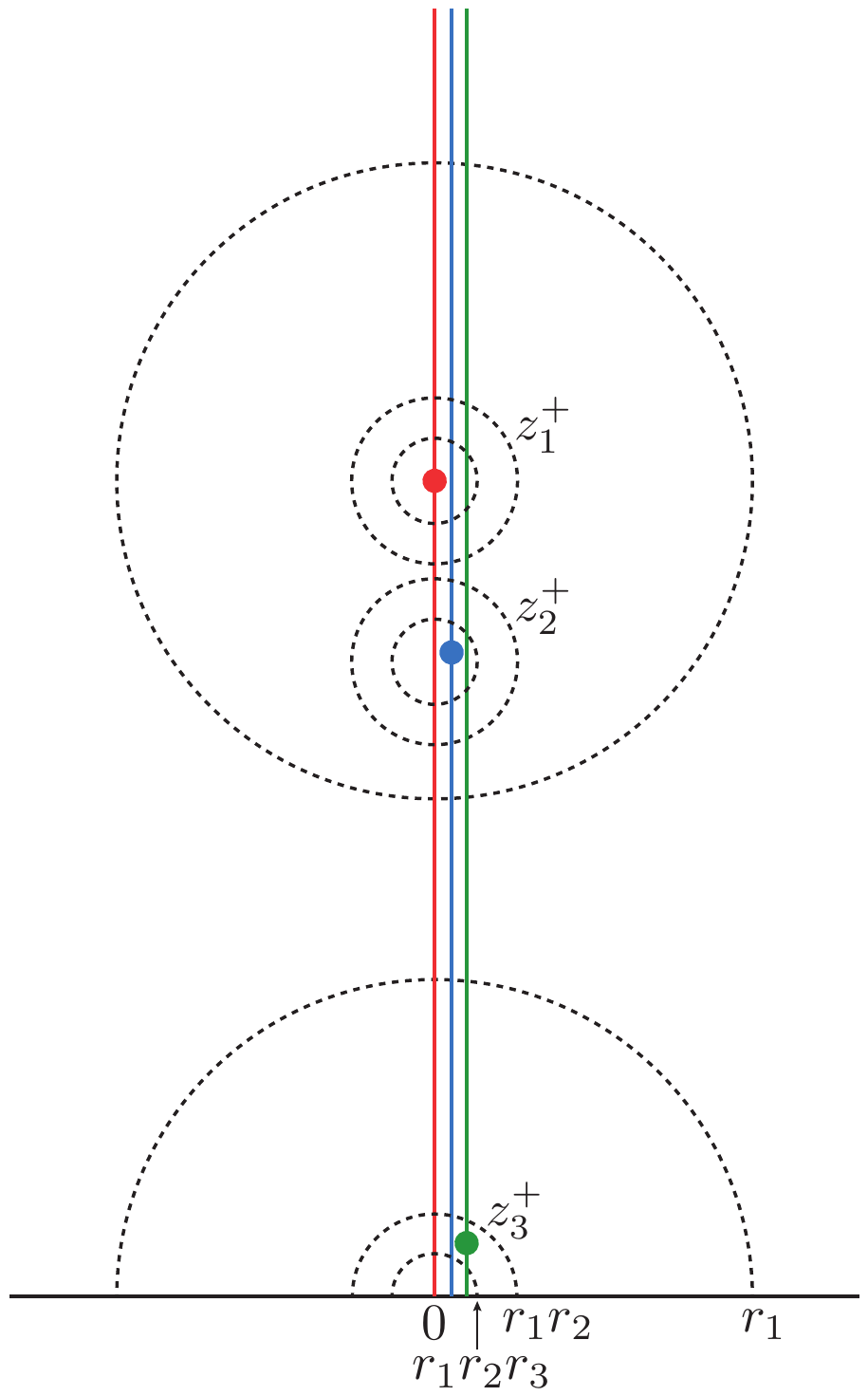}
\centering
\caption{ }
\label{fig:glue4}
\end{subfigure}
\qquad
\begin{subfigure}[t]{0.3\textwidth}
\includegraphics[scale=0.45]{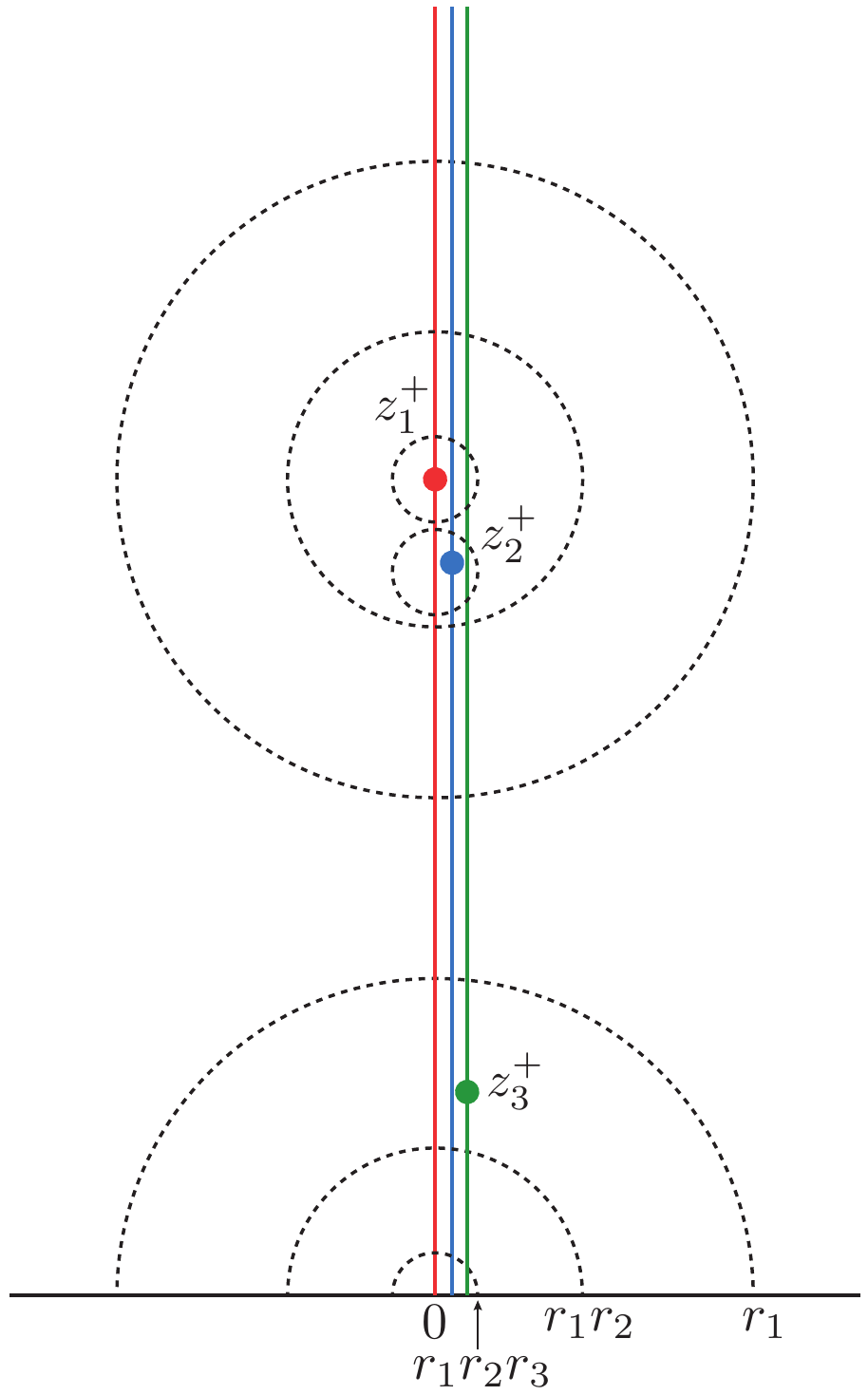}
\centering
\caption{ }
\label{fig:glue5}
\end{subfigure}
\centering
\caption{Glued models for a stable popsicle with 3 different alignment data}
\label{fig:glue_3}
\end{figure}

Let us discuss where to align $\Sigma^+$.
There are three choices, $ \Sigma_1,p_1, p_2$, and each choice corresponds to a different limit.
The planar models for each case are illustrated in Figure \ref{fig:glue_3}.
The codimension of each stratum are $2,3,3$ respectively.
Since popsicle lines become distinct in $\Sigma_2$, so $\Sigma^+$ cannot be aligned to $\Sigma_2$.

\begin{enumerate}
\item The case of $\Sigma_1 \Join \Sigma^+$:
This is when the order of convergence for  $\Sigma_1$ and $\Sigma^+$ are the same.
Note that conformal structure of $\Sigma_1$ and $\Sigma^+$ are unique up to biholomorphic map, but
as they are aligned, we need to use the same biholomorphic map for $\Sigma_1$ and $\Sigma^+$.
As a result, we get one dimensional freedom. 
 More precisely, let us fix the position of $p_1, p_2$ for $\Sigma_1$ and fix $p_0, z_1^+$ and the popsicle line for $\Sigma^+$.
 Each of $\Sigma_1$ and $\Sigma^+$ has the translation symmetry, which we use it to fix $z_3^+$ in $\Sigma_1$.
Therefore, the position of $z_2^+$ on the unique popsicle line in $\Sigma^+$ is free. As $\Sigma_2$ has one dimensional moduli,  the stratum is two dimensional, hence is of codimension two.

\item  The case of $p_1 \Join \Sigma^+$:
 Conformal structure of $\Sigma^+, \Sigma_0, \Sigma_1$ is unique up to automorphism and
hence the stratum is one dimensional (due to $\Sigma_2$) and of codimension 3. 

In the glued model  Figure \ref{fig:glue_3}(B), $\Sigma_0$ is identified with $H$, and $\Sigma^+$ is glued at $(0,1)$ with radius $r_1$.
From the alignment $p_1 \Join \Sigma^+$, we consider a unstable disc $\Sigma_{un}$, which is glued at $0 \in H$ with the
same gluing parameter $r_1$, which gives the half annulus between radius $r_1r_2$ and $r_1$ in $H$.
Then $\Sigma_1$ is glued to $\Sigma_{un}$ to give another half annulus between radius $r_1r_2r_3$ and $r_1r_2$ in $H$.
Lastly $\Sigma_2$ is glued to $\Sigma_1$ to give the half disc of radius $r_1r_2r_3$, and popsicle lines are extended from $\Sigma_2$
to all of $H$.

\item The case of  $p_2 \Join \Sigma^+$:

This is similar to the previous case but we glue  the unstable disc (aligned with $\Sigma^+$) at  $p_2$.
\end{enumerate}
\end{example}

Note that to properly describe the glued model of a stable popsicle, we need to keep track of unstable components called auxiliary data. One key feature is that such an auxiliary data can be uniquely determined from
the stable components and the alignment data. For example, an aligned sphere with several distinct popsicle lines through a nodal point 
but with only one sprinkle  is regarded as  an unstable sphere, although it  has the unique conformal structure. This
is in fact the auxiliary sphere that we glue in at each sprinkle in the above example.
One might choose to include such auxiliary data as a part of the definition of a stable popsicle, but for simplicity we will omit them from the definition.

Let us summarize informally the rules of alignment before we give a formal treatment in the next section.
Given a general stable popsicle with discs and sphere bubbles, each sphere bubble is
equipped with an induced popsicle lines on it. We need to specify an alignment for each sphere component.
A sphere component may be aligned to a disc component if the configuration of popsicle lines of the disc and sphere is
the same. If there are more than 2 distinct popsicle lines on the sphere, then it should be aligned to a disc (not to a vertex)
and there is a unique disc, whose double has the desired configuration of popsicle lines of the sphere.
So we align the sphere with such a disc component and this will align their conformal structures as well.

If a sphere component has a unique circle which is the common popsicle line, it is more delicate to
describe the rule, but roughly, it should be aligned to either a nodal
point between disc components where exactly the same set of popsicle lines pass through, or
to a disc component with the same popsicle lines. In this case, the nodal point or the disc component that is aligned should be further away from the root than the the disc component that the sphere sits on.

When several sphere components are aligned to the same disc nodal point, then
we need to also describe their respective orders of convergence, and this will be a part of the alignment data as well.

\section{New compactification of popsicles with interior marked points: formal definitions}\label{sec:3}
In this section, we define a new compactification of popsicles with interior marked points.
 We will first give a combinatorial description and then define a stable popsicles of a given combinatorial type.

Let us recall the notion of a rooted tree.
 \textbf{A rooted tree} with $n$ leaves  has $n+1$ semi-infinite edges with a single preferred one called the root, and the rest are called leaves.
 All rooted trees are oriented from the root toward the leaves. 
 \begin{defn}\label{defn:pot}
 This  orientation induces the following \textbf{partial order} on vertices and edges: For any $a,b \in \mathrm{Vertex}(T) \sqcup \mathrm{Edge}(T)$, we declare $a < b$ when $a$ is contained in the unique path from the root to $b$.
 
 For each edge $e$ of $T$, we denote by $T_e$ a subtree consisting of $e$ itself and all the vertices and edges that are greater than $e$.
 \end{defn}

	\begin{defn} 
	\textbf{A rooted ribbon tree} is a rooted tree $T$ with cyclic orders on adjacent edges for each vertex $v$ of T.
	Each vertex $v$ of a rooted ribbon tree has a single distinguished edge coming from the roots, denoted as $e_{v,0}$.
	The rests are arranged according to the cyclic structure as $\{e_{v,1}, \ldots, e_{v, \mathrm{val} (v)-1}\}$. 
	The leaves are arranged in a similar way, denoted as $\{l_1, \ldots l_n\}$.
	\end{defn}  

Using a rooted ribbon tree, we can describe limits of popsicles without sphere bubbles following \cite{AS}.
	\begin{defn}(Broken popsicles without sphere bubbles) 
	\label{broken popsicles without sphere bubbles}
	Let $T$ be a rooted ribbon tree with $n$ leaves and let $\phi: F \to \{1, \ldots, n\}$ be a non-decreasing flavor function. 
	A $\phi$-flavored stable broken popsicle (without sphere bubbles) modelled on $T$ consists of the follxowing data;
		\begin{enumerate}
		\item \textbf{decomposition of $F$}: a decomposition $F = \sqcup_{v} F_v$ where $v$ runs through $\mathrm{vertex}(T)$.
		\item \textbf{induced decomposition of $\phi$}: maps $\phi_v: F_v \to \{1, \ldots \mathrm{val}(v)-1\}$ for each vertices $v$ so that the edge $e_{\phi_v(f)}$ lies on the unique path from the root to $l_{\phi(f)}$ of $T$.
		\item \textbf{popsicles}: for each vertex $v$, a  popsicle of type $\left(\mathrm{val}(v), F_v, \phi_v\right)$
		\end{enumerate}
	We call a broken popsicle (without sphere bubble) is stable if popsicles assigned to each vertices are stable. 
	\end{defn}
	
Let us define a combinatorial notion called a layer.
Recall that $n$ interior marked points of a popsicle sphere in $\mathbb P_{m, n, \phi}$ are lying on
$m$ popsicles lines, which partitions  $n$ interior marked points into $m$ disjoint subsets. Also, the natural ordering of
popsicle lines (as vertical lines in $\C$) induces the ordering on these partitioned subsets.
	\begin{defn}\label{defn:lay}
	\label{rooted layered tree}
	\textbf{A rooted layered tree} is a rooted tree $T'$ together with 
ordered partitions of outgoing edges of from every vertices.
Namely, for each vertex $w$ with outgoing edges $\{ e_{w,1}, \cdots, e_{w, \mathrm{val}(w)-1}\}$,
we take its ordered partition
\[E_{w} := \{ E_{w, 1}, \ldots E_{w, m}\}\]
with the following ordering :  we declare $E_{w, i} < E_{w, j}$ for $i<j$. 
	We call $E_{w, i}$ an $i$-th layers at $w$. It is called stable if it is stable as a rooted tree.
	\end{defn}
We remark that a rooted tree modeling sphere bubbles does not carry a ribbon structure.

Recall from Definition \ref{defn:popdisk} that $F = \{1,\cdots, l\}$ labels interior marked points, and flavor function $\phi$ describes a popsicles line for each interior marking to sit on.
Thus, the induced marked points in the sphere bubble (which corresponds to leaves of a rooted tree) should carry this data as well.
	\begin{defn}
	An \textbf{$F$-label} of a rooted tree $T$ is an injective map $\psi: \mathrm{Leaf}(T) \to F$.
	\end{defn}
 


As we explained in the previous section, a sphere bubble of popsicles are not random and is aligned to a disc component or a nodal point.
To describe the alignment data, we introduce the notion of a coloring.

When $C$ is a set of some colors, a map $f: \mathrm{Leaf}(T) \to C$ associates a color to each leaf of $T$.
	\begin{defn}
	\label{definition of coloring}
	Let $T$ be a rooted tree, equipped with a map $f: \mathrm{Leaf}(T) \to C$. 
	For an edge $e$ of $T$, denote by $\mathrm{Leaf}_e(T) \subset \mathrm{Leaf}(T)$ the leaves of $T$ that are contained in the subtree $T_e$  (see Definition \ref{defn:pot}).

       The image of  an induced map $f: \mathrm{Leaf}_e(T) \to C$ describes a \textbf{color set} of $e$ with respect to $f$
\[C^f(e) :=\left\{f(l) \mid l \in \mathrm{Leaf}_e(T) \right\}\subset C.\]
	We call this assignment a \textbf{$C$-coloring} of a rooted tree with respect to $f$, denoted as $C^f$.	
%
	\end{defn}
The intuition behind the definition is simple. 
We "color" all the edges along the unique path from the root to each leaf $l$ by $f(l)$. 
Each edge $e\in \mathrm{Edge}(T)$ may carry multiple colors, and the set of those colors are denoted as $C^f(e)$.
There are two different colorings we would like to consider (corresponding to a disc and a sphere):

	\begin{itemize}
	\item Let $T$ be a rooted ribbon tree with $n$ leaves. A cyclic order on leaves determine the unique order preserving map $f: \mathrm{Leaf}(T) \to C=\{1, \ldots, n\}$ given by $f(l_i) =i$. We call it a \textbf{canonical coloring} of a rooted ribbon tree. We denote the color set of edge $e$ with respect to a canonical coloring simply as $C(e)$, omitting $f$.
	\item  Let $T'$ be a rooted tree with an $F$-label $\psi: \mathrm{Leaf}(T') \to F=\{1, \ldots, l\}$. Further suppose that we are given a nondecreasing flavor function $F \to C =\{1, \ldots, n\}$ as in \ref{broken popsicles without sphere bubbles}. The composition $\phi\circ\psi : \mathrm{Leaf}(T') \to \{1, \ldots, n\}$ determines a $C$-coloring of $T'$. We call this coloring an \textbf{induced coloring} on an $F$-labeled rooted tree, denoted as $C^{(\phi, \psi)}$
	
	In particular, if in addition $T'$ is layered (Definition \ref{defn:lay}), then  we can define the color set of a single layer $E_{w, k}$ from the coloring $C^{(\phi, \psi)}$:
 \[C^{(\phi, \psi)}\left(E_{w, k}\right) := \bigcup_{e \in E_{w, k}} C^{(\phi, \psi)} (e) \subset C.\]
	\end{itemize}
The notion of coloring will help us to describe the alignment data, which we will explain from now.

Given one popsicle sphere or trees of popsicle spheres, we should be able to tell whether they
may arise in the limit of popsicle discs. If they arise from such limits, there are alignments of conformal structures.
We will first describe when it is possible to align a popsicle sphere component to a popsicle disc component or a nodal point. 
Next, we will define an appropriate  compatibility between the choices of alignment data for every sphere components.

 Let us define the notion of alignability.
	\begin{defn}
	\label{aligning trees}
	Let $T$ be a rooted ribbon tree with $n$-leaves, and let $\phi: F=\{1, \ldots, l\} \to C=\{1, \ldots, n\}$  be a  (non-decreasing) flavor function.
	Let $T'$ be a rooted layered tree with an $F$-label $\psi: \mathrm{Leaf}(T') \to F$.
	Then $T$ carries a canonical coloring $C$ and $T'$ carries an induced coloring $C^{(\phi, \psi)}$ as  above.
	\begin{enumerate}
	\item We say $w \in \mathrm{Vertex}(T')$ and $v\in \mathrm{Vertex}(T)$ are alignable if for each $k\geq1$
	\[C^{(\phi, \psi)}\left(E_{w, k}\right) \subset C\left(e_{v, i_k}\right)\] 
	for some $i_k\geq1$, and $i_{k+1}>i_k$ for $\forall k\geq1.$ In fact, such $i_k$ is unique if it exists. 
	\item We say  $w \in \mathrm{Vertex}(T')$ and $e \in \mathrm{Edge}(T)$ are alignable if $w$ has only one layer $E_w$ and 
	\[C^{(\phi,\psi)}\left(E_w\right) \subset C(e).\]
	\end{enumerate}
	\end{defn}
	
The first case is when a sphere component is aligned to a disc component.
Each layer $E_{w,k}$ in the sphere bubble is from a vertical line given by  the (possibly overlapping) popsicle line(s) of color(s) $C^{(\phi, \psi)}\left(E_{w, k}\right)$.
 The first alignability is to guarantee that  the alignable disc should at least carry such popsicle lines as well.
More precisely,  the disc have a nodal (or marked) point
that contains such colors. This condition should hold for every layer of $w$. Later, this will allow us to relate the conformal structures of a disc and a sphere.

If a popsicle sphere is aligned to a vertex, we will match the conformal structure of the sphere with that of the unstable disc (with an input and an output
that are connected by a single common popsicle line), and hence the second alignability condition is needed.

\begin{example}
Let us give an example of the first alignability condition using the Figure \ref{fig:aligned model}.
The tree on the left describes a neighborhood of a vertex in a rooted ribbon tree with colors containing $\{1,2,\ldots,8\}$.
This vertex corresponds to a disc with 4 inputs and one output marked point. 
Each input is connected to an output via a geodesic in the disc, each of which are colored by $\{ \{1,2\}, \{3\}, \{4\}, \{5,6,7\}\}$. 
The layered tree on the right describes a neighborhood of a vertex in a rooted layered tree with colors containing $\{1,2,3,5,7\}$.
This vertex corresponds to a sphere with 5 inputs and one output, whose popsicle lines as drawn on the right.
We can check that the first alignability condition holds for each layer.

	\begin{figure}[h]
	\centering
	\def\svgwidth{17cm}
 	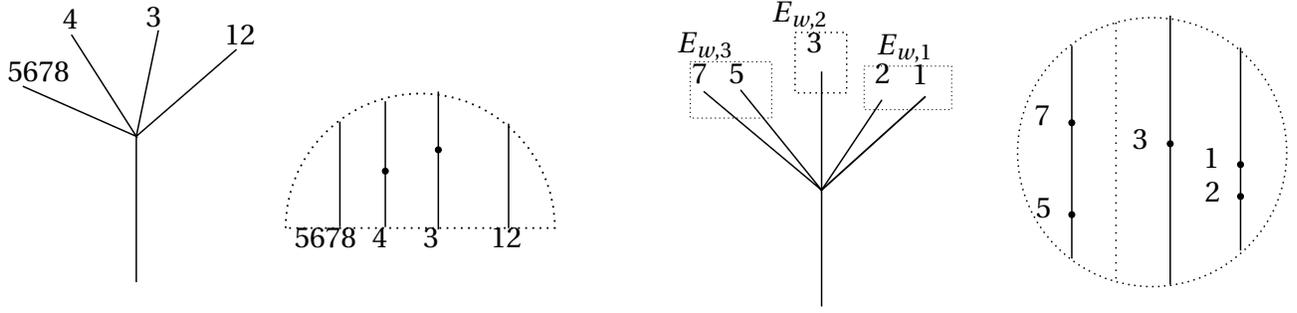
	\caption{Example of an alignable pairs and associated planar models}
	\label{fig:aligned model}
	\end{figure}
	
Let us make a few remarks. 
In this example, the colors $\{4,6,8\}$ does not appear on the right.
This is because we only keep track of a popsicle line whose sprinkle lies on the particular tree of sphere bubbles.
The sprinkles for $\{4,6,8\}$ should be in other trees of sphere bubbles or in some disc components.

Even though we do not record it, the popsicle lines of the color $\{4,6,8\}$ do exist, as one can see the dotted line on the right for the color 4.
Later, if we align the conformal structures of the disc and sphere component, the vertical lines on the left and on the right will be aligned
as in the Figure \ref{fig:glue_3}.
\end{example}

Now, let us explain the tree model for a general stable popsicle with sphere bubbles.
For each disc $D^2_v$ for the vertex $v$, there may be several tree of sphere bubbles that are attached to $D^2_v$.
We use $i,j$ to label these trees $\{T_{v,i,j}\}$, where $i$ labels the order of distinct popsicle lines on $D^2_v$, 
and $j$ labels the order of trees on each popsicle line.

\begin{defn}(Tree model for general popsicles)
Let us fix a non-decreasing flavor function $\phi: \widetilde F=\{1, \ldots, l\} \to \{1, \ldots, n\}$. A tree model for popsicles with $n$-boundary input is a tuple $\left(\widetilde T, \widetilde F, \widetilde \psi, \Phi, \Psi \right)$ consists of \;
   \begin{enumerate}
    \item \textbf{a decomposition of $\widetilde T$}: a decomposition $\widetilde T= T \bigsqcup_{v \in \mathrm{Vertex}(T)} \{T_{v,i,j}\}$ where $T $ is a rooted ribbon tree with $n$-leaves, and $\{T_{v,i,j}\}$ are rooted layered trees. 
    \item \textbf{a decomposition of $\widetilde F$}: a decomposition $\widetilde F = \bigsqcup_{v\in \mathrm{Vertex}(T)} \widetilde F_v$ and $\widetilde F_v = F_v \sqcup ( \sqcup_{i,j} F_{v,i,j})$ such that 
    \[\phi\left(F_v\right) \subset C_{e_{v,0}}, \hskip 0.2cm \phi\left(F_{v,i,j}\right) \subset C_{e_{v,i}},\]
    \item \textbf{an induced decomposition of $\phi$}: a collection of induced map 
    \[\phi_v : F_v \sqcup ( \sqcup_{i,j} F_{v,i,j}) \to \{1, \ldots, \mathrm{val}(v)-1\}\]
    such that $\phi(v), \hskip 0.1cm v\in F_v$ is defined as in Definition \ref{broken popsicles without sphere bubbles}, and $\phi\left(F_{v,i,j}\right):= i$
    \item \textbf{F-labels}:  a collection of bijective $F_{v,i,j}$-labels $\widetilde\psi = \left \{ \psi_{v, i,j}: \mathrm{Leaf}\left(T_{v,i,j}\right) \leftrightarrow F_{v, i,j} \right\}$,
    \item \textbf{alignment data I}: an order-preserving function 
    \[\Phi: \mathrm{Vertex}(\sqcup T_{v,i,j}) \to \mathrm{Vertex}(T)\sqcup\mathrm{Edge}(T)\] 
    satisfying the following compatibility condition; 
    \begin{itemize} 
     \item $w$ and $\Phi(w)$ are alignable,
     \item $\Phi \vert _{\mathrm{Vertex}(T_{v,i,j})}\subset \mathrm{Vertex}(T)_{>v}\cup \mathrm{Edge}(T)_{>v} $,
     \item $\Phi$ is strictly order-preserving on $\Phi^{-1}(\mathrm{Vertex}(T))$.
    \end{itemize} 
    \item \textbf{alignment data II}: a collection of strictly order-preserving surjection 
    $$\Psi = \left \{\Psi_e : \Phi^{-1}(e) \to \{1, \ldots, m_e\}, \hskip 0.2cm m_e \in \mathbb N\right \}$$
     for each $e$ whose $\Phi^{-1}(e)$ is nonempty.
   \end{enumerate}
   Here, order-preserving means with respect to the partial order on a rooted tree given by its orientation. We call a pair $(\Phi, \Psi)$ an alignment data for the model.
\end{defn}
The alignment data I describes the alignment for each sphere component. Note that the order preserving condition naturally arises from the compactification process.
The alignment data II concerns the case when there are several sphere components that are aligned to an edge $e$. In this case, we need to specify their relative speed of
formation of bubbles. The strictly order preserving condition is enforced within each tree of sphere bubbles.

Now, it is not difficult to associate an actual broken popsicle for the given tree model.
\begin{defn}
For a given tree model $\left(\widetilde T, \widetilde F, \widetilde \psi, \Phi, \Psi \right)$ for popsicles, we obtain a general broken popsicle of type $\left(\widetilde T, \widetilde F, \widetilde \psi, \Phi, \Psi \right)$ in the following way; 
\begin{enumerate}
 \item(popsicle discs on a ribbon tree) For each vertex $v$ of a rooted ribbon tree $T$, assign a popsicles $\Sigma_v$ of type $(\mathrm{val}(v), \phi_v)$. Our popsicle disc carries not only $F_v$-many interior marked points, but also $\#\left\{\{F_{v,i,j}\}\right\}$-many nodal points over which sphere bubbles are attached.
 \item (aligned with vertex) For a vertex $w$ of $T_{v,i,j}$ such that $\Phi(w)$ is a vertex, we assign a popsicle sphere of type $\left( \mathrm{val}(\Phi(w))-1, \mathrm{val}(w)-1, \phi_w \right)$, where $\phi_w$ is constructed as follows. A color of a layer satisfies $C^{(\phi, \psi)}\left(E_{w, k}\right) \subset C(e_{v,i_k})$ for the unique $i_k$. If $e_j \in E_{w,k}$, then we define $\phi_w(w^+_j) := i_k$. Finally, we require that a popsicle sphere $\Sigma_w^+$ and a popsicle disc $\Sigma_{\Phi(w)}$ are aligned to each other. 
 \item (aligned with edges) for a vertex $w$ of $T_{v,i,j}$ such that $\Phi(w)$ is an edge, we assign a popsicle sphere of type $(1, \mathrm{val}(w), \phi_w)$, where $\phi_w$ is just a constant map with value $1$.  It means that all incoming special points are on the single line.  
\end{enumerate}
\end{defn}
Here is a simple example of a tree model and its associated popsicle.
	\begin{figure}[h]
	\centering
	\def\svgwidth{10cm}
	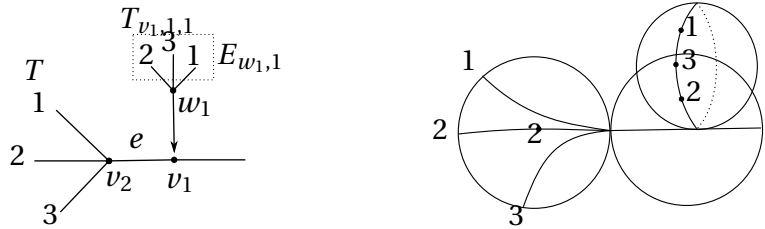
	\caption{Examples of combinatorial model for a general popsicle}
	\label{fig:sphere}
	\end{figure}
	
We call the model is stable if it is stable as stable maps (after forgetting the popsicle structure and alignments).
  We write a moduli space of general stable popsicles of type $\star = \left(\widetilde T, \widetilde F, \widetilde \psi, \Phi, \Psi \right)$ with $n$-boundary inputs by $P_{n, \widetilde F, \phi, \star}$. It is not hard to see that for a fixed $n$ and $\phi: F \to \{1, \ldots, n\}$, only finitely many combinatorial types of stable trees are possible. The compactified moduli space we consider is given by 
\[\overline{P}_{n,   F, \phi}:=\bigsqcup_\star P_{n, \widetilde F, \phi, \star}.\]

\section{Quantum cap action and $J$-holomorphic popsicles with $\Gamma$-insertions}\label{sec:qc0}

Let us recall quantum cap action and  introduce $J$-holomorphic popsicles with $\Gamma$-insertions.

%


\subsection{Quantum cap action}
A quantum cap action  of quantum cohomology can be defined on Fukaya category of a compact symplectic manifold \cite{auroux07}, and analogously
that of symplectic cohomology  can be defined on  wrapped Fukaya category of a Liouville manifold  \cite{Ga}.
Let us briefly recall the definition.
We refer readers to Appendix \ref{basic Floer theory} for a brief account of symplectic cohomology $SH^\bullet(M)$ and the wrapped Fukaya category $\mathcal{WF}(M)$.

Let $P_{n, \{i\}} \hskip 0.2cm (1\leq i \leq n)$ be a moduli space of discs with one interior marked point $z_1^+$ and boundary marked points $z_0,\ldots, z_n$ such that $z_0, z_1^+, z_i$ lie on a geodesic of $D^2$. 
By applying an automorphism of the disc, we may say that $z_1^+, z_0,z_i$ are precisely $0,1,-1$ of the disc.
Consider a moduli space 
\[\mathcal P_{n, \{i\}}(\Gamma; a_1, \ldots, a_{i-1}, b , a_{i+1}, \ldots, a_n, a_0)\]
of pseudo-holomorphic maps from $S\in P_{n, \{i\}}$ with an interior insertion $\Gamma \in SH^\bullet(M)$ at the puncture $z_1^{+}$ and with boundary insertions $(a_1, \ldots, a_{i-1}, b, a_{i+1}. \ldots, a_n)$. Here, we use a different symbol $b$ to emphasize that it is a bimodule input.

 \begin{defn}\label{defn:cap}
  A cochain level quantum cap action of $\Gamma$ is an $\AI$-bimodule map is defined by 
    \begin{align*}
      \cap \Gamma : T(\WF)\otimes \underline{\WF} \otimes T(\WF) & \to \WF \\
      (a_1, \ldots, a_{i-1}, \underline b, a_{i+1}, \ldots, a_n)  &\mapsto (-1)^{\star_{n, \{i\}}-\deg \Gamma} \bold{P}_{n, \{i\}}(\Gamma; a_1, \ldots, a_{i-1}, b, a_{i+1}, \ldots, a_n),\\
      \star_{n, \{i\}}&=\sum_{j<i} (j-1) \cdot \deg a_i +i\cdot \deg b + \sum_{i<j}j\deg a_j
    \end{align*}
    where  $\bold{P}_{n, \{i\}}(\Gamma;. a_1, \ldots, a_{i-1}, b, a_{i+1}, \ldots, a_n)$ is a sum of orientation operators associated to the moduli spaces for all possible $a_0$ (see Appendix \ref{basic Floer theory}).

  \end{defn}
   It is easy to check that $\cap \Gamma$ is a  $\AI$-bimodule map from diagonal $\AI$-bimodule  $\mathcal{WF}(M)_\Delta$ to itself:
\begin{equation*} 
 \cap \Gamma  : \mathcal{WF}(M)_\Delta \to \mathcal{WF}(M)_\Delta. 
\end{equation*}
 Also, this action is closely related to the closed-open map (see Appendix \ref{basic Floer theory}) in the following way. 
 \begin{prop}\label{QCtoCOprop}Up to homotopy, we have
    \[ \cap \Gamma(b)=m_2\left(b, \mathrm{CO}_{L_2}(\Gamma)\right)\]
  where $\textrm{CO}_{L_2}: SH^\bullet(M) \to CW^\bullet(L_2,L_2)$ is a word-length zero component of the closed-open map. 
  \end{prop}
  \begin{proof}
  Consider a $1$-parameter family of moduli space of holomorphic discs with 
  \begin{itemize}
    \item one outgoing boundary marking $z_0$ at $1$
    \item one incoming boundary marking $z_1$ at $-1$
    \item one moving interior marking $z_1^+$ at $-it, \hskip 0.2cm t\in[0,1]$
  \end{itemize}
  At $t=0$, we get $P_{1,\{1\}}$. At $t=1$, we get a moduli space of discs with disc bubble containing interior marked point. See Figure \ref{QCtoCO}. It corresponds to a disc moduli space governing $m_2\left(b, \mathrm{CO}_{L_2}(\Gamma) \right)$. 
  \end{proof}
  
   \begin{figure}[h]
    \includegraphics[scale=0.7, trim=0 630 0 10, clip]{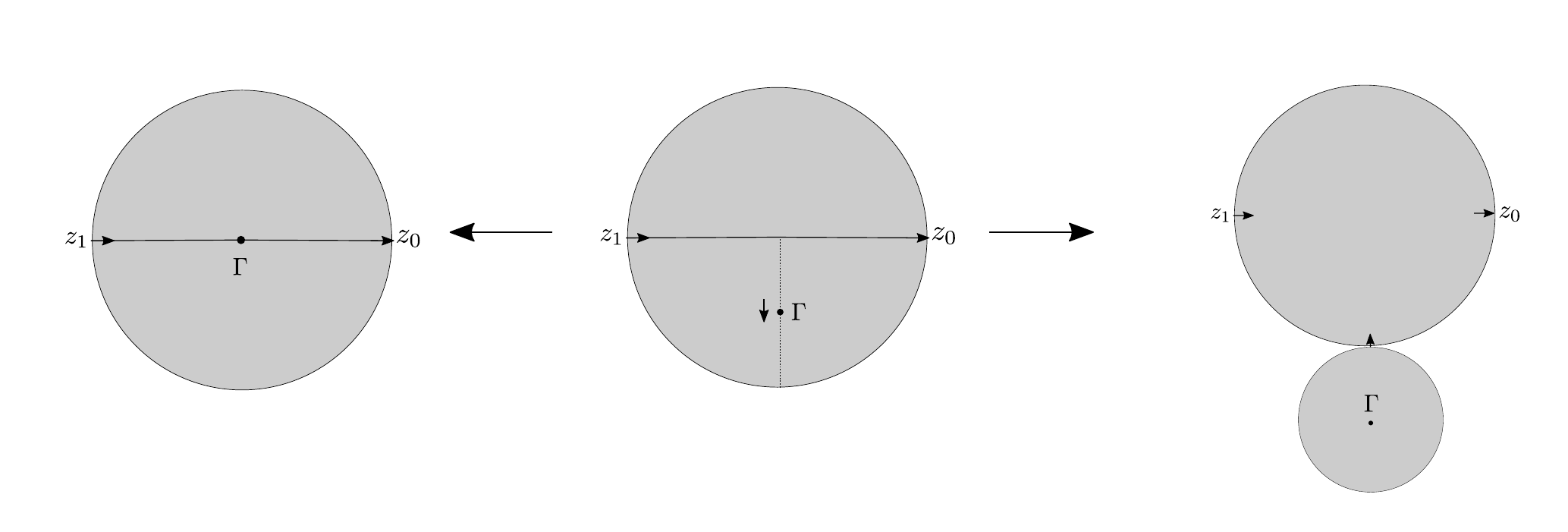}
    \centering
    \caption{Homotopy between $\cap \Gamma$ and $\mathrm{CO}(\Gamma)$}
    \label{QCtoCO}
  \end{figure}

We will construct a new $\AI$-category $\CG$ associated to the quantum cap action $\cap \Gamma$.
We first define its Hom space and the differential $M_1$ using the following $\AI$-bimodule structure. 
\begin{defn}\label{cor:tc}  
We first define $\CG$ as an $\AI$-bimodule over $\mathcal{WF}(M)$ from the following distinguished triangle:
    \[\begin{tikzcd}
    \mathcal{WF}(M) \arrow[r, "\cap \Gamma"] & \mathcal{WF}(M) \arrow[r]& \CG \arrow[r]& \phantom a 
      \end{tikzcd}\]
In particular, cohomologically we have
  \[\Hom_{\CG}^\bullet(L_1, L_2) \simeq \Cone\left(\begin{tikzcd}CW^\bullet(L_1, L_2) \arrow[r, "\cap \Gamma"] & CW^\bullet(L_1, L_2) \end{tikzcd}\right)\]
  \end{defn}
Instead of the $\AI$-bimodule structure on $\CG$, we will construct an $\AI$-category structure on $\CG$ in the next section.
The action $\cap \Gamma$ vanishes up to homotopy in $\CG$.

Intuitively, 
   \begin{itemize}
   \item objects of $\CG$ are twisted complexes
     \[\left(\begin{tikzcd}[column sep=3em] L \arrow[r, "\textrm{CO}(\Gamma)"] & L \end{tikzcd}\right), \hskip 0.2cm L\in \WF(M).\] 
   In many cases, they can be realized as a geometric surgery of $L$ with itself along $\textrm{CO}(\Gamma).$
   \item the space of morphisms is a "half" of the morphisms between twist complexes.  It consists of 
    \["a+\epsilon b"= \begin{pmatrix}a&0\\b&a\end{pmatrix} \in CW^\bullet \left( \begin{tikzcd}[column sep=3em] L_1 \arrow[r, "\textrm{CO}(\Gamma)"] & L_1 \end{tikzcd}, \begin{tikzcd}[column sep=3em] L_2 \arrow[r, "\textrm{CO}(\Gamma)"] & L_2 \end{tikzcd} \right)\]
   \end{itemize}
   This intuition works well in differential graded world (see Section \ref{counterpart}), but fails in all possible ways in the $\AI$-world. For example, $\AI$-compositions do not preserve the above space of morphisms. 
  We will consider a new type of $J$-holomorphic curves to construct a precise $\AI$-structure: Namely, the new $\AI$-structure that we will put on $\CG$ is not an algebraic one, but
  rather a geometric one, which are $J$-holomorphic maps from popsicles. It is natural in a sense that the cap action uses one hyperbolic geodesic line, and for the $\AI$-structure,
  we work with arbitrarily many geodesic lines on a disc.
  

\subsection{Popsicle maps with $\Gamma$-insertions}
Let us introduce pseudo-holomorphic maps from popsicles with $\boldsymbol{\Gamma}$-insertions. 
Consider Lagrangians $L_0,\cdots, L_n$ in the wrapped Fukaya category of $M$.
 \begin{defn}
  Let $a_i \in CF^\bullet (L_{i-1}, L_i)$ and  $\Gamma\in SH^\bullet(M)$. Define 
    \[ {\mathcal P}_{n, F, \phi}(\Gamma; a_1, \ldots, a_n, a_0)\] 
  be a  moduli space of pseudo-holomorphic maps 
  \[\left\{u:S\to M \; \Big| \; S\in P_{n, F, \phi}\right\}\]
  that satisfies 
  \begin{itemize}
    \item a boundary segment from $z_{i}$ to $z_{i+1}$ goes to $L_i$, 
    \item a boundary marking $z_i$ goes to $a_i$,
    \item all interior markings are asymptotic to $\Gamma$.
  \end{itemize}
Such a map $u$ is called a regular pseudo-holomorphic popsicle.
  \end{defn}
Using the compactification  $\overline{P}_{n,   F, \phi}$ of the domain popsicles in the last section, the compactification for regular pseudo-holomorphic popsicles can be carried out
in a standard way. The new feature of alignment data that was used in the compactification of the domain, concerns restriction of conformal structures on various disc and sphere components
in the limiting stable popsicles. And we do \textbf{not} need to align Floer data for the aligned domains as the alignment is a restriction only on the domain.

As we have shown that $\overline{P}_{n, F, \phi}$ is a manifold with corners, we can choose Floer data continuous and consistently by extending it from the lowest dimensional strata in an inductive manner  as usual.  We put no other restriction on the choice of domain-dependent Hamiltonians $H_S$, $F_S$ and almost complex structures $J_S$.
Transversality can be achieved in a standard way as we do not require compatibility of almost complex structures assigned at  aligned domains.
 
\begin{remark}
We also want this choice to be mildly equivariant following \cite{AS}, \cite{Se18}. Suppose several sprinkles are on the same popsicle line for a broken popsicle of type $\star$. Then a symmetry group acts on the corresponding strata $P_{n, \widetilde F, \star}$ by permuting their positions. We want our choice of Floer data to be invariant with respect to this action. It does not harm transversality argument because $P_{n, \widetilde F, \star}$ does not have any fixed point of the action and hence does not pose any problem on surjectivity of linearized Fredholm problem. 
\end{remark}

 Fix a collection of Hamiltonian chords $a_i \in \chi(L_{i-1}, L_i), \hskip 0.2cm i=1, \ldots n$.
 Let us describe  possible limits in the compactification
\[\overline{\mathcal P}_{n,   F, \phi} \left(\Gamma; a_1, \ldots, a_n, a_0 \right).\]
   As a first step, we allow our a tree model $\widetilde T$ of a broken popsicle may have semistable vertex. More precisely, a rooted ribbon tree $T\subset \widetilde T$ may have a bivalent vertices $v$ with $\widetilde F_v = \emptyset$, and a rooted layered tree $T_{v,i,j} \in \widetilde T$ may contains bivalent vertices.   We need to specify an alignment data for the new bivalent vertices on a tree of sphere bubbles so that it is also  compatible with the original alignment data.   This can be done as in the last section and we omit the details. In addition, we need to add cylinders (semi-stable spheres) splitting off at the interior marked points of a popsicle disc (say $v$). We need to add the corresponding new trees to the family $\{T_{v,i,j}\}$ but bubbling is local and do not carry alignment data.
          
Next, choose a Hamiltonian chord $a_{e}$ for each finite edge $e\in \mathrm{Edge}(T)$,  and choose a Hamiltonian orbit $\gamma_{e'}$ for each edge $e' \in \mathrm{Edge}(T_{v,i,j})$ . Let $v \in \mathrm{Vertex}(T)$ be a vertex and let $\Phi^{-1}(v)=\{w_i\}$. We assign the following moduli space of maps
	\begin{equation}\label{eq:pv}
	\mathcal P_{v}:=\left \{(u_v, \{u_{w_i}\})\vert u_v: \Sigma_v \to M, \hskip 0.2cm  u_{w_i}: \Sigma^+_{w_i} \to M, \hskip 0.2cm (\Sigma_v, \Sigma_{w_i}^+) \textrm{ are aligned}\right\}
	\end{equation}
such that 
	\begin{itemize}
	\item each $u_v$ and $u_{w_i}$ satisfy pseudo-holomorphic equation \eqref{eq:cr}.
	\item $u_v$ satisfies Lagrangian boundary condition induced on $\Sigma_v$, 
	\item $u_v$ and $u_{w_i}$ satisfy asymptotic condition on their strip-like/cylindrical ends. 
	\item If $\Sigma_v \simeq Z$ (resp. $\Sigma_{w_i} \simeq S^1\times \mathbb R$) is semistable, then $u_v$ (resp. $u_{w_i}$) is non-constant.
	\end{itemize}
	An aligned pair $(\Sigma_v, \{\Sigma_{w_i}\})$ carries a translation $\R$ symmetry if $\Sigma_v$ and $\sigma^+_{w_i}$ are all semi-stable. In that case, we define a similar moduli space as a quotient with respect to that symmetry group
	\begin{equation}
	\mathcal P_{v}:=\left \{(u_v, \{u_{w_i}\})\vert u_v: \Sigma_v \to M, \hskip 0.2cm  u_{w_i}: \Sigma^+_{w_i} \to M, \hskip 0.2cm (\Sigma_v, \Sigma_{w_i}^+) \textrm{ are aligned}\right\}/\R.
	\end{equation}

Similarly, when $\Psi^{-1}_e(k) =\{w_j\}$ for $e\in \mathrm{Edge}(T)$, we can define (following the case of \eqref{eq:pv}) 
	\begin{equation}
	\mathcal P_{e, k}:=\left\{\{u_{w_j}\}\vert u_{w_j}: \Sigma^+_{w_j} \to M, \hskip 0.2cm (\Sigma_{w_j}^+)_j \textrm{ are aligned}\right\},
	\end{equation}
 
 Define 
\[\mathcal P_{n, \widetilde F, \phi, \star} \left(\Gamma; a_1, \ldots, a_n, a_0 \right):= \bigcup _{a_e, \gamma_{e'}}\left(\bigcup_{v \in \mathrm{Vertex}(T)}  \mathcal P_{v}\right) \cup \left(\bigcup_{e \in \mathrm{Edge}(T)}\bigcup_{1\leq k \leq m_e} \mathcal P_{e,k} \right)\]
and finally we set
\[\overline{\mathcal P}_{n,   F, \phi} \left(\Gamma; a_1, \ldots, a_n, a_0 \right):= \bigcup_{\star} {\mathcal P}_{n, \widetilde F, \phi, \star} \left(\Gamma; a_1, \ldots, a_n, a_0 \right). \]

  \begin{lemma} For a generic choice of universal and consistent Floer data, we have the following.
  \begin{enumerate}
    \item Moduli spaces $\OL {\mathcal P}_{n, F, \phi}(\Gamma; a_1, \ldots, a_n, a_0)$ are smooth manifolds with corners and compact. 
    \item For a given input $\Gamma$ and $a_i, \hskip 0.1cm (i=1, \ldots, n)$, there are only finitely many $a_0$ for which 
    $$\OL {\mathcal P}_{n, F, \phi}(\Gamma; a_1, \ldots, a_n, a_0)$$
    is non-empty .
    \item It is a manifold of dimension 
    \[|F|(1-\deg\Gamma)+n-2 +\deg a_0 -\sum_{i=1}^n\deg a_i\]
  \end{enumerate}
  \end{lemma}
  \begin{proof}
Recall that the domain broken popsicles are shown to be smooth manifolds with corners in Theorem \ref{thm:corner}.
The desired transversality and compactness for pseudo-holomorphic maps can be proved as well using the standard methods and the above theorem.
 A standard index formula tells us that 
  \begin{align*}\textrm{dim}_\R \overline {\mathcal P}_{n, F, \phi}(\Gamma; a_1, \ldots, a_n,a_0) & =  \textrm{dim}_\R  \overline{\mathcal M}_{|F|;n,1}(\Gamma, \ldots, \Gamma;a_1, \ldots, a_n, a_0) - |F|\\
  &=(2|F|+n-2)+\deg a_0-\sum_{i=1}^n\deg a_i-|F|\cdot \deg \Gamma -|F|\\
  &= |F|(1-\deg \Gamma) +n-2 +\deg a_0 -\sum_{i=1}^n\deg a_i.
  \end{align*}
  \end{proof} 
  Let  $\bold{P}^\Gamma_{n, F, \phi}$
  be a sum of orientation operators associated to the zero-dimensional component of 
  $$\OL{\mathcal P}_{n, F, \phi}(\Gamma; a_1, \ldots, a_n, a_0)$$ for all possible $a_0$. A degree of this operator is $2-n-|F|(1-\deg\Gamma)$. Also,  $\bold{P}^\Gamma_{n, F, \phi}$ describes $\AI$-structure $\{m_n\}$ if $F$ is empty, and describes quantum cap action $\cap \Gamma$ if $|F|=1$.

\section{The new $\AI$-category $\CG$}\label{sec:qc1}
In this section, we will define a new $\AI$-category $\CG$ of a Liouville manifold $M$ using popsicles with $\Gamma$-insertions, where $\Gamma$ is a certain symplectic cohomology class of $M$.
A key feature of the new $\AI$-category is that the image of quantum cap actions by $\Gamma$  vanishes in cohomology of $\CG$.

To control the new sphere bubbling phenomenon  of the popsicle moduli space, we restrict ourselves to the following setting.
\begin{assumption}\label{as:main}
We will assume that a Liouville manifold $M$ satisfies the following in this section.
		\begin{enumerate} 
		\item $c_1(M) = 0$.
		\item Reeb flow $\mathcal R_t$ on $\partial M$ defines a free $S^1$-action.
		\item linearized Reeb flow is complex linear.
		    \item  Robbin-Salamon index of a principal Reeb orbit component is non-negative.
		   \item Symplectic cochain $\Gamma_M$ that represents the fundamental class of principal orbits is closed.
		\end{enumerate}
\end{assumption}
More details on  these assumptions will be given in Section \ref{sec:bbass}.  
 The Milnor fiber quotients that we will investigate later in this paper do not satisfy the first two hypotheses, and
 we will explain later how to generalize the construction in those cases.


Under the above assumptions, we prove the following theorem in this section.
\begin{thm}\label{thm:fulou}
Given a Liouville manifold $M$ and the element $\Gamma_M \in SH^\bullet(M)$ satisfying the above assumption, we define a new $\Z$-graded $\AI$-category $\CG$ such that 
\begin{enumerate}
\item $\CG$ has the same set of objects as the wrapped Fukaya category $\WF(M)$,
\item for two objects $L_1,L_2$, its morphisms are given by
 \[\Hom_{\mathcal{C}_\Gamma}(L_1, L_2) = CW(L_1, L_2) \oplus CW(L_1, L_2)\epsilon \]
 where $CW(L_1,L_2)$ is the morphism space for $\WF(M)$ and $\deg \epsilon =-1+ \deg \Gamma_M,$
\item a natural inclusion $\Psi : \WF(M) \to \CG$ is an $\AI$-functor,
\item regarding the $\AI$-category $\CG$ as an $\AI$-bimodule over $\WF(M)$ (using $\Psi$), we have a distinguished triangle of 
$\AI$-bimodules
\begin{equation*} 
\begin{tikzcd}   \mathcal{WF}(M)_\Delta \arrow[r, "\cap \Gamma "] & \mathcal{WF}(M)_\Delta \arrow[r] & \CG \arrow[r] & \;\; \end{tikzcd}
\end{equation*}
\end{enumerate}
\end{thm}

Roughly speaking, $\AI$-operations are defined by  {\em popsicles with $\Gamma$-insertions at
sprinkles}.
One of the difficulty is to rule out any codimension one (aligned) sphere bubbles with $\Gamma$-insertions, and we achieve it using  Conley-Zehnder indices and action values
under the Assumption \ref{as:main}. The rest of codimension one strata  are analogous to
those of Abouzaid-Seidel \cite{AS}, and Seidel \cite{Se18}  and we obtain the desired $\AI$-category.

\begin{remark}
Popsicle structures were introduced by Abouzaid-Seidel \cite{AS} and Seidel \cite{Se18}.
It is a good moment to point out differences and similarities. 

 In \cite{AS},\cite{Se18},  sprinkles were allowed to coincide resulting a different compactification of popsicles. 
In \cite{AS}, popsicles were used to mark the places to put sub-closed one forms for continuation maps between linear Hamiltonians. 
They localize the continuation map and obtain a big chain complex involving countable family of Hamiltonians $\{nH\}_{n \in \mathbb{N}}$. 
This is a definition of wrapped Fukaya category in a linear Hamiltonian setting. 
 Seidel \cite{Se18} considered the  continuation map for Lefschetz fibration, but he also considered its cone
\[CF^\bullet(L_0,L_1;-H) \to CF^\bullet(L_0,L_1;0) \to CF^\bullet(L_0,L_1;\mathrm{conti}),\] 
not only its localization. Popsicle maps were used to construct an $\AI$-algebra structure on $CF^\bullet(L_0,L_1;\mathrm{conti})$. 
This brings out the effect of removing contributions from compact part of Lefschetz fibration, and recovers the Floer cohomology of the fiber. 

Our geometric setting is different from the references because we regard sprinkles as genuine inputs for the symplectic cohomology class $\Gamma$. 
But the algebraic properties we desire are similar to those of \cite{Se18}. 
In our case, we want to remove the effect of the quantum cap action $\cap \Gamma$ from the wrapped Fukaya category.
Namely, wrapped generator in the image of quantum cap action will be killed.

  \end{remark}

\subsection{Codimension one analysis}
 Let us discuss the codimension one strata of the compactification $$\overline{\mathcal P}_{n,   F, \phi} \left(\Gamma; a_1, \ldots, a_n, a_0 \right).$$
The obvious ones are given by splitting off a non-trivial  cylinder (at sprinkles) or a  strip (at boundary markings). These will correspond to the differential
of $\Gamma$ or $a_i$'s. Beyond these ones, the remaining codimension one strata arise when the domain regular popsicle degenerates.
Recall from \eqref{codimension formula}, that we can explicitly determine the codimension of the broken popsicle 
of type $\star=(\widetilde T, \widetilde F, \widetilde \psi, \Phi, \Psi)$ to be 
\[\vert \mathrm{Vertex}(T) \vert -1 + \sum_{e\in \mathrm{Edge}(T)}m_e .\]

Since $m_e \geq 0$, codimension can may arise in the following cases: 
	\begin{enumerate}
	\item $\vert \mathrm{Vertex}(T) \vert =1$ and $m_e=1$ for a single $e$ and $m_{e'}=0$ for other $e'$.
	\item $\vert \mathrm{Vertex}(T) \vert =2$ and $m_e=0$. 
	\end{enumerate}
The first case do not arise, since there is only one disc component, there is no possible edge $e$.
The second case is quite interesting. 
Label two disc components as $\mathrm{Vertex}(T)=\{v_1, v_2\}$ with an edge $e$ from $v_1$ to $v_2$.

If there are no sphere components, then any such stratum is given by a broken popsicle consisting of two popsicles without sphere bubbles, and these will describe the $A_\infty$-relations
(as in Abouzaid-Seidel \cite{AS} and Seidel \cite{Se18}).

Let us consider the case that there is at least one sphere component.
As the alignment data $m_e$ counts sphere bubbles that are aligned to the edge $e$, the condition $m_e=0$ implies that all sphere components $\Sigma_{w}^+$ (if there is) should be aligned to the disc component $v_2$. Note that there is no restriction on the number of sphere components. From the restriction on the alignment data, all sphere components should be
attached to the disc $v_1$ (if they are stacked on each other, then it violates the strictly order preserving condition).
We may have several sphere bubbles in codimension one! (see Figure \ref{fig:twobubble} for an example).

	Therefore, sphere components in the compactification  provide  genuine obstructions of defining $A_\infty$-structure from $\overline{\mathcal P}_{n,   F, \phi} \left(\Gamma; a_1, \ldots, a_n, a_0 \right)$.   Fortunately, we can show these obstructions vanish in certain favorable situations. The following degree estimate that we establish will be crucial for this purpose.

Let us first set up the notations for such codimension one strata with sphere bubbles.
Consider the following pseudo-holomorphic broken popsicle with appropriate boundary conditions:
\[u:=(u_{v_1}, u_{v_2}, \{u_{w_i}\}): \Sigma = \Sigma_{v_1} \cup \Sigma_{v_2} \cup \Sigma^+_{w_1} \cup \cdots \cup \Sigma^+_{w_m} \to M,\]
We assume that all spheres $\Sigma^+_{w_i}$ are aligned to the stable disc $\Sigma_{v_2}$.
We assume that $u$  is  Fredholm-regular and rigid. We make the following notations.
	\begin{itemize}
	\item  $\Sigma_{v_2}$ has $n$ boundary insertions, 
	\item $k_{v_i} = F_{v_i}$ denotes a number of sprinkles on $\Sigma_{v_i}$ for $i=1,2$, 
	\item $k_{w_i}=\mathrm{val}(w_i)-1$ denotes a number of incoming insertion of $\Sigma^+_{w_i}$ (where $\Gamma$'s are placed) and $k_{w_i} \geq 2$,
	\item $\gamma_{w_i} \in \mathcal O$ denotes an output of $u_{w_i}$,
	\item $\mathcal M$ denotes a moduli space of underlying aligned pairs of popsicles $\Sigma = \left(\Sigma_{v_1},\Sigma_{v_2}, \{\Sigma_{w_j}^+\}\right)$. 
	\end{itemize}
Observe that we can decompose its tangent space as 
\begin{equation}\label{eq:dfspb}
T_\Sigma \mathcal M \simeq  \big( \R^{k_{v_1}} \times \R^{m-1}\big) \times  \big( T\mathcal M_{0, n+1} \times \R^{k_{v_2}} \big) \times \R^{k_{w_1}-1} \times \cdots \times \R^{k_{w_m}-1}.
\end{equation}
Here, the first two components handle the deformation of marked points on the discs $\Sigma_{v_1}$, $\Sigma_{v_2}$ respectively  and the component  $\R^{k_{w_i}-1}$ for that of
the aligned sphere $\Sigma_{w_i}$ for $i=1,\cdots,m$.

	\begin{prop}
	\label{degree inequality 1}
	In this situation, the following inequality holds;
	\begin{equation}
	k_{w_i} \deg \Gamma - k_{w_i}-n+3 \leq \deg \gamma_{w_i} \leq k_{w_i} \deg \Gamma - k_{w_i} +1
	\end{equation}
	\end{prop}
	
	\begin{proof}
	Consider a linearized Fredholm problem asssociated to $u$.
	It can be written (using the decomposition  \eqref{eq:dfspb}) 
	\begin{equation}
	D_u: T_\Sigma \mathcal M \times \left( V_{v_1} \times  V_{v_2} \times V_{w_1} \times \cdots \times V_{w_k}\right) \to \left(W_{v_1} \times  W_{v_2} \times W_{w_1} \times \cdots \times W_{w_k} \right)
	\end{equation}
	Let us write  $D$ in the matrix form.
	\begin{equation} D_u=
	\begin{pmatrix}
		 D_{u_{v_1}}&0 	&   0	& 0	&0   & \cdots & 0 \\
	0&D_{0v_2}	&  D_{v_2v_2}	& 0	& 0	& \cdots & 0 \\
	0&D_{0,w_1}	& 0 	& D_{w_1w_1}	& 0	& \cdots 	& 0 \\
	\vdots & \vdots 	  &\vdots	& \vdots 	&\\
	0& D_{0w_k}	& 0  	& 0 	& 0	&  \cdots & D_{w_kw_k}
	\end{pmatrix}.
	\end{equation} 
	Let us explain each components.  For $v_1$ and $v_2$, we have 
	$$D_{u_{v_1}}:  \big( \R^{k_{v_1}} \times \R^{m-1}\big)  \times V_{v_1} \to W_{v_1}$$
		\[D_{u_{v_2}} = (D_{0v_2}, D_{v_2v_2}):   T\mathcal M_{0, n+1}  \times \left(\R^{k_{v_2}}\times V_{v_2} \right) \to W_{v_2}\] 
	is a linearized Fredholm operator associated to $u_{v_2}$ and 
	\[D_{u_{w_i}} = (D_{0w_i}, D_{w_iw_i}):  T\mathcal M_{0, n+1}  \times \left( \R^{k_{w_i}-1}\times V_{w_i}\right) \to W_{w_i}\] 
	are linearized Fredholm operators associated to $u_{w_i}$ respectively. 
	In this way, we can separate the contribution of  $T\mathcal M_{0, n+1}$, which affects the disc $v_2$ as well as the spheres $\{w_i\}$ due to the alignment.	
	
	Since $u$ is Fredholm regular and rigid, $D_u$ is surjective and $\mathrm{Ker}(D_u)=\{0\}$.	
	And each $D_{u_v}$ and $D_{u_{w_i}}$ is surjective because $D_u$ is surjective. It means that $u_v$ and $u_{w_i}$ are Fredholm regular in their own rights. 
	
	Next, observe that $\mathrm{Ker}D_{u_{w_i}} \cap \{0\}\times \left(\R^{k_{w_i}-1} \times V_v\right)=\{0\}$. If not, we can extend a non-zero vector $(0, \alpha, \beta) \in \mathrm{Ker}D_{u_{w_i}} \cap \{0\} \times \left( \R^{k_{w_i}-1 }\times V_v \right)$ to 
	\[(0, \ldots, 0, \alpha, \beta, 0, \ldots, 0) \in \mathrm{Ker}D_u,\]  
	which contradicts to $\mathrm{Ker}D_u=\{0\}$. Therefore, $\mathrm{Ker}D_{u_{w_i}}$ must projects isomorphically into $T\mathcal M_{0, n+1}$ as a vector space, which again implies $\mathrm{dim}\mathrm{Ker}D_{u_{w_i}} \leq \mathrm{dim} T\mathcal M_{0,n+1}$. Combine these two observations with Fredholm index theorem, we get 
	\[0\leq \mathrm{dim}\mathrm{Ker}D_{u_{w_i}} = \mathrm{Index}D_{u_{w_i}}=\mathrm{dim}T\mathcal M_{0,n+1} + (k_{w_i}-1) + \deg \gamma_{w_i} - k_{w_i} \deg \Gamma \leq \mathrm{dim}T\mathcal M_{0, n+1}.\]
	 The proposition follows from these inequalities.
	\end{proof}
	
	\begin{cor}
	\label{degree inequality 2}
	In this situation, there is at least one component $u_{w_i}$ such that 
	\[\deg \gamma_{w_i} \leq k_{w_i} \deg \Gamma - \frac{k_{w_i}}{2}.\]
	\end{cor}
	\begin{proof}
	Since $k_{w_i}\geq2$ for all $i$, we know that the number of sphere component (denoted by $m$) is less then the half of the number of interior marked points. Therefore,
	\begin{equation}
	\sum \deg \gamma_{w_i} \leq (\sum k_{w_i}) \deg \Gamma - (\sum k_{w_i}) + m \leq (\sum k_{w_i}) \deg \Gamma - (\sum k_{w_i})/2.
	\end{equation}
	The corollary follows by the pigeon hole principle.
	\end{proof}

\subsection{Vanishing of sphere bubbles}\label{sec:bbass} 
We prove the following key proposition, which allows us to define the desired $\AI$-structure.
	\begin{prop}\label{prop:v1}
	Under the assumption \ref{as:main},   sphere bubbles of codimension one do \textbf{not} appear in the compactification $\overline{\mathcal P}_{n,   F, \phi} \left(\Gamma; a_1, \ldots, a_n, a_0 \right).$ 
		\end{prop}
Let us first explain the assumptions in more detail.
The first assumption $c_1(M) =0$ is to use degree arguments for the vanishing.
The second assumption on the existence of free $S^1$ action on $\partial M$ is given for two purposes.
The first is that  we will have such a $S^1$-action (although it is not free) in our main application, and the second reason is that $S^1$-action make the symplectic cohomology setting to be Morse-Bott and we can use associated spectral sequences.
We recall the third assumption from \cite{KvK16} which makes the local system associated to the Morse-Bott components to be trivial:
\begin{defn}
A linearized Reeb flow is complex linear if there exist a compatible  complex structure $J$ for the contact structure $\xi$ (with $d\lambda$) 	such that for every periodic Reeb orbit, its linearized Reeb flow is complex linear with respect to some unitary trivialization of $(\xi, J, d\lambda)$.
\end{defn}
Under the assumption of the $S^1$-action, Reeb orbits are degenerate and it is more convenient to work with the Robbin-Salamon index.
We will  make a perturbation in the Morse-Bott setting 
to define the Hamiltonian orbit $\Gamma_M$ that corresponds to the (fundamental class of) the set of principal orbits. 
We suppose that $\Gamma_M$ is a closed element in the symplectic cochain complex. Perturbation only guarantees that it is closed in the neighborhood of the Morse-Bott component, but
 pseudo-holomorphic cylinders might escape such a neighborhood, and closedness has to be checked using the geometry of $M$.

Now, let us explain the proof of the proposition  in two steps.
\subsubsection{Step I: spectral sequence associated to an energy filtration}
	\begin{lemma}\label{energy spectral sequence}
	There is a spectral sequence converging to $SH^*(M)$ 
	\[E^{pq}_1 = 
	\left \{ 
	\begin{array}{ll}
	H^q(M) & (p=0) \\
	H^{p+q- p\cdot\mu_{RS}(\partial M_1)}(\partial M) & (p <0) \\
	0 & (p>0)
	\end{array}
	\right.\]
	\end{lemma}
\begin{proof}
We only sketch the construction (see  \cite{Se06} and \cite{KvK16} for more details).  Choose a Hamiltonian function $H$ which is $C^2$-small in the interior, and also a $C^2$-small time dependent Hamiltonian $F_t$ which is constructed from Morse function $f$ on $\partial M$ so that $H_t = H+F_t$. Nontrivial orbits of $H_t$ appear as a small perturbation of possible family of (degenerate) orbits. Since the action of $\mathcal R_t$ is free, there are integer many such families, classified by its period, all diffeomorphic to $\partial M$. An action value of orbits of period $n$ is given by
    \begin{align*} 
    \mathcal A_{H_{S^1}}(\gamma) &:= -\int_{S^1}\gamma^*\lambda+\int_0^1H_{S^1}(\gamma(t))dt\\
    &=-2\int_0^1r^2dt+\int_0^1H(\gamma(t))dt+\int_0^1 F(\gamma(t))dt, \hskip 0.2cm (H_{S^1}=H(r)+F(r,t))\\
    &= -\int_0^1r^2+\mathrm{error} \\
    &\in (-r^2-\epsilon, -r^2+\epsilon), \hskip 0.2cm \epsilon= \mathrm{max}f - \mathrm{min}f.
    \end{align*}  
Here, we choose small enough $\epsilon \ll 1$ so that $[-(r+1)^2-2\epsilon,  -(r+1)^2 +2\epsilon]$ is disjoint from $[-r^2-2\epsilon, -r^2 +2\epsilon]$ for all $r \in \Z_{\geq0}.$ We define an action filtration on $CH^*(M; H_{S^1})$ as
\[F^{-n} := \left\{ \gamma \in \mathcal O_{H_{S^1}} \vert \mathcal A_{H_{S^1}}(\gamma) > -n^2- \epsilon \right\},\]
the spectral sequence associated to $F^*$ converges to $SH^*(M)$, and its first page is
	\[E^{pq}_1 = 
	\left \{ 
	\begin{array}{ll}
	H^q(M) & (p=0) \\
	H^{p+q + \star_p}(\partial M) & (p <0)
	\end{array}
	\right.\]
A degree shifting number $\star_p$ (for $p <0$)  is given as $\mu_{RS}(\partial M_p)$, where $\partial M_p$ denotes a component corresponds to the period $-p$ component. Notice that, we have $\mu_{RS}(\partial M_p) = p\cdot \mu_{RS}(\partial M_1)=-p\cdot \deg(\Gamma_M)$ because $\Gamma_M$ is chosen to be  $1_{\partial M}\in H^0(\partial M_1)$. 
\end{proof}  

\subsubsection{Step II: vanishing}
\begin{prop}\label{prop:vanishing}
If Robbin-Salamon index of a principle component (that is, $\mu_{RS}(\partial M_1)$)is non-negative, then popsicle spheres with two or more $\Gamma_M$ inputs vanish.
\end{prop}
\begin{proof}
For simplicity, we write $\mu:=\mu_{RS}(\partial M_1)$. 

Let $\xi$ be an output of a sphere component satisfying Corollary \ref{degree inequality 2} with $N$-many $\Gamma_M$ inputs. Since $\deg(\Gamma_W) = -\mu$, we obtain the following inequality  by Corollary \ref{degree inequality 2}
	\begin{equation}
	\label{degree inequality 3}
	\deg \xi \leq N\cdot \deg \Gamma_M - \frac{N}{2} = -N\cdot \mu-\frac{N}{2}.
	\end{equation}

Suppose $\mu=0$. Then Lemma \ref{energy spectral sequence} implies $SH^*$ is concentrated in positive degrees, while \ref{degree inequality 3} implies the degree of $\xi$ is less then $-N/2$. This implies the proposition. 

For the general case $\mu \geq 0$, we proceed as follows. Suppose $\xi \in F^{-l}$. Since $\mu$ is non-negative, we must have $\deg \xi \geq -l \cdot \mu$ from Lemma \ref{energy spectral sequence}.
Combining with the inequality from Corollary \ref{degree inequality 2}, we have 
	\begin{equation}
	\label{degree inequality 41}
	-l\cdot \mu \leq \deg \xi \leq -N\cdot \mu-\frac{N}{2} \Longrightarrow l\geq N\left (1+\frac{1}{2\mu}\right)
	\end{equation}
On the other hand,  from  the action inequality, we have
	\begin{equation}
\frac{-l^2 + \epsilon}{N}\geq \frac{\mathcal A_{H_{S^1}}(\xi)}{N} \geq \sum_{i=1}^N \mathcal A_{H_{S^1}}(\Gamma_M)\geq N \cdot(-1-\epsilon)	\end{equation}
Here, we choose our weights at each cylindrical ends as $w_0^-=N$ and $w_i^+=1$. 
Thus we obtain
	\begin{equation}
	\label{action inequality0}
	 l^2 \leq N^2(1-\epsilon)-\epsilon
	\end{equation}

Combine \ref{degree inequality 41} and \ref{action inequality0}, we should have the following inequality due to the sphere component.
	\[ N^2 \left(1+\frac{1}{\mu}\right)^2 \leq l^2 \leq  N^2 (1-\epsilon) -\epsilon. \]
If we choose $\epsilon$ small enough,  this inequality does not hold for any integer $N\geq 2$. 
Hence such a sphere bubble does not arise.
\end{proof}
    

  \subsection{$\boldsymbol{A_\infty}$-category $\boldsymbol{\CG}$}
  We construct a new $\Z$-graded $A_\infty$-category $\CG$ under the assumption  \ref{as:main}.
  
   We start with the following cancellation results from Abouzaid-Seidel.
     \begin{prop}(\cite{AS}) \label{two interior markings implies vanishing}
  If $\phi:F\to \{1, \ldots, n\}$ is not injective, then $\bold{P}^\Gamma_{n, F, \phi}$ vanishes.
  \end{prop}
  \begin{proof}
    The assumption means that at least one popsicle stick carries more than two interior markings. Then $Sym^\phi$ contains a nontrivial transposition. Since we put the same class $\Gamma$ for all interior markings, the transposition extends to $\OL{\mathcal P}_{n, F, \phi}(\Gamma; a_1, \ldots, a_n)$ also. It induces an orientation-reversal automorphism on $\OL{\mathcal P}_{n, F, \phi}$.  Therefore the contribution of this moduli space should vanish.
  \end{proof}
  Now we can focus on the case when $\phi:F \to \{1, \ldots, n\}$ is injective. Then $F$ can be considered as a subset $\{i_1,\ldots, i_k\}\subset \{1, \ldots, n\}$. In this case, we omit a notation $\phi$ and simply write $\overline{\mathcal P}_{n,F}$ and $\bold{P}^\Gamma_{n, F}$. 
  \begin{defn}
  An admissible cut of $F$ consists of 
    \begin{enumerate}
      \item $n_1, n_2 \geq 1$ such that $n_1+n_2=n+1$,
      \item a number $i \in \{1, \ldots, n\}$,
      \item $F_1\subset \{1, \ldots, n_1\}$ and $F_2\subset \{1, \ldots, n_2\}$ such that $|F_1|+|F_2|=|F|$,
    \end{enumerate}
    satisfying the following properties;
    \begin{itemize}
      \item $F\supset \{k|k\in F_1, k<i\}$ and $F\supset \{k+n_2-1|k\in F_1, k>i\}$,
      \item $F\supset \{k+i-1|k\in F_2\}$.
    \end{itemize}
    If $i\notin F_1$, then this completely recovers $F$. Otherwise, $F$ has one more element among $\{i, i+1, \ldots, (i+n_2-1)\}.$
  \end{defn}
  An admissible cut describes a stratum $P_{n_1,F_1}\times P_{n_2,F_2}$ of a moduli space $\overline P_{n, F}$. They describe precisely codimension $1$ strata whose associated flavour $\phi_j: F_j \to \{1, \ldots, n_j\} \hskip 0.2cm (j=1,2)$ is still injective (when sphere bubbles of codimension one do not appear). Combined with \ref{two interior markings implies vanishing}, we get a quadratic relation
  \begin{align*}
  \sum_{\forall\textrm{admissible cuts}}&(-1)^\clubsuit  \bold{P}^\Gamma_{n_1, F_1}(a_1, \ldots, a_{i-1}, \bold{P}^\Gamma_{n_2, F_2}(a_i, \ldots, a_{i+|F_2|-1}),a_{i+|F_2|}, \ldots, a_{n}) =0, \\
  \clubsuit &= n_2 i +i+1+n_1|F_2|+(n_2+|F_2|)\left(\sum _{j\leq i+|F_2|} \deg a_j\right)\\
  &+\left| \left\{(f_1, f_2) \in F_1 \times F_2 | f_1<f_2 \hskip 0.1cm \textit{inside $F$}\right\} \right|
  \end{align*}

 Now we are ready to define a new $A_\infty$-category. 
  \begin{defn}\label{defn:newa}
The $\AI$-category $\mathcal{C}_\Gamma$ for
 $\Gamma_M \in SH^\bullet(M)$ from Assumption \ref{as:main}, consists of  
    \begin{enumerate}
      \item a set of objects $Ob(\mathcal{C}_\Gamma)=Ob(\mathcal{WF}(M))$. 
      \item morphisms between two objects
        \[\Hom_{\mathcal{C}_\Gamma}(L_1, L_2) = CW(L_1, L_2) \oplus CW(L_1, L_2)\epsilon \]
      Here $\deg\epsilon=-1 + \mu_{RS}(\Gamma)$, and we denotes the element of this complex by $c:=a+\epsilon b$. 
      \item An $A_\infty$-structure $\{M_n\}_{n=1}^\infty$ is given as follows.
       We may write 
       $$M_n(c_1,\ldots,c_n) = M_n^a(c_1,\ldots,c_n) + \epsilon M_n^b(c_1,\ldots,c_n)$$       
      \begin{enumerate}
      \item Suppose $c_i = a_i$ for all $i$ ( all the inputs do not have $\epsilon$ components), then  we set 
             $$M_n(a_1,\ldots,a_n) = m_n(a_1,\ldots,a_n)$$ where $\{m_n\}$ is the $\AI$-operation for $\mathcal{WF}(M)$.
      \item Suppose $c_i = \epsilon b_i$ for $ i  \in \{i_1,\ldots, i_k\}$, and $c_i = a_i$  for  $i  \notin \{i_1,\ldots, i_k\}$.   Then we set 
       $$F=\{i_1, \ldots i_k\}, \;\;\; \widehat F^j =\{i_1, \ldots,\widehat{i_j}, \ldots, i_k\},$$
       and define 
$$M_n^a(c_1,\ldots, c_n) =(-1)^{\star^a_{n, F}} \bold{P}^\Gamma_{n ,F}(a_1,\ldots,b_{i_1},\ldots,b_{i_j},\ldots, b_{i_k},\ldots, a_n) $$
$$ M_n^b(c_1,\ldots, c_n)   = \sum_{j=1}^k (-1)^{\star_{j-1} + \star^b_{n, \widehat {F} ^j}} \bold{P}^\Gamma_{n ,\widehat F^j}(a_1,\ldots,b_{i_1},\ldots,b_{i_j},\ldots, b_{i_k},\ldots, a_n)$$

If we use the common notion $x_i$ to denote $a_i$ and $b_i$.
Then  
\begin{eqnarray*}
\star^a_{n, F}  &=&  \sum_j j \deg x_j + \sum_{f \in F, l>f} (\deg x_l -1) \\
\star^b_{n, \widehat{F}^j} &=& \sum_j j \deg x_j + \sum_{f \in \widehat{F}^j , l>j} (\deg x_l -1) \\
\star_j &=& \sum_{l =1}^{j} (\deg x_l -1)
\end{eqnarray*}
       \end{enumerate}
    \end{enumerate}
  \end{defn}
  As a sanity check, let us check the degree of $M_n$. Recall that the degree of $\bold{P}^\Gamma_{n, F}$ is $2-n-|F|(1- \mu_{RS}(\Gamma_M))$, and
  the corresponding operation in terms of $M_n$ has  degree $2-n$ because the degree of $\epsilon$ is $-1+\mu_{RS}(\Gamma_M)$ and we have $|F|$ many $\epsilon$-inputs.
%

We follow the sign analysis of Abouzaid-Seidel \cite{AS}, to which we refer readers for full details.
Using the orientation operators, they have separated the signs from the popsicle domain and the popsicle maps.
They used  the standard gluing theorem identify the sign factors from the popsicle maps, and the signs from
degeneration of popsicle domains and the rearrangements of factors determined the sign convention.
As we have a fixed interior input $\Gamma_M$, the contribution from the popsicle maps can be still glued using the standard gluing theorem, whereas
the remaining part of the sign analysis is the same as the reference. Thus, the same sign convention can be used in our cases as well.

\subsection{Proof of $\boldsymbol{\AI}$-identity}
   \begin{prop}
  $\mathcal{C}_\Gamma$ is an $A_\infty$-category. Namely, for any composable $(c_1,\ldots,c_n)$,
  we have
  $$ \sum_{n_1+n_2=n+1}(-1)^{\sum_{j=1}^{i-1}|c_j|'} M_{n_1}(c_1,\ldots,c_{i-1}, M_{n_2}(c_i, \ldots,c_{i+n_2-1}),\ldots,c_n) =0.$$
  \end{prop}
  \begin{proof}
We check the identity on each component of the output. 
We first show that 
$$\sum  \big( M_{n_1}^a (\ldots, M_{n_2}^a(\ldots),\ldots) \pm M_{n_1}^a (\ldots, M_{n_2}^b(\ldots),\ldots)\big)= 0.$$
This identify follows from the compactification of  popsicle moduli spaces. Namely,
a codimension one stratum of the popsicle moduli space corresponds to a term in the above equation. 
In Figure \ref{fig:pop0}, we illustrated   corresponding broken popsicles for the case $|F|=4$.
 Even for broken popsicles, we can consider a sequence of hyperbolic geodesics connecting
$z_0$ and $z_{i_j}$. In the figure, dotted lines are such geodesics that do not contain a sprinkle.

\begin{figure}[h]
\includegraphics[scale=0.75, trim=0 700 0 10, clip]{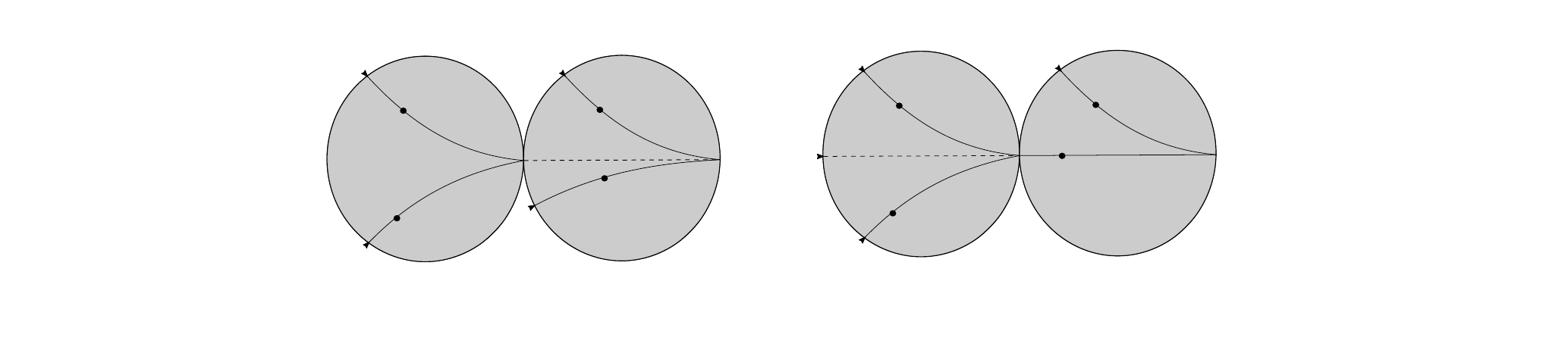}
\centering
\caption{\label{fig:pop0}$\AI$-identity with $a$-output}
\end{figure}

Next we show that
$$\sum  \big( M_{n_1}^b (\ldots, M_{n_2}^a(\ldots),\ldots) \pm M_{n_1}^b (\ldots, M_{n_2}^b(\ldots),\ldots)\big)= 0.$$
This identify follows from the compactification of popsicle moduli spaces for $\widehat{F}^j$ for all $j$.
In Figure \ref{fig:popb}, we illustrated  corresponding broken popsicles  for the case $|F|=4$ and $j=1$. We also expressed the forgotten geodesic for $M^b$-operation as dotted lines.
\begin{figure}[h]
\includegraphics[scale=0.75, trim=0 700 0 10, clip]{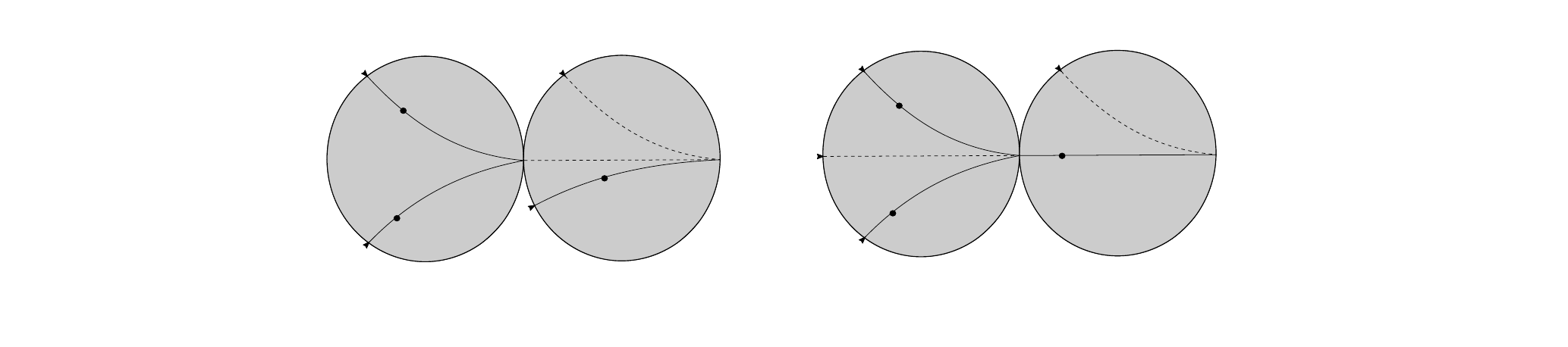}
\centering
\caption{\label{fig:popb}$\AI$-identity with $b$-output}
\end{figure}



 \end{proof}

\subsection{Example: $\boldsymbol{M_2}$-operation}

Let us examine the following Leibniz rule for the input $(a, \epsilon b)$. 
\begin{equation}\label{eq:leib1}
 M_1(M_2(a,\epsilon b)) + M_2(M_1(a),\epsilon b) + (-1)^{|a|'} M_2(a,M_1(\epsilon b))=0.
 \end{equation}
For simplicity, we will omit the signs from the formulas.
From the definition
$$M_2(a,\epsilon b) = \bold{P}^\Gamma_{2,\{2\}}(a,b) + \epsilon m_2(a,b)$$
$$M_1(\epsilon b) = \bold{P}^\Gamma_{1,\{1\}}(b) + \epsilon m_1(b) $$
We have 
\begin{eqnarray*}
M_1(M_2(a,\epsilon b)) &=& M_1(\bold{P}^\Gamma_{2,\{2\}}(a,b) + \epsilon m_2(a,b)) = m_1 (\bold{P}^\Gamma_{2,\{2\}}(a,b)) + 
\big( \bold{P}^\Gamma_{1,\{1\}} (m_2(a,b)) + \epsilon m_1 (m_2(a,b)) \big)\\
M_2(M_1(a),\epsilon b) &=& M_2(m_1(a), \epsilon b) =  \bold{P}^\Gamma_{2,\{2\}}(m_1(a),b) + \epsilon m_2(m_1(a),b)\\
M_2(a,M_1(\epsilon b)) &=& M_2(a, \bold{P}^\Gamma_{1,\{1\}}(b) + \epsilon m_1(b)) =
m_2(a, \bold{P}^\Gamma_{1,\{1\}}(b)) + \bold{P}^\Gamma_{2,\{2\}}(a,m_1(b))+ \epsilon m_2(a, m_1(b))
\end{eqnarray*}
If we collect the terms with $\epsilon$  in \eqref{eq:leib1}, 
we obtain the original $\AI$-identity
$$\epsilon\big( m_1(m_2(a,b)) + m_2(m_1(a),b) + (-1)^{|a|'}m_2(a,m_1(b))\big)=0.$$
\begin{figure}[h]
\includegraphics[scale=0.7, trim=0 600 0 0, clip]{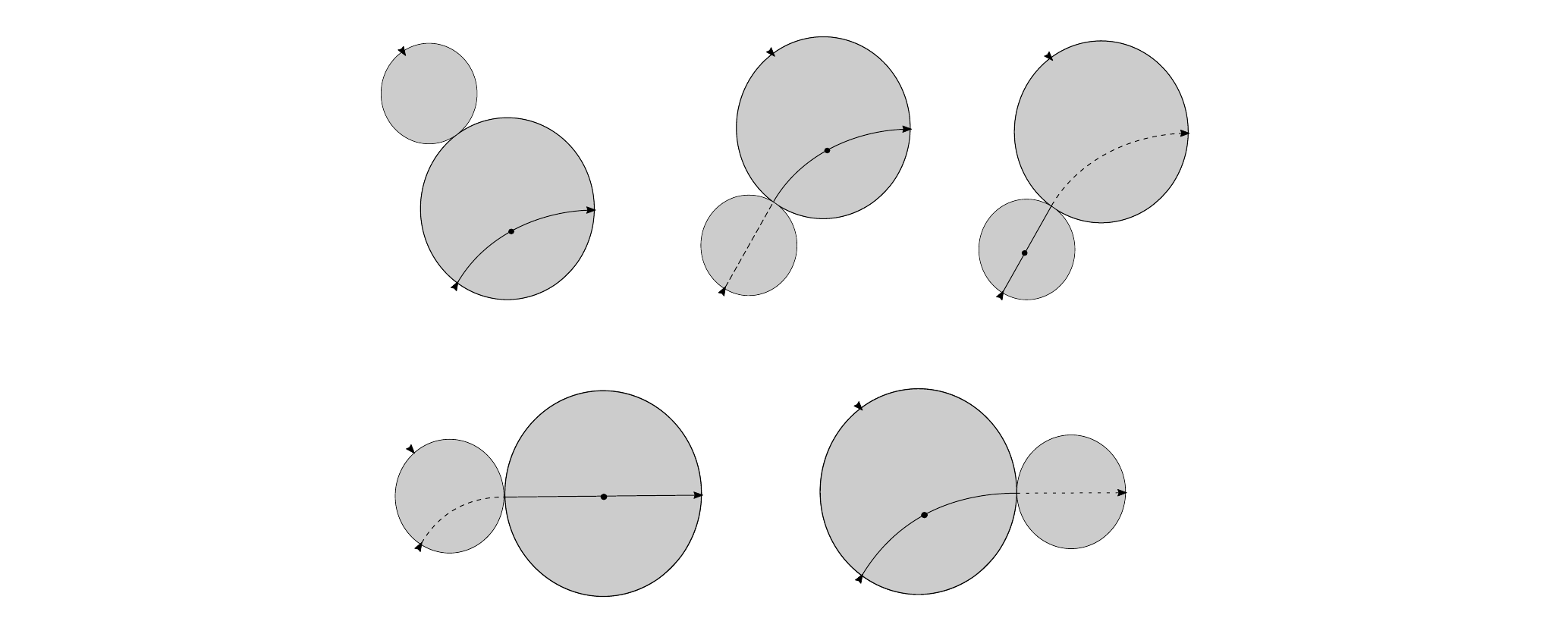}
\centering
\caption{\label{fig:Leib1}Leibniz rule for the inputs $(a, \epsilon b)$}
\end{figure}
Collecting the terms without $\epsilon$ in \eqref{eq:leib1}, we get the following(up to sign)
$$  \bold{P}^\Gamma_{2,\{2\}}(m_1(a),b) + \bold{P}^\Gamma_{2,\{2\}}(a, m_1(b)) + m_2(a,  \bold{P}^\Gamma_{1,\{1\}}(b)) +  \bold{P}^\Gamma_{1,\{1\}}(m_2(a,b))  + m_1 \bold{P}^\Gamma_{2,\{2\}}(a,b) =0.$$
These terms correspond to the codimension one degenerations (given by disc bubblings) in Figure \ref{fig:Leib1}.
Here dotted lines just indicate  paths to the $0$-th vertex, and do not give any restriction to the domain.
Hence one may remove dotted lines to find the corresponding $\AI$-operations.

Let us examine Leibniz rule for the input $(\epsilon b_1, \epsilon b_2)$.
Namely, we want to verify
\begin{equation}\label{eq:leib2}
 M_1(M_2(\epsilon b_1,\epsilon b_2)) + M_2(M_1(\epsilon b_1),\epsilon b_2) + (-1)^{|b_1|} M_2( \epsilon b_1,M_1(\epsilon b_2))=0.
 \end{equation}
 We have 
\begin{eqnarray*}
M_1(M_2(\epsilon b_1,\epsilon b_2)) &=& M_1(\bold{P}^\Gamma_{2,\{1,2\}}(b_1,b_2) + \epsilon \bold{P}^\Gamma_{2,\{2\}}(b_1,b_2) + \epsilon \bold{P}^\Gamma_{2,\{1\}}(b_1,b_2) )  \\
&=&   m_1(\bold{P}^\Gamma_{2,\{1,2\}}(b_1,b_2)) + \bold{P}^\Gamma_{1,\{1\}}( \bold{P}^\Gamma_{2,\{2\}}(b_1,b_2) + \bold{P}^\Gamma_{2,\{1\}}(b_1,b_2) ) + \epsilon m_1 \big( \bold{P}^\Gamma_{2,\{2\}}(b_1,b_2) + \bold{P}^\Gamma_{2,\{1\}}(b_1,b_2) \big) 
\\
M_2(M_1(\epsilon b_1),\epsilon b_2) &=& M_2(\bold{P}^\Gamma_{1,\{1\}}(b_1) + \epsilon m_1(b_1), \epsilon b_2)\\ 
&=&  \bold{P}^\Gamma_{2,\{2\}}(\bold{P}^\Gamma_{1,\{1\}}(b_1),b_2) + \epsilon m_2(\bold{P}^\Gamma_{1,\{1\}}(b_1),b_2) \\
&& + \bold{P}^\Gamma_{2,\{1,2\}}(m_1(b_1),b_2) + \epsilon \bold{P}^\Gamma_{2,\{2\}}(m_1(b_1),b_2) + \epsilon \bold{P}^\Gamma_{2,\{1\}}(m_1(b_1),b_2)  \\
M_2(\epsilon b_1,M_1(\epsilon b_2)) &=& M_2(\epsilon b_1, \bold{P}^\Gamma_{1,\{1\}}(b_2) + \epsilon m_1(b_2)) \\                                                                          
&=& \bold{P}^\Gamma_{2,\{1\}}( b_1, \bold{P}^\Gamma_{1,\{1\}}(b_2)) + \epsilon m_2(b_1, \bold{P}^\Gamma_{1,\{1\}}(b_2)) \\
&&+ \bold{P}^\Gamma_{2,\{1,2\}}(b_1,m_1(b_2)) + \epsilon \bold{P}^\Gamma_{2,\{2\}}(b_1,m_1(b_2)) + \epsilon \bold{P}^\Gamma_{2,\{1\}}(b_1,m_1(b_2))
 \end{eqnarray*}
 
 The following figure \ref{fig:Leib2} describes the terms without $\epsilon$ in the above (in the same order).
 \begin{figure}[h]
\includegraphics[scale=0.7, trim=0 600 0 0, clip]{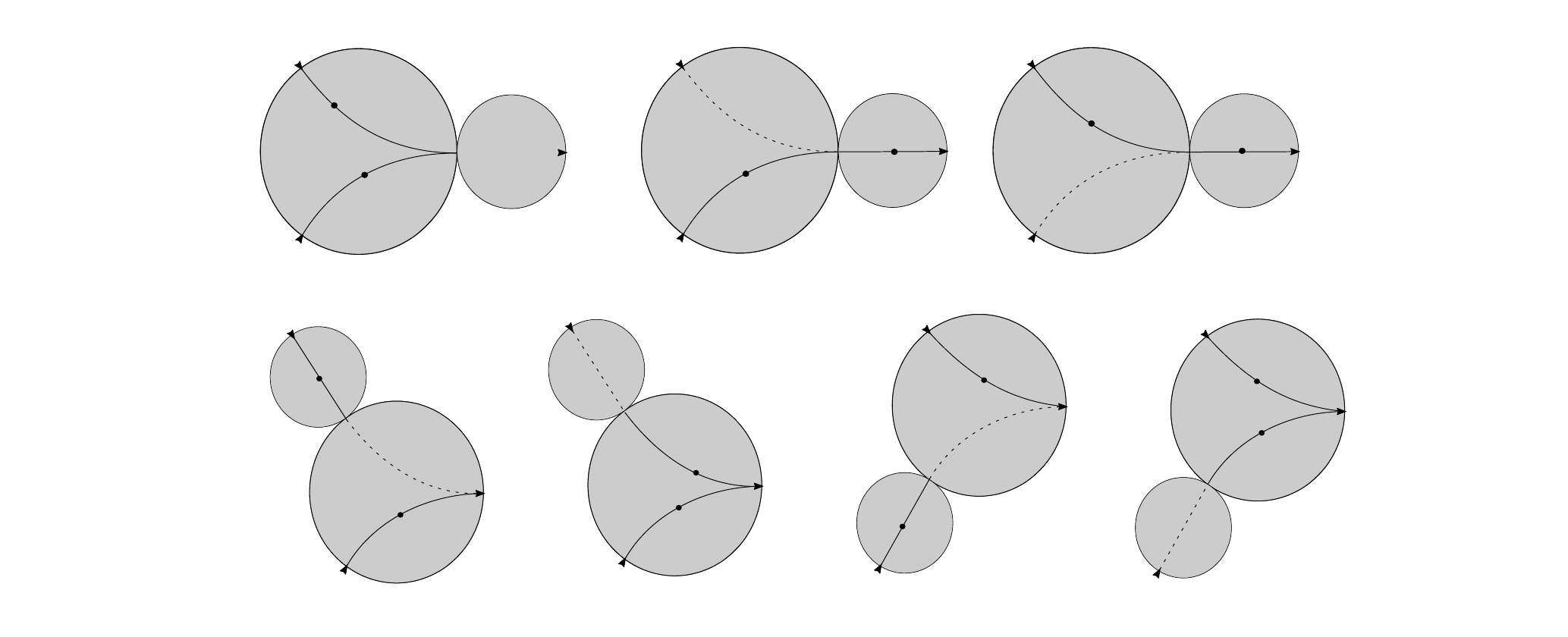}
\centering
\caption{\label{fig:Leib2}Leibniz rule for the inputs $(\epsilon b_1, \epsilon b_2)$}
\end{figure}
It is not hard to see that these arise from codimension one boundary strata of $\OL{P}_{2, \{1,2\}}$.  The terms with $\epsilon$ are similar.


\section{Mirror counterpart: Hypersurface restriction}\label{counterpart}
A mirror construction of last section for algebraic geometry is just a categorical reformulation of the restriction to a hypersurface. 
 
 \subsection{Restriction to a hypersurface in $\boldsymbol{D^bCoh}$}
  Let $S$ be an algebra. Choose an element 
    \[g \in  Z(S) \cong HH^0(S,S)\] 
  The DG-bimodule $\begin{tikzcd}S \arrow[r, "g"] &S\end{tikzcd}$ is quasi-isomorphic to an ideal quotient $S/(g)$ which carries a natural algebra structure. One can directly construct DG algebra structure on the bimodule itself:

\begin{defn}
Define a DG algebra 
      \[\mathcal B:= S[\epsilon]\bigg/\left ( \begin{array}{l}\epsilon^2=0\\d\epsilon =g \end{array}\right ), \hskip 0.2cm \]
      Here $\deg \; \epsilon = -1$ and the differential $d$ on $S$ is set to be zero.
\end{defn}

  Further assume that $S$ is commutative. Consider an affine variety $X=\Spec(S)$ and a hypersurface $Y=V(g)$ with an inclusion $i:Y \hookrightarrow X $. We have the following elementary lemma whose proof is omitted.
\begin{lemma}  
 We have an isomorphism $\mathcal B \simeq i_*\mathcal O_Y$. Moreover, we have the following.
    \begin{enumerate}
      \item A sheaf $\cF$ on a hypersurface $Y$ corresponds to an $\mathcal B$-module object. It is a pair $(i_*\cF, h_{\cF})$ where $i_*\mathcal F$ is a pushforward of $\cF$ equipped with a homotopy $h_\cF$ between the zero map and a multiplication of $ g$. It is an action of $\epsilon \in \mathcal B$.
      \item  Moreover, 
        \[\Hom_Y(\cF_1, \cF_2)\simeq \Hom_\mathcal B ((i_*\cF_1, h_{\cF_1}), (i_*\cF_2, h_{\cF_2})).\] 
    \end{enumerate}
\end{lemma}
For the sheaf $\mathcal O_Y$ on $Y$, its pushfoward $i_*\mathcal O_Y$ has a simple free resolution.
    \[
      \begin{tikzcd}
        0\arrow[r]&\mathcal O_X \arrow[r,"g"]&\mathcal O_X \arrow[r] &i_*\mathcal O_Y \arrow[r] &0
      \end{tikzcd}
    \]  
  An action of degree $-1$ element $\epsilon$, or a homotopy $h$, is given as follows.
     \[
      \begin{tikzcd}
        0\arrow[r]&\mathcal O_X \arrow[r,"g"]  &\mathcal O_X \arrow[r]& 0\\
        0\arrow[r] \arrow[ur, "0", leftarrow]&\mathcal O_X \arrow[r,"g"]  \arrow[ur, "id", leftarrow]&\mathcal O_X \arrow[r] \arrow[ur, "0", leftarrow]& 0
      \end{tikzcd}
     \]
     
 A category of coherent sheaves on $Y$ is described as $\mathcal B$-modules of $X$. 
    \begin{thm}
    Let $Y\subset X$ as before. 
    Then 
      \[DCoh(Y)\simeq \mathcal B-mod(DCoh(X))\]
    \end{thm}
     \begin{proof} A concise categorical proof using Lurie's Barr-Beck theorem can be found in Corollary 3.3.1 in  \cite{Preygel2011}. We present an elementary proof to illustrate the idea. Since everything is affine, it is enough to consider a structure sheaf $\mathcal O_Y \in DCoh(Y)$.
         Computation shows that the morphism complex is 
    \begin{align*}
    \Hom^{-1}_\mathcal B ((i_*\mathcal O_Y, h), (i_*\mathcal O_Y, h)) &\simeq \left \{
        \begin{array}{c|c}
        \begin{tikzcd}[ampersand replacement=\&]
          \mathcal O_X \arrow[r,"g"] \& \mathcal O_X \\
          \mathcal O_X \arrow[r,"g"] \arrow[ur,"a_{21}", leftarrow]\& \mathcal O_X 
        \end{tikzcd}
        &   \textrm{$a_{21}$ can be arbitrary}
        \end{array}
        \right \} \\
      \Hom^{0}_\mathcal B ((i_*\mathcal O_Y, h), (i_*\mathcal O_Y, h)) &\simeq \left \{
        \begin{array}{c|c}
        \begin{tikzcd}[ampersand replacement=\&]
          \mathcal O_X \arrow[r,"g"]\ar[d, "a_{11}"]\& \mathcal O_X \arrow[d,"a_{22}"]\\
          \mathcal O_X \arrow[r,"g"] \& \mathcal O_X 
        \end{tikzcd}
        & a_{11}\circ h=h\circ a_{22} \hskip 0.2cm \textrm{which implies} \hskip 0.2cm  a_{11}=a_{22}
        \end{array}
        \right \} \\
     \Hom^1_\mathcal B ((i_*\mathcal O_Y, h), (i_*\mathcal O_Y, h)) &\simeq \left \{
        \begin{array}{c|c}
        \begin{tikzcd}[ampersand replacement=\&]
          \mathcal O_X \arrow[r,"g"] \arrow[dr, "a_{12}"]\& \mathcal O_X \\
          \mathcal O_X \arrow[r,"g"] \& \mathcal O_X 
        \end{tikzcd}
        & h\circ a_{12}=0 \hskip 0.2cm \textrm{which implies $a_{12}=0$}
        \end{array}
        \right \}   
      \end{align*}
    Therefore, $\Hom^\bullet_\mathcal B ((i_*\mathcal O_Y, h), (i_*\mathcal O_Y, h))$ is isomorphic to
    \[ H^\bullet (\begin{tikzcd} \Hom_X(\mathcal O_X, \mathcal O_X) \arrow[r, "g"] &\Hom_X(\mathcal O_X, \mathcal O_X) \end{tikzcd})\simeq \Hom_Y^\bullet (\mathcal O_Y, \mathcal O_Y).\]
    \end{proof}
  Intuitively,     
  objects of $\mathcal B-mod(DCoh(X))$ are cones
     \[\left(\begin{tikzcd}\cF[1] \arrow[r, "g"] & \cF \end{tikzcd}\right), \hskip 0.2cm \cF\in DCoh(X).\] 
   It is quasi-isomorphic to a quotient $\cF/(g)$ and the space of morphisms is a half of the original one.  It consists of 
    \[ \begin{pmatrix}a&0\\b&a\end{pmatrix} \in \Hom_X^\bullet \left( \begin{tikzcd}\cF_1[1] \arrow[r, "g"] & \cF_1 \end{tikzcd}, \begin{tikzcd}\cF_2[1] \arrow[r, "g"] & \cF_2 \end{tikzcd} \right).\]
This is closed under DG operations whereas $\AI$-analogue is not (see the paragraph after the Definition \ref{cor:tc}).

 \subsection{Restriction to a graph hypersurface in Matrix factorizations}
 Let us explain an analogous construction for matrix factorizations.
 First, let us recall its definition for readers' convenience.
 
 Let $R$ be a polynomial algebra over the algebraically closed field $k$ of characteristic 0.

\begin{defn}
For $W \in R$, a DG category of matrix factorization of $W$, denoted by $\MF(W)$ consists of the following:
\begin{enumerate}
\item its objects are matrix factorization $(P,\delta_P)$ of $W$.  $P$ is $\mathbb{Z}/2$-graded free $R$-module and $\delta_P$ is an odd degree endomorphism such that $\delta_P^2 = W \cdot id$.
\item $\Hom_{\MF(W)}^\bullet \left((P, \delta_P),(P', \delta_{P'})\right):=\left(\Hom_R^\bullet(P, P'), d\right)$
with usual composition $\circ$. A differential $d$ on morphisms is defined as
$$d(\phi) = \delta_{P'} \circ \phi - (-1)^{\deg(\phi)} \phi \circ \delta_{P}.$$
\end{enumerate}
\end{defn}

  We are interested in the following situation. Consider a polynomial of the form $U=U_1(x_1, \ldots, x_{n-1})+x_n\cdot U_2(x_1, \ldots, x_{n-1})$. We consider a graph of some polynomial $f$
  \begin{align*}
    \mathbb A^{n-1} &\to \mathbb A^n\\
    (x_1, \ldots, x_{n-1}) &\mapsto (x_1, \ldots, x_{n-1}, f)
  \end{align*}
  and a pull-back 
    \[V(x_1, \ldots, x_{n-1})=U(x_1, \ldots, x_{n-1}, f)=U_1+f\cdot U_2\]  
  along the graph. 
  We explain how to obtain a similar relation between $\MF(U)$ and $\MF(V)$. We start by collecting functorial properties between two matrix factorization categories, which we refer to \cite{Orlov} and \cite{POS11a}. 
     
     Let $X=\{U= 0\}\subset \C^n$ and $Y=\{V=0\}\subset \C^{n-1}$. We view $Y$ as a hypersurface $\{x_n=f\}\subset X$ . A closed embedding $Y \hookrightarrow X$ is proper and has a finite tor-dimension. A usual adjoint pair of functors $(i^*, i_*)$ extends to categories of singularities. 
    \[i^*: D^b_{sg}(X) \longleftrightarrow D^b_{sg}(Y):i_*\]
    On the other hand, there are Orlov's equivalences
      \[\MF(U) \simeq \overline{D^b_{sg}(X)}, \hskip 0.2cm \MF(V) \simeq \overline{D^b_{sg}(Y)}\]
    Here, $\overline{\mathcal{C}}$ denotes Karoubi completion of a category $\mathcal{C}$. This functor sends 
      \[M=(
        \begin{tikzcd}
          M^{odd} \arrow[r,"\phi_{10}", shift left=1] & M^{even}\arrow[l, "\phi_{01}", shift left=1]
        \end{tikzcd} 
        ) \mapsto \mathrm{coker}(\phi_{10})
      \]
    We have an induced pair
      \[i^*: \MF(U) \longleftrightarrow \MF(V):i_*\]
    \begin{prop}\label{prop:MFadjoint}
    Let 
      \[M=(\begin{tikzcd}
          M^{odd} \arrow[r,"\phi_{10}", shift left=1] & M^{even}\arrow[l, "\phi_{01}", shift left=1]
        \end{tikzcd})\in \MF(U),
      \]
      \[N=(\begin{tikzcd}
          N^{odd} \arrow[r,"\psi_{10}", shift left=1] & N^{even}\arrow[l, "\psi_{01}", shift left=1]
        \end{tikzcd})\in \MF(V).
      \] 
      Then
      \begin{enumerate}
        \item $(i^*, i_*)$ is an adjoint pair.
        \item $i^*M\simeq M|_{x_n=f} \in \MF(V)$.
        \item $i_*N \simeq N\otimes \left(\begin{tikzcd} \mathbb C[x_1, \ldots, x_n]\arrow[r, bend left=30, "x_n-f"]& \mathbb C[x_1, \ldots, x_n] \arrow[l, bend left=30, "U_2"] \end{tikzcd}\right)\in \MF(U)$
        \item $(i_* \circ i^*)M = \Cone((x_n-f):M[1]\to M) \in \MF(U)$
      \end{enumerate}
    \end{prop}  
    \begin{proof}
    The first proposition is proven in more general setup. See \cite{POS11a} Section 2.1. 
    Second proposition follows from the fact that cokernel commutes with tensor product.
      \[\mathrm{coker}(\phi_{10})\otimes_{\C[x_1,\ldots, x_n]} \C[x_1,\ldots, x_{n-1}] \simeq \mathrm{coker}(\phi_{10}|_{x_n=f}).\] 
    To prove a third proposition, we should specify Fourier-Mukai kernel of a push-forward functor. Write
    \[V(x_1, \ldots, x_{n-1})-V(y_1, \ldots, y_{n-1}) = \sum_i^{n-1} (x_i-y_i)\cdot V_i\] 
    Define a Koszul-type matrix factorization $\Gamma$ of $V(\vec x)-U(\vec y)= V(\vec x)-\left(V(\vec y) +(y_n-f) U_2(\vec y)\right )$ as 
    \[\Gamma:=\left( \Lambda^\bullet \langle e_1, \ldots, e_n\rangle, \left(\sum_1^{n-1}(x_i-y_i)i_{e_i} +(y_n-f) i_{e_n}+\sum_1^{n-1}V_i (\cdot \wedge e_i) +U_2(\cdot \wedge V)\right)\right ).\]
    Under Orlov's equivalence $\Gamma$ corresponds to a stabilization of a graph $\Gamma_{Y\to X}$. Therefore a Fourier-Mukai functor associated to $\Gamma$ is a pushforward functor. Notice that
    \[-\otimes \Gamma \simeq -\otimes \Delta_V \otimes \left(\begin{tikzcd} \mathbb C[x_1, \ldots, x_n]\arrow[r, bend left=30, "x_n-f"]& \mathbb C[x_1, \ldots, x_n] \arrow[l, bend left=30, "U_2"] \end{tikzcd}\right)\]
    where $\Delta_V$ is a stabilized diagonal of $V$. This proves the third proposition. 
     
    For the fourth proposition, observe that $i_*\circ i^*M$ goes to  
      \[\mathrm{coker} \left(\phi_{10}|_{x_n=f}:M^{odd}|_{x_n=f} \to M^{even}|_{x_n=f}\right)\] 
    under Orlov's equivalence. It is easy to see that the periodic tail of the following double complex realizes the matrix factorization associated to that module. 
      \[
      \begin{tikzcd}
        \cdots \arrow[r, "\phi_{10}"]  & M^{even} \arrow[r,"\phi_{01}"] \arrow[d, "x_n-f"] & M^{odd} \arrow[r,"\phi_{10}"] \arrow[d, "x_n-f"]&  M^{even} \arrow[d, "x_n-f"]\\
        \cdots \arrow[r, "\phi_{10}"] & M^{even} \arrow[r,"\phi_{01}"]& M^{odd} \arrow[r,"\phi_{10}"]&  M^{even}
      \end{tikzcd}
      \]
    This is equal to $\Cone((x_n-f): M[1] \to M)$.  
    \end{proof}
  The fourth of Proposition \ref{prop:MFadjoint} means
  \[(i_*\circ i^*)M \simeq \left(\begin{tikzcd}M[1] \arrow[r, "x_n-f"]& M\end{tikzcd}\right) = M[\epsilon]\bigg/\left ( \begin{array}{l}\epsilon^2=0\\d\epsilon =x_n-f \end{array}\right ).\]
  This is a perfect analogy of $\mathcal B$-module objects we considered in the last subsection. 
  \begin{defn}\label{defn:newdg}
  Define a DG category $\MF(U)|_{x_n-f}$ as follows:
  \begin{itemize}
    \item its objects consist  of the matrix factorizations $(i_*\circ i^*)M$ for $M\in \MF(U),$
    \item its morphisms $\Hom_{\MF(U)|_{x_n-f}}\left( (i_*\circ i^*)M_1, (i_*\circ i^*)M_2 \right)$ consist  of  
    \[a+\epsilon b \in \Hom_{\MF(U)} \left( (i_*\circ i^*)M_1, (i_*\circ i^*)M_2 \right)=\Hom(M_1[\epsilon], M_2[\epsilon])\]
    Differentials and compositions are induced from $\MF(U)$.
  \end{itemize} 
  \end{defn}  
  The next proposition explains the relation between MF of the restriction $V = U|_{x_n-f=0}$ and the DG category   $\MF(U)|_{x_n-f}$. 
    \begin{cor}\label{cor:DGmodel} 
    Let 
    \[i_*\circ i^* : \MF(U) \to \MF(U)|_{x_n-f}\] 
    be a natural inclusion and  
    \[c: \MF(U)|_{x_n-f} \to \MF(V)\] 
    be a DG functor sending $(i_*\circ i^*)M$ to  the matrix factorization $M|_{x_n-f=0}$. 
    \begin{enumerate}
    \item The following diagram commutes.
  \[\begin{tikzcd} \MF(U) \arrow[r, "i_*\circ i^*"] \arrow[dr, swap, "i^*"] &\MF(U)|_{x_n-f} \arrow[d, "c"] \\  
   &\MF(V) \end{tikzcd}\] 
    \item A category $\MF(U)|_{x_n-f}$ fits into a diagram of distinguished triangle of DG bimodules:
    \[\begin{tikzcd} \MF(U) \arrow[r, "x_n-f"] & \MF(U) \arrow[r, "i_* \circ i^*"] & \MF(U)|_{x_n-f} \arrow[r] & \phantom{a} \end{tikzcd}\]
    \item $c$ is fully faithful. It is a quasi-equivalence whenever $i^*:\MF(U) \to \MF(V)$ is essentially surjective.
    \end{enumerate}
    \end{cor}
    \begin{proof}
      The first and the second proposition follows directly from the definition. For the third one, observe 
      \[i_*\circ i^*M \simeq M[\epsilon]\bigg/\left ( \begin{array}{l}\epsilon^2=0\\d\epsilon =x_n-f \end{array}\right ).\]
   Therefore, we have
    \begin{align*}
      \Hom_{\MF(U)|_{x_n-f}}\left( (i_*\circ i^*)M_1, (i_*\circ i^*)M_2 \right) & \simeq \Hom_{\C[x_1, \ldots, x_n,\epsilon]/(d\epsilon = x_n-f)}(M_1[\epsilon], M_2[\epsilon])\\
      & \simeq \Hom_{\MF(U)} (M_1, M_2[\epsilon])\bigg/\left ( \begin{array}{l}\epsilon^2=0\\d\epsilon =x_n-f \end{array}\right )\\
      & \simeq \Hom_{\MF(U)} (M_1, (i_*\circ i^*)M_2)\\
      & \simeq \Hom_{\MF(V)}(i^*M_1, i^*M_2). 
    \end{align*}
    If $i^*$ is essentially surjective, we have $i_*N  \simeq \left(\begin{tikzcd}M[1] \arrow[r, "x_n-f"]& M\end{tikzcd}\right) $ for some $M\in \MF(U)$. This implies $c$ is also essentially surjective. 
    \end{proof}
In this way, we can regard  $\MF(U)|_{x_n-f}$ as another DG-model representing $\MF(V)$.


\section{$\AI$-category for a weighted homogeneous polynomial with a symmetry group}\label{sec:LGFukaya}

 In this section, we will generalize  the construction of Section \ref{sec:qc1} and define a new $A_\infty$-category $\cF(W,G_W)$ for
 a  weighted homogeneous polynomial $W$ paired with its group $G_W$ of diagonal symmetries. 
 From the monodromy transformation of $W$ and its flow induced on the boundary of Milnor fiber $M_W$, we define
 a Reeb orbit $\Gamma_W$ in the quotient orbifold $[M_W/G_W]$ and use $J$-holomorphic popsicles with $\Gamma_W$-insertions
 to define such an $\AI$-category.

\subsection{Weighted homogeneous polynomials and variation operator}\label{variation operator}
\begin{defn}
We say a weighted homogeneous polynomial $W$ is 
	\begin{enumerate}
	\item log Fano if $\sum_{k=1}^n w_i -h >0$
	\item log Calabi-Yau if $\sum_{k=1}^n w_i -h =0$
	\item log general if $\sum_{k=1}^n w_i -h < 0$
	\end{enumerate}
\end{defn}
 Assume that $W$ has an isolated singularity at the origin. Set $V_t = V_t(W) = \{ z \in \C^n \mid W(z) = t\}$. $V_0$ is an hypersurface of isolated singularity at 0 and $V_t$ ($t \neq 0$) is non-singular. \textit{Milnor fiber} $M_W$ is defined to be  $V_1(W)$.
For the well-known Milnor fibration 
$$\frac{W}{|W|}:S^{2n-1}_\epsilon \setminus K \to S^1$$
with $K =(S^{2n-1}_\epsilon \cap V_0)$,  its fiber is diffeomorphic to $M_W$.
 Geometric monodromy $\Phi_W:\C^n \to \C^n$ is defined by
\begin{equation}\label{eq:gm}
\Phi_W(x_1,\ldots, x_m) =(e^{2\pi i w_1/h} x_1,\ldots, e^{2\pi iw_m/h} x_m)
\end{equation}
which restricts to $M_W$.
It is known that $S^{2n-1}_\epsilon \setminus K$ is diffeomorphic to the manifold
obtained by identifying two ends of $M_W \times [0,1]$ by $\Phi_W$
(see \cite{Milnor} Lemma 9.4).

It is another famous theorem of Milnor that the homotopy type of $M_W$ is a bouquet of $(n-1)$ spheres.
One may define the notion of vanishing skeleton, but it is usually very singular (for Brieskorn-Pham singularities, it is called \textit{Pham's spine}). If we perturb  $W$ to a complex Morse function $W'$, then we may obtain vanishing spheres in the Milnor fiber of $W'$,
and  Fukaya-Seidel directed $\AI$-category of $W$ is defined on these objects. 
But since   $G_W$-symmetry is broken after the perturbation, it is not known how to work equivariantly within Fukaya-Seidel category framework.

On the other hand, it is well-known that non-compact Lagrangians are related to vanishing cycles via the variation operator.
 Consider monodromy homomorphism (from a parallel transport fixing the boundary)
\begin{equation}\label{eq:mh}
h_*: H_*(\OL{M}_W) \to H_*(\OL{M}_W), \;h_*: H_*(\OL{M}_W, \partial \OL{M}_W) \to H_*(\OL{M}_W, \partial \OL{M}_W).
\end{equation}

\begin{defn}  
A \textbf{variation operator} (around the origin in $\C$) 
\begin{equation}\label{eq:vo}
\mathrm{var}: H_{n-1}(\OL{M}_W, \partial \OL{M}_W) \to H_*(\OL{M}_W).
\end{equation}
is defined by sending $[c] \mapsto (h_* - id)([c])$.
\end{defn}
This map is known to be an isomorphism. We will find a symplectic categorical analogue of this variation operator for weighted homogenous polynomials in the next section. We define a distinguished Reeb orbit $\Gamma_W$ from the geometric monodromy \eqref{eq:gm}. 
The analogue of monodromy homomorphism \eqref{eq:mh} will be given by the quantum cap action by $\Gamma_W$
$$ \cap \Gamma_W : \mathcal{WF}(L,L) \to \mathcal{WF}(L,L).$$
Then, the analogue of the variation operator \eqref{eq:vo} will be the new Fukaya category $\mathcal{C}_{\Gamma_W}$,
with a distinguished triangle of $\AI$-bimodules
    \[\begin{tikzcd}
    \mathcal{WF}(M) \arrow[r, "\cap \Gamma_W"] & \mathcal{WF}(M) \arrow[r]& \CG \arrow[r]& \phantom a 
      \end{tikzcd}\]


  \subsection{Diagonal symmetry group of weighted homogeneous polynomials}
  We collect some facts about weighted homogeneous polynomials and their diagonal symmetry group
from  \cite{FJR13} and a reference therein. 
  
  Let $W$ be a weighted homogeneous polynomial of weight $(w_1, \ldots, w_n;h).$
    \begin{defn}
    The maximal diagonal symmetry group $G_W$ of $W$ is defined as 
    \[G_W:=\left\{\vec \lambda \in (\C^*)^n\mid W(\lambda_1\cdot x_1, \ldots, \lambda_n\cdot x_n)=W(x_1, \ldots, x_n)\right\}.\] 
    \end{defn}
    Note that geometric monodromy \eqref{eq:gm} is an element of $G_W$.
    A class of polynomials that we are interested in are those with finite diagonal symmetries. 
    \begin{defn}[\cite{FJR13} Definition 2.1.5]
    A weighted homogeneous polynomial $W$ is called non-degenerate if 
    \begin{enumerate}
      \item $W$ contains no monomial of the form $x_ix_j$ for some $i\neq j$;
      \item $W$ is an isolated singularity at the origin.  
    \end{enumerate} 
    \end{defn}
   If $W$ is non-degenerate, the number of monomials of $W$ must be greater than or equal to the number of variables. The first finiteness condition that we need is the following.
    \begin{prop}[\cite{HK}, \cite{FJR13}]
    Let $W$ be a non-degenerate weighted homogeneous polynomial. Then
      \begin{enumerate}
      \item The weight $(w_1, \ldots, w_n;h)$ of $W$ is bounded by $\frac{w_i}{h} \leq\frac{1}{2}$ and it is unique.
      \item The maximal diagonal symmetry group $G_W$ is a finite abelian group. 
      \end{enumerate}
    \end{prop}
     For a subgroup $K\subset G_W$, let $\C^{N_{K}}:=(\C^n)^K$ be the set of fixed points of $K$. Notice that they are always a coordinate plane. Let $W_K$ be a restriction of $W$ on $\C^{N_K}$. The second finiteness condition that we will use is the following.
  \begin{prop}[\cite{FJR13} Lemma 2.10]
    Let $W$ be a non-degenerate weighted homogeneous polynomial. Then 
    \begin{enumerate}
    \item $W_K$ is a non-degenerate weighted homogeneous polynomial on $\C^{N_K}$ for $\forall K<G_W$.
    \item $G_{W_K}$ is canonically embedded inside $G_W$.
    \end{enumerate} 
  \end{prop} 
  An orbifold strata $\{\textrm{Fix}(K)\}_{K<G_W}$ on $\C^n$ is compatible with restrictions. This will enable us to do inductive arguments on the number of variables later on. A pair $(W, G)$ of non-degenerate homogeneous polynomial $W$ and a subgroup $G<G_W$ of diagonal symmetries is called a \textit{orbifold Landau-Ginzburg model}.

 
  \subsection{Orbifold wrapped Fukaya category}
  The starting point of our construction is an orbifold wrapped Fukaya category of $[M_W/G_W]$. We follow \cite{Se2} closely. 
   
    We consider the collection $\mathcal W$ of Lagrangians   $L \in [\Mil_W/G_W]$ with the following properties.
      \begin{itemize}
    \item It is given by  the family of embedded Lagrangians $\{L_i\}_{i \in I}$ and an action of $G_W$ on the index set $I$
    such that $g (L_i )= L_{g \cdot i}$.
    \item $L$ lies away from the singular locus.
    \item It is either compact or conical at the end.
    \item It is equipped with an orientation and a spin structure (compatible with $G_W$-action).
    \item  $G_W$-orbits  intersect transversally to each other without triple intersections. 
     \end{itemize} 
  \begin{defn}
  A $\Z/2$-graded \textit{orbifold wrapped Fukaya category}  $\mathcal {WF}([M_W/G_W])$ consists of;
    \begin{enumerate}
     \item a set of objects $\mathcal W$;
     \item For two such Lagrangians $L_0, L_1\in\mathcal W$,  its morphism space is defined by 
  \[CW^{ \bullet}(L_0, L_1):= \left(\bigoplus_{g, h\in G_W} CW^\bullet\left(g\cdot \widetilde{L_0}, h\cdot \widetilde{L_1}\right)\right)^{G_W}\]
  where $\widetilde{L_i}$ is a lift of $L_i$ and we take $G_W$-invariant part.
     \item   $A_\infty$-operations are  induced from the $m_k$-operation of $\mathcal {WF}(M_W)$. 	
   \end{enumerate}
  \end{defn}
  
 Following Seidel \cite{Se2}, we can make the $\AI$-structure $G_W$-equivariant. We omit the details and refer readers to \cite{Se2}.
	In addition to the usual $\Z/2$-grading, we can also define $\frac{1}{\vert G_W\vert } \Z$ -graded category.
	Consider a canonical holomorphic volume form  $\Omega = \mathrm{Res}_{W}dz_1\wedge\cdots\wedge dz_n$ on a fiber of $W:\C^n \to C$. 
	It is not difficult to show that we can define the Milnor fiber $M_W$ so that $\Omega$ induces a non-vanishing $n$-form on
	 $\Omega_{M_W}$ on $M_W$.
	Regarding a finite group action of $G_W$ on $M_W$, $\Omega_{M_W}$ does not descend to $[M_W/G_W]$ because $G_W \not\subset SL_n$ (hence $\Z$-grading does not exist).
	But its tensor power $\Omega_{M_W}^{\otimes \vert G_W \vert}$ does descend. 
	This allows us to define $\frac{1}{\vert G_W\vert } \Z$-graded version of wrapped Fukaya category following \cite{Se}, \cite{AAEKO}.
         We can also forget the fractional grading and simply work with $\Z/2$-grading as well.
  
A holomorphic disc $u:S\to \Mil_W$ defining $A_\infty$-structure can be considered as a smooth holomorphic disc 
$\overline u:S\to [M_W/G_W]$  ramified at orbifold points accordingly. Conversely, because $S$ has a vanishing orbifold fundamental group, all smooth holomorphic discs lifts to $M_W$ in a $G_W$-equivariant manner. Therefore,  
 Fukaya category of $[M_W/G_W]$ is the same as $G_W$-equivariant Fukaya category of $M_W$.


  \subsection{Monodromy flow and orbits}
  
  Fix a non-degenerate weighted homogeneous polynomial $W$ of weight $(w_1, \ldots, w_n;h)$ once and for all. Choose a slightly rescaled symplectic form $\omega$ and a Liouville form $\lambda$ on $\C^n$ (see \cite{KvK16} for the relation to the standard one) as
  \begin{equation*}\omega= \sum_k \frac{1}{2\pi iw_k}dz_k\wedge d\overline z_k, \hskip 0.2cm \lambda= \sum_{k}\frac{i}{4\pi w_k}(z_kd\overline z_k-\overline z_kdz_k).
  \end{equation*} 
  \begin{defn}
   The \textit{monodromy flow} is the Hamiltonian flow 
   \[\Phi_W(s)(x_1, \ldots, x_n):=(e^\frac{2\pi i w_1 s}{h}x_1, \ldots, e^\frac{2\pi i w_n s }{h}x_n )\] 
   of a quadratic Hamiltonian $H:=\frac{1}{2} \sum_{i=1}^n |x_i|^2$. The \textit{monodromy transformation} $\Phi_W=\Phi_W(1)$ is a time-$1$ flow
    \[x_i \mapsto e^{\frac{2\pi iw_i}{h}}x_i.\]
  \end{defn}
  
  Geometrically, a Hamiltonian action of $H$ is a lifting of a rotation action on the base of a fibration $W:\mathbb C^n \to \mathbb C$.  
We have $W\circ \Phi_{W}(s) = e^{2\pi i s}W$, which means that the flow of $H$ acts as a circle action of an $S^1$ family of Milnor fiber $W = e^{2\pi is}$. Set $s=1$ then we get a desired automorphism. 
  
  The monodromy flow restricts to a singular fiber $W^{-1}(0)$ and a \textit{link} of $W$ 
    \[L_{W, \delta} := W^{-1}(0) \cap S^{2n-1}_\delta\]
for small $\delta > 0$. The monodromy flow $\Phi_W(s)$ becomes a Reeb flow $\mathcal R$ on $L_{W, \delta}$, where the contact one form is given by a restriction of $\lambda$. Starting from any point $x\in L_{W, \delta}$, we get a Reeb chords 
  \[\gamma: [0, 1] \to L_{W, \delta}, \hskip 0.2cm \gamma(0)=x, \hskip 0.2cm \gamma(1) =\Phi_W(x)\]
  Notice that $\Phi_W$ always gives an element $J$ of a maximal symmetry group $G_W$. Reeb chords become \textbf{orbits} of a quotient $\left(L_{W, \delta}/ G_W\right)$. A space of time-$1$ Reeb orbits of the quotient is a total space $L_{W, \delta}/G_W$. 
  
  The link $L_{W, \delta}$ can be symplectically identified with $W^{-1}(1) \cap S^{2n-1}_{\delta}$ in the following way (see \cite{Sei00}). Choose a cutoff function $\psi$ with $\psi(t^2)=1$ for $t^2 \leq\frac{\delta}{3}$ and $\psi(t^2)=0$ for $t^2\geq\frac{2\delta}{3}$. Define 
  \[F := \{x\in B^{2n}_\delta \mid W(x)=\psi(|x|^2)\}.\]
  It was shown in \cite{Sei00} that $F$ is symplectic manifold with $\partial F=L_{W,\delta}$. Moreover $F$ is diffeomorphic to $W^{-1}(1) \cap B^{2n}$ by a smooth cobordism
  \[Bor_s:=\{x\in B^{2n}_\delta \mid W(x)=s\psi(|x|^2)+(1-s)\}.\]
  Shrinking $\delta$ if it is necessary, we see that $F$ and $M_W$ contain $W^{-1}(1) \cap B^{2n}_{\delta/3}$ as a Liouville submanifold and Liouville isotopy compresses both $F$ and $M_W$ to $\mathrm{Int}\left(W^{-1}(1) \cap B^{2n}_{\delta/3} \right)$. Since $\psi$ only depends on $|x|^2$, the whole construction is compatible with the action of $G_W$. Therefore $L_{W, \delta}/G_W$ serves as a model for the contact type boundary of $M_W/G_W$.
  
  \subsection{Main theorem}
 The reader would realize that we are in a similar situation to the case of Assumption \ref{as:main}. Reeb flow $\mathcal R_t$ is an $S^1$-action on $L_W/G_W$, and therefore $L_W/G_W$ provides a Morse-Bott component of Hamiltonian orbits. The orbit that corresponds to its fundamental class will denote by $\Gamma_W$. We state our main theorem.
 
\begin{thm}\label{thm:fulg}
Let $W$ be a nondegenerate weighted homogeneous polynomial $\sum_1^nw_i -h \geq0$.  
With the monodromy orbit $\Gamma_W$ of $[M_W/G_W]$, we define a new $\frac{\Z}{|G_W|}$-graded $\AI$-category $\mathcal{C}_{\Gamma_W}$ such that 
	\begin{enumerate}
	\item $\mathcal{C}_{\Gamma_W}$ has the same set of objects as the  $\frac{\Z}{|G_W|}$-graded wrapped Fukaya category $\WF([M_W/G_W])$,
	\item for two objects $L_1,L_2$, its morphisms are given by
		\[\Hom_{\mathcal{C}_\Gamma}(L_1, L_2) = CW(L_1, L_2) \oplus CW(L_1, L_2)\epsilon \]
		where $CW(L_1,L_2)$ is the morphism space for $\WF([M_W/G_W])$ and $\deg \epsilon =-1 + \mu_{RS}(\Gamma_W)$,
	\item a natural inclusion $\Psi : \WF([M_W/G_W]) \to \mathcal{C}_{\Gamma_W}$ is an $\AI$-functor,
	\item regarding the $\AI$-category $\mathcal{C}_{\Gamma_W}$ as an $\AI$-bimodule over $\WF([M_W/G_W])$ (using $\Psi$), we have a distinguished triangle of $\AI$-bimodules
		\begin{equation*} 
		\begin{tikzcd}   \mathcal{WF}([M_W/G_W])_\Delta \arrow[r, "\cap \Gamma_W "] & \mathcal{WF}([M_W/G_W])_\Delta \arrow[r] & \CG \arrow[r] & \;\; \end{tikzcd}
		\end{equation*}
	\end{enumerate}
\end{thm}
 It turns out that we still need a substantial amount of work to generalize Theorem \ref{thm:fulou} to the current setting, and
 we will give the proof in the next section.
 
 Let us just mention about the fractional grading here.
One can assign factional gradings to Hamiltonian orbits as in the case of wrapped Fukaya category following \cite{Se}, \cite{AAEKO}.
First, let us explain the underlying $\Z/2$-grading. Recall that $\Z/2$-grading for symplectic cohomology complex is easy
to define. Choose a trivialization of the tangent bundle along a
hamiltonian orbit, and the parity of the Conley-Zehnder index of the
linearized flow gives the desired grading. Note that the change of
trivialization does not change the parity of the index. Robbin-Salamon
index of a symplectic loop under any trivialization is even. Thus  one find that  $\Z/2$-grading
for $\Gamma_W$ is still even, since under the trivialization of the
tangent bundle over a principal orbit, we still have a symplectic
loop.

Now, additional fractional grading can be assigned as usual.  Recall that we have non-vanishing form $\Omega_{M_W}^{\otimes \vert G_W \vert}$ on $[M_W/G_W]$,
 which can be used to define fractional Conley-Zehnder or Robbin-Salamon indices 
 (see \cite{Mc}).
 In particular, for $\Gamma_W$, Robbin-Salamon index of
 the orbit with respect to the form $\Omega_{M_W}^{\otimes \vert G_W \vert}$ 
is denoted as $\mu_{RS}^{\Omega^{\otimes \vert G_W\vert}}$, and we set
\[\mu_{RS}(\Gamma_W) =  \frac{1}{\vert G_W\vert}\mu_{RS}^{\Omega^{\otimes \vert G_W\vert}}\]

%
%
%

\subsection{Fukaya category $\cF(W, G)$ for any  subgroup $\boldsymbol{G < G_W}$}\label{sec:ggw}
\begin{defn}
	The  $\Z/2$-graded Fukaya category $\cF(W, G_W)$ of a Landau-Ginzburg orbifold $(W, G_W)$ is defined to be the underlying
$\Z/2$-graded $\AI$-category of $\mathcal C_{\Gamma_W}$.
\end{defn}
For any subgroup  $G < G_W$, we define the Fukaya category $\mathcal{F}(W,G)$ using the dual group $G_W^T$-action on $\cF(W, G_W)$. 

Let us first explain the algebraic setting.
Let $\mathcal{C}$ be an $\AI$-category with an action by finite abelian group $H$. By taking a $H$-invariant part,
we may obtain an $\AI$-category $\mathcal{C}^H$   whose
object is a $H$-family of objects $\{ h \cdot \tilde K\}_{h \in H}$, denoted as $K$,
and its morphisms between $K$ and $K'$ are given by 
$$\bigoplus_{h \in H} \Hom_\mathcal{C}(K, hK') \cong \left( \bigoplus_{h_1,h_2\in H} \Hom_\mathcal{C}(h_1\tilde K, h_2\tilde K')\right)^H.$$
Then the character group $H^* = \Hom(H, \C^*)$ naturally acts on the $\AI$-category $\mathcal{C}^H$:
 we define $\chi \in H^*$ action for an element $X \in \Hom_\mathcal{C}(K, hK')$ (resp. its dual formal variable $x$) as 
$$\chi \cdot X = \chi(h^{-1}) X  \;\;\;\; (\textrm{resp.} \; \chi \cdot x = \chi(h) x ).$$

In \cite{Se2}, Seidel observed that one can recover $\mathcal{C}$ from $\mathcal{C}^H$ using semi-direct product.
By taking a tensor product of $\Hom_{\mathcal{C}^H}(K,K')$ with the group ring $\C[H^*]$, and
extending the $\AI$-operation suitably, we obtain a semi-direct product $\AI$-category $\mathcal{C}^H \rtimes H^*$,
which can be canonically identified with $\mathcal{C}$.
More precisely, the component $\Hom_{\mathcal{C}^H}(K,K') \otimes \chi$ in the morphism space can be identified with the $\chi$-eigenspace of $H$-action on $\mathcal{C}$ (see \cite{CL_ks} for an explicit identification). 
Geometrically, it means that $\AI$-operations in $\mathcal{C}^H$ actually come from $\mathcal{C}$.
For the case of Fukaya category, this corresponds to the fact that a smooth holomorphic discs in the quotient can be always lifted.

Let us discuss the case of the $\AI$-category $\cF(W, G_W)$.
For a $J$-holomorphic disc without $\Gamma_W$-insertion, 
$$m_k(\chi \cdot w_1,\ldots, \chi \cdot w_k)= \chi \cdot m_k(w_1,\ldots,w_k)$$
holds because disc has a lift to upstairs and still is a disc.  
When we walk along the boundary Lagrangians, we come back to the original branch that we start with.

We have additional operations of popsicles which have $\Gamma_W$ interior insertions.
These $J$-holomorphic discs do not have lifts to upstairs, so we may proceed as follows. 
Denote by $g_W \in G_W$ be the group element for the monodromy of $W$ in \eqref{eq:gm}.
 For each sprinkle with $\Gamma_W$ insertion, we can make a branch cut in the domain $D^2$ to the boundary arc between two
marked points $z_0,z_1$. We can make the cuts disjoint from each other. Then, the resulting disc lifts to the cover, but
if we restrict to the arc $z_0z_1$, it has discontinuous lift.  Namely, $z_0z_1$ boundary arc has as many $g_W$ jumps as the number of sprinkles. This can be used to show that 
$$m_{k,F}^{\Gamma_W} (\chi \cdot w_1,\ldots, \chi \cdot w_k) = \chi (g_W^{-|F|}) m_{k,F}^{\Gamma_W} (w_1,\ldots,w_k)$$
To accommodate this new factor, we introduce $G_W^*$-action on the formal parameter $\epsilon$ in the Definition \ref{defn:newa}. It is not hard to check that
this gives the desired $G_W^*$-action on $\mathcal{C}_{\Gamma_W}$.
\begin{prop}
We set $\chi \cdot \epsilon = \chi(g_W^{-1}) \epsilon$.
Then $\AI$-category $\mathcal{C}_{\Gamma_W}$ admits $G_W^*$-action.
\end{prop}
For any subgroup $G <G_W$, we define its dual group  $G^T = \Hom(G_W/G, \C^*)$ following Berglund-Henningson \cite{BH95}.
We are ready to define a Fukaya category of the Landau-Ginzburg orbifold $(W,G)$.
\begin{defn}\label{defn:gf}
 We define  a Fukaya $\AI$-category for $(W,G)$ by
$$\mathcal{F}(W,G) :=  \mathcal{C}_{\Gamma_W} \rtimes G^T.$$
Here, the quotient map $G_W \to G_W/G$ induces the map $G^T \to G_W^*$ and hence a $G^T$-action on $\mathcal{C}_{\Gamma_W}$.
\end{defn}

\section{Twisted Reeb orbits in the Milnor fiber quotients and their indices}\label{sec: index computation}
We give a proof of Theorem \ref{thm:fulg} in this section and the next, generalizing the construction of Section \ref{sec:qc1}.
The main scheme is the same, but our orbifold $[M_W/G_W]$ does not satisfy Assumption \ref{as:main}, and hence we have the following issues.
	\begin{enumerate}
	\item $[M_W/G_W]$ have a nontrivial $c_1$.
	\item The action of Reeb flow $\mathcal R_t$ is not free. We need a classification of Morse-Bott components. 
	\item Inequality relating a period and CZ index is required. 
	\item We need to define $\Gamma_W$ and show that it is closed in a suitable sense. 
	\end{enumerate}
We will address each of these items, and then Theorem \ref{thm:fulg} should follow as in construction of Section \ref{sec:qc1}.

The first item can be overcome by introducing the fractional grading as we explained. Let us discuss the remaining issues one by one.
\subsection{Classifying twisted Reeb orbits} 
The second issue is that there can be several Morse-Bott components of Hamiltonian orbits in $[M_W/G_W]$. 
We investigate such orbits and their fractional gradings in this subsection.


In the construction of $\mathcal C_{\Gamma_W}$, we only use the principal orbit $\Gamma_W$. Even though we do not use other Reeb orbits as interior insertions, they might appear in the compactification of popsicles. It is well-known that a perturbed $J$-holomorphic cylinder with bounded energy in the smooth locus of the orbifold can limit to a periodic orbit in the orbifold locus.  
And the purpose of the study of these non-principal orbits is to rule out their appearances in the compactification of our popsicles.

A loop in the orbifold $[M_W/G_W]$ is given by a $G_W-$twisted loop in $M_W$. 
\begin{defn}
For an element $g \in G_W$, a {\em a $g$-twisted loop} $(g, \gamma)$ of period $l$ is a path $\gamma : [0,l] \to M_W$ such that $g\cdot \gamma(l) = \gamma(0)$. 
We have a $G_W$-action on the space of twisted loops given by 
$$k\cdot(g, \gamma) = (kgk^{-1}, k\gamma),\;\; k\in G_W.$$
\end{defn}

Now, let us restrict ourselves to twisted Reeb orbits on the link of the singularity $W$.
Let us first introduce some notations. 
\begin{defn}
	For each element $g  \in G_W$, we set
	\[\vec{\theta}_g = \mathrm{diag}(\theta_{g,1}, \ldots, \theta_{g,n}), \;\;  \theta_{g, i} \in [0, 1)\;\;  \forall i.\]
	so that  we have 
	\[g=\mathrm{exp}(2\pi i \vec{\theta}_g):= \big(\mathrm{exp}(2\pi i \theta_{g,1}), \ldots, \mathrm{exp}(2\pi i \theta_{g,n})\big).\]
	In particular, for a given system of weights $(w_1,\ldots,w_n;h)$ of $W$, we have
	\[\vec{\theta}_J = (w_1/h, \ldots, w_n/h), \hspace{0.2cm} J = \mathrm{exp}(2\pi i \vec{\theta}_J) \in G_W. \]
	Notice that Reeb flow $\mathcal R_t$ is given by $\mathrm{exp}(2\pi i t \vec{\theta}_J)$ and $J$ is just $\mathcal R_1$.
\end{defn}

\begin{defn}
We denote by $(\C^n)^{g, l}$ the space of $g$-twisted Reeb orbits of period $l$ in $\C^n$. We also denote by $\Sigma_{g, l}$ the space of $g$-twisted Reeb orbits in the link of $W$:  
	\[\Sigma_{g,l} = (\C^n)^{g, l} \cap L_W.\]
\end{defn}

Notice that $p \in (\C^n)^{g, l}$ precisely when $ g \circ \mathcal R_l(p) =p$, and hence  $(\C^n)^{g, l}$ is a fixed locus of $ g \circ \mathcal R_l$-action on $\C^n$.  Equivalently, we have
	\begin{equation}\label{glz}
	(\C^n)^{g, l} :=\left\{ (z_1, \ldots, z_n) \in \C^n : \hspace{0.1cm} \textrm{if $z_i \neq0 $, then} \hspace{0.1cm} \theta_{g, i} + l \theta_{J, i}=0 \hspace{0.1cm} \mathrm{mod} \Z  \right\}
	\end{equation}
A set $\left\{\Sigma_{g, l}\right\}_{g\in G_W, l\in \R_{>0}}$ classifies all possible Morse-Bott components of twisted Reeb orbits.
\begin{lemma}
We have
	\[\Sigma_{g, l} \simeq \Sigma_{gJ^{-1}, l+1}, \;\; \Sigma_{J^{-1}, 1} \simeq L_W.\]
\end{lemma}
The lemma follows immediately from $\mathcal R_1 = J$. One can also easily check that there are finitely many diffeomorphism types of $\Sigma_{g, l}$.

\begin{example}(Principal component) 
	Since $\mathcal R_1 = J$, every point in $L_W$ arises $J^{-1}$-twisted Reeb orbit of period $1$, and hence, $\Sigma_{J^{-1}, 1} \simeq L_W$. 
	This gives the principal component of the $J^{-1}$-twisted Reeb orbits in $[M_W/G_W]$ of period 1. 
	The orbit $\Gamma_W$, a fundamental class of of this component, is called principal orbits. 
	We have $J^{-l}$-twisted Reeb orbits of period $l$ in $[M_W/G_W]$ for any $l\in \mathbb N$ in a similar manner.  
\end{example}

We may ask whether the period of a twisted Reeb orbit is always an integer as we saw in the previous section.
The answer is no in general, and this complicates the story quite a bit.
Namely, we can have Reeb orbits of shorter periods (starting at $p$) if some of the coordinates (of $p$) vanish.
Here is one example of a fractional period of $l = \frac{1}{3}$.

\begin{example}(A component with a fractional period)
\label{ex1}
 Consider $W= x^2y +y^4$. Its weight  is $(3,2;8)$, and  $\vec{\theta}_J = (\frac{3}{8}, \frac{1}{4})$. 
		For $(g, l)= (J^{-3}, \frac{1}{3})$, we have  $(\C^2)^{g, l} = \C_x \times \{0\}$. Indeed, 
		 $$J^{-3} \circ \mathcal R_{\frac{1}{3}}=(e^{\frac{7\pi i}{4}}, e^{\frac{\pi i}{2}}) \circ (e^{\frac{\pi i}{4}}, e^{\frac{\pi i}{6}})=(1, e^{\frac{2\pi i}{3}}).$$ 
Note that $\C \times \{0\}$ does not arise as a fixed locus for any $g \in G_W$, and moreover
$W \big\vert_{\C \times \{0\}} = 0$. 
\end{example}

Let us comment on the relationship between twisted sectors and twisted Reeb orbits, although we do not use it in this paper.
As $G_W$ is abelian, a twisted sector of the quotient orbifold $[M_W/G_W]$ for an element $k \in G_W$ is
given by $[(M_W)^k/G_W]$ where $(M_W)^k$ is the $k$-fixed locus. On the other hand, as $k$ acts linearly on $\C^n$
we can consider the restriction of $W$ on the $k$-fixed subspace $\textrm{Fix}(k) \subset \C^n$ to obtain $W^k : \textrm{Fix}(k)\to \C$.
One can easily check that $(M_W)^k$  can be identified with $M_{W^k}$.

 As $\mathcal R_t$ commutes with $G_W$-action,  the canonical Reeb flow on the link $L_{W^k}$
is the restriction of that of $L_W$. In particular, if a point of the Reeb flow $\gamma$ is fixed by $k$, then the whole trajectory is fixed by $k$.
A twisted Reeb orbit $(g,\gamma)$ gives an element of the chain complex $SC^\bullet([M_W/G_W])$,
but if (a point of ) it is fixed by some $k$, then it also gives rise to an element of $SC^\bullet([M_{W^k}/G_W])$ as well.


	\begin{remark}
	Recall that $W$ raises a canonical exact sequence of symmetry groups
	\[
	\begin{tikzcd}
	0 \arrow[r]& G_W \arrow[r] & \widetilde{G}_W \arrow[r, "\chi"] &\C^*_R \arrow[r] & 0 
	\end{tikzcd}
	\]
	and $\widetilde{G}_W$ acts naturally on $W^{-1}(0)$. Notice that Morse-Bott components $\{\Sigma_{g, l}\}$ corresponds to twisted sectors of $[(W^{-1}(0)\setminus \{0\})/ \widetilde{G}_W]$. 
	\end{remark}

\subsubsection{Conley-Zehnder indices of twisted orbits} 
We have identified generators of the  chain complex $SC^*([M_W/ G_W])$.
We compute their Robbin-Salamon indices.
\begin{remark}
We do not know how to define the differential and products in our setting of perturbation scheme (which follows Abouzaid-Seidel), but we do not need them in this paper (except $m_1$ of $\Gamma_W$).
\end{remark}

We start with the following elementary observations.
\begin{lemma}\label{lem:m2}
Consider a weighted homogeneous polynomial $W$, and its maximal diagonal symmetry group $G_W$.
For each $g \in G_W$, we define $F_g \subset \{1,\ldots,n\}$ so that $$  i \in F_g  \;\textrm{if and only if } \;\; g \cdot x_i = x_i.$$
We write 
$$W = W_{F_g} + W_{\textrm{move},g}$$
where $W_{F_g}$ is the sum of monomials of $W$ that consists of variables whoses indices are in $F_g$,
and $W_{\textrm{move},g}$ is the sum of remaining monomials.
Let $\mathbf{m}$ be the ideal generated by $x_i$ with $i \notin F_g$.
Then we have 
$$W_{\textrm{move},g}   \in \mathbf{m}^2$$
\end{lemma}
\begin{proof}
Note that $g$-action leaves $W_{F_g}$ to be invariant. Since $g$-action leaves $W$ to be invariant, the same holds for
each monomial of $W_{\textrm{move},g}$.   Note that $W_{\textrm{move},g}  \in \mathbf{m}$. If a monomial of it does not lie in $ \mathbf{m}^2$,
then it cannot be invariant under $g$-action.
\end{proof}

\begin{lemma}\label{lem:gindex}
We have the following identities of partial derivatives.
\begin{enumerate}
\item For $g \in G_W$, we have
$$ \partial_iW(x) = \mathrm{exp}( 2\pi i  \theta_{g,i}) \partial_iW(gx) $$
\item We have 
\begin{eqnarray}
W(R_t(x)) &=& \mathrm{exp}(2\pi i t) W(x) \\
\partial_iW(x) &=&  \mathrm{exp}\big( 2\pi i l(-1 + \theta_{J,i})\big) \partial_iW(R_l(x))
\end{eqnarray}
\end{enumerate}
\end{lemma}
 
To compute the Robbin-Salamon index of a twisted Reeb orbit $(g,\gamma)$, we first find a good frame for the tangent bundle of the Milnor fiber $M_W$ with respect to $g$-action.

\begin{defn} 
We define another index set $I_{g,l} \subset  \{1, \ldots, n\}$
$$I_{g,l} =\left\{i \colon \partial_iW\vert_{(\C)^{g, l}} = 0 \right\}$$
We will simply write it as $I$ sometimes.
\end{defn}
\begin{prop}\label{prop:frame}
Let $\gamma$ be a $g$-twisted loop of period $l$.
The following frame (for a fixed $j_0 \in I^c$) trivializes $T(M_W)|_{\gamma}$: 
 	\[\left\{ e^{-2\pi i \theta_{g, i}t/l} e_i \colon i\in I \right\} \cup \left\{ \frac{e_j}{\partial_jW}-\frac{e_{j_0}}{\partial_{j_0} W} \colon j  \in I^c \right \}.\]
	 \end{prop}
\begin{proof}
From the definition of $I$, we have
$$dW\vert_{(\C^n)^{g, l}} = \sum_{i \in I^c} \partial_iW dz_i.$$
Thus it is clear that the above vectors lie in the Kernel of $dW$.
It remains to check that it is compatible with $g$-twisting $g \cdot \gamma(l) = \gamma(0)$.
It is clear for the basis vectors of the first type.


 For the second type, from the above lemma, we have
\[\left(  \partial_jW (\gamma(l)) \right) = e^{2\pi i \theta_{g,j}}  \left(  \partial_jW (g \gamma(l)) \right) =  e^{2\pi i \theta_{g,j}}  \left(  \partial_jW (\gamma(0)) \right)\]
Therefore, we obtain the claim as follows.
\[ e^{2\pi i \theta_{g,j}} \cdot \frac{e_j}{\partial_jW (\gamma(l)) } = \frac{e^{2\pi i \theta_{g,j}} e_j}{ e^{2\pi i \theta_{g,j}} \partial_jW (\gamma(0))} = \frac{e_j}{\partial_jW (\gamma(0)) }\]
\end{proof}

Now, we measure how much the frame rotates with respect to the canonical volume form 
 $\Omega_{M_W}$.
\begin{lemma}\label{lem:rf}
 The rotation number (along $\gamma$) of the frame  in Proposition \ref{prop:frame} w.r.t.  $\Omega_{M_W}$ is
$$ \left(-\sum _{i \in I} \theta_{g, i} + \sum_{j \in I^c} (l\theta_{J, j}-l)\right) $$
\end{lemma}
\begin{proof}
Given the frame,  $\partial_{z_{j_0}} W \neq 0$ for the chosen $j_0$ in $I^c$.  (If $I^c$ is empty, then choose any $j_0$ 
with such property). 
Hence we have
\begin{equation}\label{eq:omegaw}
 \Omega^n_{\C^n}/dW = \frac{ dz_1 \cdots \widehat{dz_{j_0}} \cdots dz_n}{\partial_{z_{j_0}}W}
 \end{equation}
By evaluating the basis vectors of the first type (for $i \in I$), we obtain the rotation number $- \sum _{i \in I} \theta_{g, i}$.
For the basis vectors of the second type (for $i \in I^c$), note that if we evaluate such vectors to the form $\Omega_{M_W}$, we obtain
the factor $$\prod_{j \in I^c} \frac{1}{\partial_{z_{j}}W} $$
where the contribution of $j_0 \in I^c$ come from the coefficient of  \eqref{eq:omegaw}.
From the lemma \ref{lem:gindex}, the rotation number of such a factor is 
$$\sum_{j \in I^c} l (\theta_{J,j} - 1)$$
\end{proof}

To compute the Maslov index of the linearized Reeb flow, we write the linearized Reeb flow  in terms of this frame: for $t \in [0,l]$, 
\[\left\{ e^{2\pi i ( \theta_{J, i}+\theta_{g, i}/l)t}  \big( e^{-2\pi i \theta_{g, i}t/l} e_i \big) \colon i\in I \right\} \cup \left\{ e^{2\pi i   t}\left(\frac{e_j}{\partial_jW}-\frac{e_{j_0}}{\partial_{j_0} W}\right) \colon j, j_0 \in I^c \right \}\]
\begin{lemma}\label{lem:RS}
Robbin-Salamon index of the linearized Reeb flow w.r.t. the frame in Proposition \ref{prop:frame} is
 	\begin{align}
	&\sum _{i\in I} Q(l \theta_{J, i} + \theta_{g, i}) + \;\; (\vert I^c\vert -1) Q(l), \\
	Q(s) &= \left\{
		\begin{array}{ll}
		2s & s\in  \mathbb Z\\
		2\lfloor s \rfloor +1 & s \notin \mathbb Z
		\end{array} 
		\right.
		\end{align}
 \end{lemma} 
 
 Combining Lemma \ref{lem:rf}, \ref{lem:RS}, we obtain the following
 \begin{prop}\label{prop:RS}
 Robbin-Salamon index of the linearized Reeb flow in $\Sigma_{g,l}$ is
 $$ \mu_{\mathrm{RS}}(\Sigma_{g, l})  = 2\left(-\sum _{i \in I} \theta_{g, i} + \sum_{j \in I^c} (l\theta_{J, j}-l)\right) 
	+\sum _{i\in I}Q(l \theta_{J, i} + \theta_{g, i}) + (\vert I^c\vert -1) Q(l) $$
\end{prop}	
\begin{remark}
$\mu_{\mathrm{RS}}(\Sigma_{g, l})$ always lies in $\frac{1}{\vert G_W\vert } \Z$.
\end{remark}

Now, we rewrite the above expression of $\mu_{\mathrm{RS}}(\Sigma_{g, l})$ as 	
\begin{equation}\label{eq:rsmod}
	 \mu_{\mathrm{RS}}(\Sigma_{g, l})  = 2l ( \sum_{i =1}^n \theta_{J, i} -1) + \sum _{i\in I}\left( Q(l\theta_{J, i}+\theta_{g, i}) - 2(l\theta_{J, i}+\theta_{g, i})\right) + (\vert I^c\vert -1) (Q(l)-2l)
\end{equation}
	
Let us illustrate our formula in some examples.
\begin{example}\label{ex: index}
	\begin{enumerate}
		\item For  $g=J^{-l}$,we have $(\C^n)^{g, l} = \C^n$ and $I = \emptyset$.
		\[ \mu_{\mathrm{RS}}(\Sigma_{J^{-l}, l}) = 2l (\sum \theta_{J, i} -1) = \frac{2l(\sum_{k=1}^n w_i -h)}{h}.\]
		
		\item  For the example \ref{ex1} of $W= x^2y +y^4$ with  $(g, l)= (J^{-3}, \frac{1}{3})$.  We have 
$$\vec{\theta}_g = \left(\frac{7}{8}, \frac{1}{4}\right), \vec{\theta}_J = \left(\frac{3}{8}, \frac{1}{4}\right).$$
On $\C \times \{0\}$, $\partial_1 W = 2xy = 0$. Thus $I = \{1\}, I^{c} =\{2\}, l=\frac{1}{3}$.
$$ \mu_{\mathrm{RS}}(\Sigma_{g, l}) =  2 \left( - \frac{7}{8} + \frac{1}{3}\left(\frac{1}{4}-1\right)\right) + Q\left(\frac{1}{3} \cdot \frac{3}{8} + \frac{7}{8}\right)+0
= -2+2= 0.$$
 
	\end{enumerate}
	
\end{example}

\subsection{Degree shifting number and its non-negativity}

The following number $\star_{g, l} \in \mathbb{Q}$ will be used for degree shifting.
\begin{defn}\label{defn:star}
We define $\star_{g, l} \in \mathbb{Q}$ as 
$$(n-1)-\mu_{\mathrm{RS}}(\Sigma_{g, l})-\frac{1}{2}(\mathrm{dim}_\R \Sigma_{g,l} +1)$$
It is easy to see that $\star_{g,l} = n - \mathrm{dim}_\C (\C^n)^{g, l} - \mu_{\mathrm{RS}}(\Sigma_{g, l})$
\end{defn}

\begin{prop}
\label{orbifold degree inequality 1}
We have the following inequality.
\[\star_{g, l} \geq -\frac{2l(\sum_{k=1}^n w_i -h)}{h}\]
\end{prop}
In particular, $\star_{g,l} \geq 0$ for log Fano and log Calabi-Yau cases.
\begin{proof}
\begin{enumerate}
\item  Suppose $W\vert_{(\C^n)^{g, l}} \neq 0$.
	
	 There is a ${j_0}$ such that $j_0 \in I^c$.
	Since $\partial_{j_0} W \neq 0$ on $(\C^n)^{g, l}$,  from the Lemma \ref{lem:gindex},
we obtain the condition 
 $$l\theta_{J, j_0} + \theta_{g, j_0} \equiv l \;\; \textrm{modulo} \; \Z.$$
Also, as $z_{j_0}$ is fixed under the action of $g\circ \mathcal R_l$, we obtain another condition
 $$l\theta_{J, j_0} + \theta_{g, j_0} \equiv 0 \;\; \textrm{modulo} \; \Z.$$
Therefore, $l \in \mathbb N$ and $g\circ \mathcal R_l = g\circ J^l \in G_W$. Hence $(\C^n)^{g, l}$ is a fixed locus of $gJ^l$-action.

Recall from Lemma \ref{lem:m2} that 
$$W = W_{F_{gJ^l}} + W_{move,gJ^l}$$
and  $W_{move,gJ^l} \in \mathbf{m}^2$.
Hence, if $j \notin F_{gJ^l}$,  $\partial_j W_{move,gJ^l}$ still contains at least one moving variable hence vanishes on $(\C^n)^{g, l}$.
Therefore, $$F_{gJ^l}^c \subset I_{g,l}.$$
Since $W$ is a non-degenerate polynomial, we also have $F_{gJ^l} \subset I_{g,l}^c$, which implies that $I_{g,l}^c= F_{gJ^l}$.

	Therefore, using the identity \ref{eq:rsmod}, we have
	\begin{align}
	\star_{g, l}&= n - \mathrm{dim}_\C (\C^n)^{g, l} - \mu_{\mathrm{RS}}(\Sigma_{g, l}) = \vert I_{g,l} \vert -  \mu_{\mathrm{RS}}(\Sigma_{g, l})\\
	& = - 2l ( \sum_{i =1}^n \theta_{J, i} -1) + \sum _{i\in I_{g,l}}\left(1- ( Q(l\theta_{J, i}+\theta_{g, i}) - 2(l\theta_{J, i}+\theta_{g, i})) \right) - (\vert I_{g,l}^c\vert -1) (Q(l)-2l)\\
	&  \geq - 2l ( \sum_{i =1}^n \theta_{J, i} -1)
	\end{align}
	Here, we use the fact that $Q(s) -2s<1$ for all $t$ and $Q(l)-2l =0$ as $l \in \Z$. 
	
\item	Suppose $W\vert_{(\C^n)^{g, l}} = 0$.
  If $i \in I_{g,l}^c$, then $\partial_i W \vert_{(\C^n)^{g, l}} \neq 0$, but the latter vanishes from the assumption.
Hence, $I_{g,l}^c = \emptyset$ and $I_{g,l} = \{1,\cdots, n\}$.
Consider the action of $g \circ R^l$ on $(z_1,\cdots,z_n)$ and say $i \in K \subset \{1,\cdots,n\}$ if and only if $z_i$ is fixed under this action.
Then $ \mathrm{dim}_\C (\C^n)^{g, l} = \vert K \vert $. Also, we have $l\theta_{J, i}+\theta_{g, i} \in \Z$ for $i \in K$.
Therefore, we have
	\begin{align}
	\star_{g, l}&=  n -\mathrm{dim}_\C (\C^n)^{g, l} - \mu_{\mathrm{RS}}(\Sigma_{g, l}) \\
& = n -|K|  - 2l ( \sum_{i =1}^n \theta_{J, i} -1) - \sum_{i=1 }^n \left(  Q(l\theta_{J, i}+\theta_{g, i}) - 2(l\theta_{J, i}+\theta_{g, i}) \right) + (Q(l)-2l)  
\\ 
& =  - 2l ( \sum_{i =1}^n \theta_{J, i} -1) + \sum_{ i \notin K}  \left(  1 - ( Q(l\theta_{J, i}+\theta_{g, i}) - 2(l\theta_{J, i}+\theta_{g, i})) \right) 
+ (Q(l) - 2l) \\
	&  \geq - 2l ( \sum_{i =1}^n \theta_{J, i} -1)
\end{align}
This completes the proof.
\end{enumerate}
 \end{proof}

\section{Monodromy orbit $\Gamma_W$ and vanishing of spheres with $\Gamma_W$}\label{sec:proof2}
\subsection{Morse-Smale function for perturbation}
Let us first define a Morse-Smale function on  $\left(L_{W, \delta}/G_W \right)$ that is needed for perturbations of Morse-Bott components.
 Although Morse-Smale function might not exist for an orbifold, our situation is rather special so that we can construct one as follows.
 
  \begin{lemma}\label{equivariant Morse function}
    There is an $G_W$-equivariant Morse function $h$ on $L_{W, \delta}$ such that 
    \begin{itemize}
    \item if a critical point $p \in \mathrm{crit}(h)$ lies in $\mathrm {Fix}(K)$ for some $K<G_W$, then the unstable manifold $W^- (p)$ of $p$ is also contained in $\mathrm{Fix}(K)$. 
    \item its Morse-Smale-Witten complex is well-defined and computes the cohomology of $\left(L_{W, \delta}/G_W \right)$.
    
  \end{itemize}  
  \end{lemma} 
  \begin{proof}
   Since orbifold strata $\{\textrm{Fix}(K)\}_{K<G_W}$ of $\C^n$ intersect transversally to the link, it induces orbifold strata for $L_{W, \delta}$. We construct a Morse function $h'$ in an inductive way. The action of a diagonal symmetry group is effective on the link. Therefore the deepest strata does not have a nontrivial fixed locus. Choose a Morse function on it and extend it strata by strata in a way that it depends on the radial direction of the normal bundle of the strata inside the next one. We may choose our extension so that the negative Morse flow line of our extension flows into the lower strata.  This is possible because our orbifold strata are induced from coordinate planes. When several stratum intersect along a lower dimensional one, their normal directions are compatible.
   
   Morse flow of $h'$ respects orbifold strata. If $p\in \mathrm{crit}(h')\cap \mathrm{Fix}(K)$, then its unstable manifold $W^-(p)$ is contained in $\mathrm{Fix}(K)$. If $q \in \mathrm{crit}(h') \cap \mathrm{Fix}(K')$ for some $K\leq K'$, then its stable manifold $W^+(q)$ intersect $W^-(p)$ at smooth points of $\textrm{Fix}(K)$ only. Therefore, the moduli space of gradient flow trajectories can be defined as in the smooth case. They intersect transversally after a small perturbation of $h'$ in the smooth part of $\mathrm{Fix}(K)$. We do it strata by strata to obtain $h$. 

   Because of this first property, Morse-Witten complex of $h$ makes sense. i.e, its differential does square  to zero. Moreover, the unstable manifolds of $h$ provide a cell decomposition of $\left(L_{W, \delta}/G_W \right)$. Hence Morse-Witten complex computes the singular cohomology of $\left(L_{W, \delta}/G_W \right)$ (for a similar argument for $G=S^1$ in the context of contact homology, see \cite{Bourgeois03}).

  \end{proof}

\subsection{A symplectic cohomology cocycle $\Gamma_W$}
Finally, let  us define $\Gamma_W$.

We adapt the idea of \cite{CFHW96} and \cite{KvK16}. Choose a small normal neighborhood $\nu\left(\Sigma_{J^{-1}, 1}/G_W\right) \simeq L_{W, \delta}/G_W \times (-\epsilon, \epsilon)$ inside $M_W/G_W$. Choose a sufficiently small, nonnegative bump function $\rho$ supported inside $(-\epsilon, \epsilon)$. Use our Morse function $h = h_{J^{-1}, 1}$ in \ref{equivariant Morse function} to define an equivariant time-dependent perturbation term by 
  \begin{align*}
    \OL{h}: \nu\left(L_{W, \delta}/G_W\right) \times S^1 &\to \R\\
    (p, v, t) &\mapsto h(Fl^\mathcal R_t(p))\rho(v)
  \end{align*} 
  \begin{defn}
  A \textit{local Floer cochain complex} 
    \[CF^\bullet_{loc}\left(L_{W, \delta}/G_W, H_{S^1}\right)=\bigoplus_{\gamma\in \mathcal O(H_{S^1})} o_\gamma\] 
  is defined to be a free module generated by Hamiltonian orbits of time-dependent Hamiltonian $H_{S^1}:= H+\OL h$. Its differential is defined similarly as a counting of pseudo-holomorphic cylinders, but only those inside $\nu\left(L_{W, \delta}/G_W\right)$. Its cohomology, called \textit{local Floer cohomology}, is denoted by $HF^\bullet_{loc}\left(L_{W, \delta}/G_W, H_{S^1} \right)$. 
  \end{defn}
  \begin{thm} As a $\Z/2$-graded vector space, 
    \[HF^{\bullet}_{loc}\left(L_{W, \delta}/G_W, H_{S^1}\right) \simeq H^{\bullet}_{Morse}(L_{W, \delta}/G_W, h; \Z)\]
  \end{thm}
  \begin{proof} See \cite{KvK16}, Proposition 8.4. One can check that $L_{W, \delta}/G_W$ satisfies all the conditions in the proposition except the first Chern class condition, which makes the isomorphism only $\Z/2$-graded. We can also check that flows and Hamiltonians in the proof are equivariant. 
  \end{proof}

  \begin{defn}\label{defn:monoW}
 The Hamiltonian orbit $\Gamma_W \in CF^\bullet_{loc}(L_{W, \delta}/G_W, H)$ is defined to be a cochain which corresponds to a fundamental class
    \[\Gamma_W \leftrightarrow \left[ L_{W, \delta}/G_W \right]\in H^\bullet \left(L_{W, \delta}/G_W;\Z\right).\]
 It will be shown that $\Gamma_W$ is indeed a cocycle. We call $\Gamma_W$ the {\em monodromy orbit}.
  \end{defn}

%
%
  
In Lemma \ref{gamma is closed}, we prove  that there are no rigid Floer cylinder emanating from $\Gamma_W$ other then the ones in local Floer complex of $\Sigma_{J^{-1}, 1}/G_W$. Although we do not define full-fledged orbifold symplectic cohomology complex, it is legitimate to say $\Gamma_W$ is closed is a Floer theoretic sense.

We also apply perturbations of $\Sigma_{g, l}$ using Morse functions $h_{g, l}$ on each Morse-Bott components. Such perturbation is well-defined since $h_{g,l}$ are $G_W$-equivariant.  Also, there are only finitely many diffeomorphism types of $\Sigma_{g, l}$. Therefore, we can assume that there is an $\epsilon>0$ such that
	\begin{enumerate}
	\item $\mathrm{max} h_{g,l} - \mathrm{min}h_{g, l} \leq \epsilon$ for $\forall g\in G_W, l\in \R_{\geq 0}$, 
	\item if $\gamma$ arises as a critical point of $h_{g, l}$, then $\mathcal A_{H_{S^1}}(\gamma) \in \left[-l^2 -\epsilon, -l^2+\epsilon\right]$, 
	\item if $\Sigma_{g_1, l_1}$ and $\Sigma_{g_2, l_2}$ are nonempty and $l_1\neq l_2$, then $[-l_1^2 -\epsilon, -l_1^2 +\epsilon] \cap [-l_2^2 -\epsilon, -l^2_2+\epsilon] = \emptyset$. 
	\end{enumerate}
Define a filtration $\cF^{-l}:= \left\{ \gamma : \mathcal A_{H_{S^1}} (\gamma) \geq -l^2-\epsilon \right\}$. Then similar to \ref{energy spectral sequence}, we get 

	\begin{lemma}\label{lem:filgw}
	An associated $\mathbb Q\times \mathbb Q$-graded of the filtration $\cF^{-\bullet}$ is 
	\[E^{pq}_0 = 
	\left \{ 
	\begin{array}{ll}
	\bigoplus_{g\in G_W} C^{q-2\mathrm{age}(g)} \left( (M_W)^g/G_W \right) & (p=0) \\
	\bigoplus_{g \in G_W} C^{p+q-\star_{g, -p}}(\Sigma_{g, -p}/G_W) & (p <0) \\
	0 & (p>0)
	\end{array}
	\right.\]	
	\end{lemma}
\begin{proof}
The cases of $p=0$ correspond to the constant orbits.
In Example 3.7 of \cite{GZ}, Gironella and Zhou explained that the relation between the Conley-Zehnder indices of constant orbits and
age of the corresponding Chen-Ruan cohomology class, which explains the degree shifting for the case of $p=0$.
For $p <0$, $\star$ was defined in Definition \ref{defn:star} to denote the degree shifting number, and hence the lemma follows.
\end{proof}

  \begin{cor}\label{cor:gin}
  $\deg(\Gamma_W)= -\mu_{\mathrm{RS}}(\Gamma_W)= -\frac{2((\sum_{i=1}^n w_i) -h)}{h}$
  \end{cor}
  \begin{proof}
  Since $\Gamma_W$ corresponds to a fundamental class of a Morse-Bott family $\Sigma_{J^{-1}, 1}$, we
  get the second inequality from Proposition \ref{prop:RS} as computed in Example \ref{ex: index}.
  It can be used to compute the degree of $\Gamma_W$ as follows.
  Note that the degree shifting number   $\star_{J^{-1}, 1} \in \mathbb{Q}$ is given by 
$$(n-1)-\mu_{\mathrm{RS}}(\Sigma_{J^{-1},l})-\frac{1}{2}(\mathrm{dim}_\R \Sigma_{J^{-1},1} +1) = -\mu_{\mathrm{RS}}(\Sigma_{J^{-1},l})$$
As $\Gamma_W$ is from the fundamental class of the principal orbits, $ \deg(\Gamma_W)- \star_{J^{-1}, 1} =0$, which implies the claim.
  \end{proof}

The construction of the full orbifold symplectic cohomology class will face a transversality problem in the current perturbation scheme of pseudo-holomorphic curves, since some generators are contained in singular locus of the orbifold. Fortunately, the element $\Gamma_W$ will comes from $\Sigma_{J^{-1}, 1}/G_W \simeq L_W/G_W$, which is inside a smooth locus of $[M_W/G_W]$.  Pseudo-holomorphic curves with such input do not suffer transversality issue since it is not contained in a singular locus of $[M_W/G_W]$.

  \begin{lemma}\label{gamma is closed}
  Suppose $H_{S^1}$ is $G$-equivariant, $H_{S^1}>0$, and $C^2$-small Morse perturbation of $H$ inside a compact region. Then there is no pseudo-holomorphic cylinder of finite energy satisfying
    \[u: S^1\times \R \to [M_W/G_W] \hskip 0.3cm \lim_{s\to \infty} u(t, s)= \Gamma_W(t).\]
  whose output 
    \[\gamma_- (t) := \lim_{s\to -\infty}u(s,t)\]
  does not lie in $\nu\left(\Sigma_{J^{-1}, 1}/G_W\right) \times \{1\}$.  
  \end{lemma}
  \begin{proof}
  At first, we can rule out the case when the output is on outside of a compact region using the idea of \cite{Se06} using action values:  for a non-trivial orbit $\gamma\in\mathcal O(H_{S^1})$ at the end, its action is given by 
    \begin{align*} A_{H_{S^1}}(\gamma) &:= -\int_{S^1}\gamma^*\lambda+\int_0^1H_{S^1}(\gamma(t))dt\\
    &=-2\int_0^1r^2dt+\int_0^1H(\gamma(t))dt+\int_0^1 F(\gamma(t))dt, \hskip 0.2cm (H_{S^1}=H(r)+F(r,t))\\
    &=-\int_0^1r^2+\epsilon \hskip 0.2cm (\epsilon<<1)
    \end{align*}  
  Nontrivial Hamiltonian orbits are appears as a small perturbation of $\Sigma_{g, l}$. An action value of such orbit is dominated by $-l^2$. Therefore, the output $\gamma_-$ cannot be an orbit of level $l>1$ by the positivity of the topological energy. 
  
  Suppose $\gamma_-$ is an orbit of period $0<l \leq1$. Since $u$ provides a homotopy between the loop $\Gamma_W$ and the orbifold loop $(\gamma_-, g)$, the (conjugacy class of) $g$ must be equal to $J^{-1}$. This means that $\gamma_-(0)$ is a fixed point of $\mathcal R_{1-l}$. Because $w_i/h\leq \frac{1}{2}$ for $\forall i$, Reeb flow $\mathcal R_{1-l}$ has a fixed point only at the origin (which is not in $M_W$). This gives the contradiction.
  
   Similarly, if $\gamma_-$ is Morse critical point in a compact region, we must have $\Phi(\gamma_-(0))=\gamma_-(1)$. It must be a fixed point of a monodromy $\Phi_W$. It is impossible because the only fixed point of $\Phi_W$ is the origin, which is not contained in $M_W$. 
  \end{proof}


 Also notice that an orbifold nodal degeneration does not affect a standard analysis of codimension $1$ boundary strata of moduli spaces because they are of codimension two. It follows from the fact that a local model of such degeneration is given by a family 
$z_1z_2 : \left[\C^2/(\Z/n)\right] \to \C$.
Here, a group $\Z/n$ acts on $\C^2$ by $(z_1, z_2) \to (\xi \cdot z_1, \xi^{-1}\cdot z_2)$ with $\xi$ an $n$-th root of unity. 

Therefore operations such as a closed-open map $\mathrm{CO}(\Gamma_W)$, quantum cap action $\cap \Gamma_W$ and popsicle operations $\bold P^{\Gamma_W}_{n , F, \phi}$ on $\mathcal {WF}([M_W/G_W])$ make sense, and we use them to define the $\AI$-operation $\{M_k\}$.

The final important piece, the vanishing of sphere operations can be shown as follows.

\begin{prop}\label{prop:vanishing2}
If $  \sum_{k=1}^n w_i-h \geq 0$, then popsicle spheres with two or more $\Gamma_M$ inputs vanish.
\end{prop}
\begin{proof}
Proof is similar to that of  Proposition \ref{prop:vanishing}.

Suppose such a sphere bubble from popsicle moduli space exist.
As Corollary \ref{degree inequality 2} still holds, we have a sphere component with $N$ insertions of $\Gamma_W$
 with an output $\xi$, satisfying the same inequality \eqref{degree inequality 3}.

From Lemma \ref{lem:filgw}, the degree of generators of symplectic cochains are either non-negative (in the constant orbit cases) or
satisfies the inequality $p+q - \star_{g,-p} \geq 0$ or $p+q \geq \star_{g,-p}$ (for other cases).
Let us denote $$\mu:= \mu_{\mathrm{RS}}(\Sigma_{J^{-1}, 1}) =\frac{2 (\sum_{k=1}^n w_i-h)}{h}$$ for simplicity. 
From Proposition \ref{orbifold degree inequality 1}, we have 
\[\star_{g, l} \geq -\frac{2l(\sum_{k=1}^n w_i -h)}{h} = - l \mu \]
Combining these, we have  
\begin{center}
if $\mu$ is non-negative and $\xi \in F^{-l}$, then $\deg \xi \geq -l\cdot\mu$.
\end{center}
Now, rest of the proof is exactly the same as that of Proposition \ref{prop:vanishing}.
\end{proof}

\begin{remark}
Although we only prove log-Fano/log Calabi-Yau cases, the readers can easily check that the same proof works for when $\mu_{\mathrm{RS}}(\Sigma_{J^{-1}, 1}) = \frac{2 (\sum_1^n w_i -h)}{h} >-\frac{1}{2}$. 
\end{remark}

\appendix

  \section{Moduli space of pseudo-holomorphic curves and perturbations}\label{basic Floer theory}
  We briefly describe a perturbation scheme for moduli spaces of pseudo-holomorphic curves that we use. We refer \cite{S08}, \cite{AS}, \cite{Ab12}, and especially \cite{A10} and \cite{Se18} from which most of the material has been borrowed.  
  
  \subsection{Setup}
   Let $(M^{2n}, \omega = d\lambda)$ be a Liouville manifold with cylindrical ends. By definition, $M$ can be decomposed into a union of a compact part and cylindrical ends 
 \[M=M_{cpt}\bigcup_{\partial M_{cpt}}\partial M_{cpt}\times[1, \infty),\]
 and Liouville flow $Z$ is of the form $Z=r\frac{\partial}{\partial r}$ at the cylindrical end where $r$ is a coordinate of $[1, \infty)$. We also assume $c_1(TM)=0$ to put $\Z$-grading everywhere.
   \begin{itemize}
    \item We will work with a function $H\in C^\infty(M, \R)$ such that $H>0$, $C^2$-small on $M_{cpt}$ and \textit{quadratic at infinity} ($H(x, r)=r^2$), and denote the class of such functions by $\mathcal H(M)$.
  \item Whenever we consider a time dependent perturbation $H_{S^1}=H+F :S^1\times M \to \R$, we assume $H_{S^1}>0$,  $C^2$-small on $M_{cpt}$ so that the time-$1$ periodic orbits of $H_{S^1}$ are non-degenerate. This is true for a generic perturbation. 
    \item a $\omega$-compatible almost complex structure $J$ is called \textit{contact type} if $\lambda\circ J=dr$
  at the end. We denote a class of such almost complex structure by $\mathcal J(M).$  
  \end{itemize}

$\mathcal W$ is  a collection   of  exact properly embedded Lagrangian submanifolds in $M$, such that
if non-compact, $L\cap \partial M_{cpt}$ is a Legendrian submanifold, and $L$ is \textbf{conical at the end}.
  Furthermore, all such $L$ is required to have vanishing relative first Chern class $2c_1(M, L)$. We attach a spin structure and a grading function on each $L$.  
   
Fix a small, time dependent perturbation $H_{S^1}:S^1\times M \to \R$ of $H$. Let $\mathcal O := \mathcal O(M, H_{S^1})$
 be a set of time-1 orbits $\gamma$ of $H_{S^1}$, and 
$\chi(L_0, L_1;H)$
 to be the set of time-1 Hamiltonian chords $a$ of $H$ from $L_0$ to $L_1$ for two Lagrangians $L_0,L_1 \in \mathcal W$.  We use a notation $o_{\gamma }$ (resp. $o_a$) to denote its orientation operator. Its degree $\deg o_\gamma$ (resp. $\deg o_a$) is given by its cohomological Conley-Zehnder index (resp. Maslov index). 
 
    \subsection{Moduli spaces}
  The moduli space of the Riemann sphere with $m+1$ marked points ($m$ for positive punctures and $1$ for the negative puncture) is denoted as $S_{m,1;0}$. The moduli space of a disc with $m$ interior marked points $\{z_1^+, \ldots z_m^+\}$ (all for positive punctures)
  and with $n+1$ cyclically ordered boundary marked points ($\{z_1,\ldots,z_n\}$ for positive punctures and $\{z_0\}$ for the negative puncture) is denoted  as $S_{m;n,1}$.  An oriented real blow-up of $\OL S$ at marked points is a surface $\vert S$ such that
boundary puncture $z_k$ (resp. interior puncture $z_l^+$) is replaced by $[0,1] \times \{z_k\}$ (resp. $S^1\times \{z_l^+\}$).
 Let us denote (with coordinate $(s,t)$)
 $$Z^+ = [0,\infty)\times[0,1], \;Z^- = (-\infty,0]\times[0,1], \;C^+ = [0, \infty)\times S^1, \;C^- = (-\infty,0]\times S^1.$$
  \begin{defn}
   A set of ends $\mathfrak S$ for $S\in S_{m;n,1}$ is a choice of 
   \begin{itemize}
     \item \textit{strip-like ends} $\epsilon_k^\pm: Z_\pm \to S, \; \lim_{s\to \pm\infty}\epsilon_k^\pm(s,t) =z_k, \; (\epsilon_k^\pm)^{-1}(\partial S) =\{t=0,1\}$ 
     \item \textit{cylindrical ends} $\delta_l^\pm: C_\pm \to S, \; \lim_{s \to \pm \infty}\delta_l^\pm(s,t) =z_l^\pm$ 
   \end{itemize}
   Such collection is said to be \textit{weighted} if each strip and cylinder is endowed with a positive real number 
   \begin{itemize}
     \item $w^\pm_{S, k}$ for each strip-like end $\epsilon^k_\pm$
     \item $v^\pm_{S, l}$ for each cylindrical end $\delta^l_\pm$
   \end{itemize}
  such that $\sum w^+_{S, k} +\sum v^+_{S, l}=\sum w^-_{S, k} +\sum v^-_{S, l}$.
  \end{defn}
  The choice of ends (without weights) is a choice complex coordinate $z$ of $S$ near punctures. It also provides an analytic coordinate $(\sigma,t)$ of $|S$ near puncture in the following way.
  \[\sigma = e^{\mp \pi s}= |z|, \hskip 0.2cm t=\mp \mathrm{arg}(z)/\pi\]
We denote Deligne-Mumford compactification of a family of surfaces $\vert S$ by $\OL {\vert S}_{m;n, 1}$. 

\begin{defn}  
  Let $(S,\mathfrak S)$ denote a holomorphic disc $S$ with ends $\{\kappa\}$ with weight $\{\nu_\kappa\}$.  
  \begin{enumerate}
    \item A basic, asymptotically compatible $1$-form $\alpha_S$ is a closed $1$-form on $S$ whose restriction on $\partial S$ vanishes, extends smoothly to $|S$ and $\alpha_S = \nu_\kappa dt$ at the interval / circle at infinity. It implies  
      \[\kappa^*\alpha_S = \nu_\kappa dt+d(e^{\mp\pi s}g_\kappa (e^{\mp\pi s}, t)), \hskip 0.3cm \forall \kappa\]
for sufficiently large $|s|$ and for some smooth function $g_\kappa$ that vanishes for $t=0,1$. 
    \item A secondary, asymptotically compatible $1$-form $\beta_S$ is a sub-closed $1$-form whose restriction on $\partial S$ vanishes, extends smoothly to $|S$, $\beta_S = \nu_\kappa dt$ at the circle at infinity, and vanishes at the intervals at infinity. It implies
    \[\kappa^* \beta_S=\left\{\begin{array}{cc} \nu_\kappa dt+d(e^{\mp\pi s}h_\kappa (e^{\mp\pi s}, t)) & \forall \textrm{cylindrical end $\kappa$}\\ d(e^{\mp\pi s}h_\kappa (e^{\mp\pi s}, t)) & \forall \textrm{strip-like end $\kappa$} \end{array}\right.\]
for sufficiently large $|s|$ and for some smooth function $h_\kappa$ that vanishes for $t=0,1$ whenever $\kappa$ is a strip-like ends.
    \item An $\mathfrak S$-adapted time-shifting map is a function $a_S:\partial \OL S\to [1, \infty)$ such that
      \[\kappa^*a_S=\nu_\kappa, \hskip 0.3cm \forall \textrm{strip-like end $\kappa$}.\]
      Here, $\kappa^*$ should be thought of the restriction along a small neighborhood of $x_k$.
  \item For a fixed Hamiltonian $H\in\mathcal H(M)$, an $S$-dependent Hamiltonian $H_S$ is said to be compatible with $(S, \mathfrak S, H)$ if 
  \[\kappa^*H_S = \frac{H \circ \psi^{\nu_\kappa}}{\nu_\kappa^2}, \hskip 0.3cm \forall \kappa\]
  \item For a fixed absolutely bounded $S^1$-dependent Hamiltonian perturbation $F$, an $S$-dependent  Hamiltonian $F_S$ is said to be compatible with $(S, \mathfrak S, F)$ if it is also absolutely bounded and 
  \[\kappa^* F_S=\left\{\begin{array}{cc} \frac{F\circ \psi^{\nu_\kappa}}{\nu_{\kappa}^2} & \forall \textrm{cylindrical end $\kappa$}\\ 0 & \forall \textrm{strip-like end $\kappa$} \end{array}\right. \]
  \item For a fixed $S^1$or $[0,1]$-dependent almost complex structure $J_t$, an $S$-dependent almost complex structure $J_S$ is called asymptotically compatible to $(S, \mathfrak S, a_S, J_t)$ if it extends smoothly to $|S$ and 
  \[J_p\in \mathcal J(M), \hskip 0.3cm \forall p\in S.\]
It implies 
  \[\kappa^*J_S = \left(\phi^{\nu_\kappa}\right)^*J_t+e^{\mp \pi s}O_{\kappa, (s,t)}, \hskip 0.3cm \forall \kappa\]
  for some error term $O_{\kappa, (s,t)}$. 
  \end{enumerate}
  \end{defn}
  \begin{remark}
  In \cite{A10}, (1), (2), and (6) is required to be strictly compatible to $\mathfrak S$ which means 
  \[\kappa^*\alpha_S = \nu_\kappa dt\]
  and similar for $\beta_S$ and $J_S$. In particular, it depends on the choice $\mathfrak S$ . An idea of asymptotic compatibility is due to \cite{Se18}. It is more flexible because it only depends smoothly on $\vert S$, not on the choice of the set of ends $\{\kappa_S\}$. It allows us to choose a special kind of ends to deal with popsicle structures later on. 
  \end{remark}
  Finally, we define
  \begin{defn}
  For a fixed surface $S$, Hamiltonian $H\in \mathcal H(M)$ and absolutely bounded time-dependent Hamiltonian $F$, a Floer data $\mathrm{Floer}_S$ consists of 
    \begin{enumerate}
      \item A collection of weighted strip and cylinder data $\mathfrak S$;
      \item A basic $1$-form $\alpha_S$ and secondary $1$-form $\beta_S$ asymptotically compatible to $(S,\mathfrak S)$;
      \item An $(S, \mathfrak S)-$adapted time-shifting map $a_S$;
      \item An $S$-dependent, $(S, \mathfrak S, H, F)$-compatible Hamiltonian $H_S$ and $F_S$; 
      \item An $S$-dependent, asymptotically compatible almost complex structure $J_S$. 
    \end{enumerate} 
  Also, we say $\mathrm{Floer}_S^1$ and $\mathrm{Floer}_S^2$ are conformally equivalent if $\mathrm{Floer}_S^2$ is a rescaling by Liouville flow of $\mathrm{Floer}_S^1$, up to constant ambiguity in the Hamiltonian terms.
  \end{defn}
      In the simplest case of a strip $S\in S_{0; 1,1}$ or a cylinder $S\in S_{1,1;0}$, we choose a canonical strip-like/cylindrical end with weights $1$ for all ends. A form $\mathrm{dt}$ is a compatible sub-closed one form.  
  A universal and consistent choice of Floer data is a choice of Floer data $\mathrm{Floer}_S$ for all $S\in S_{m;n,1}$ which varies smoothly over the moduli space. Since the space of Floer data is contractible, we can extend it to $\overline{\vert S}_{m;n,1}$.  
    
  \begin{defn}
    Let $\gamma_i\in \mathcal O$ be a time-1 Hamiltonian orbits and $a_j\in\chi(L_{j-1}, L_j), \hskip 0.2cm j=1, \ldots, n$ and $a_0\in\chi(L_n, L_0)$ be Hamiltonian chords. Define the moduli space 
$\overline{\mathcal M}_{m;n,1}(\gamma_1, \ldots, \gamma_m; a_1, \ldots, a_n, a_0)$
of maps 
$$\left\{u:S\to M \Big| S\in \overline S_{m; n,1}\right\}$$
    satisfying the inhomogeneous Cauchy-Riemann equation with respect to $J_S$ 
\begin{equation}\label{eq:cr}
\left(du-X_{H_S}\otimes \alpha_S-X_{F_S}\otimes \beta_S\right)^{0,1}=0
\end{equation}
    and the following asymptotic boundary conditions;  
    \begin{gather*}
    \lim_{s\to \infty} u\circ \epsilon^k_+(s, \cdot)=a_k \\
    \lim_{s\to -\infty} u\circ \epsilon^0_-(s, \cdot)=a_0 \\
    \lim_{s\to \infty} u\circ \delta^l_+(s, \cdot)=\gamma_l\\
    u(z) \in \psi^{a_S(z)}L_i, \;\; z\in \partial_iS, \;\; \textrm{an $i$-th boundary component of $S$.}
    \end{gather*}
    Additional perturbation terms $\alpha_S, \beta_S, H_S, F_S$ come from the universal and consistent choice of Floer data $\mathrm{Floer}_{S}$.
  \end{defn}
  The following compactness and transversality result is standard. 
  \begin{lemma} For a generic choice of universal and consistent Floer data, 
  \begin{enumerate}
    \item The moduli spaces $\overline{\mathcal M}(\gamma_1, \ldots, \gamma_m; a_1, \ldots, a_n, a_0)$ are compact. 
    \item For a given input $\gamma_i, \hskip 0.2cm i=1,\ldots ,m$ and $a_j, \hskip 0.2cm j=1, \ldots, n$, there are only finitely many $a_0$ for which $\overline{\mathcal M}_{m;n,1}(\gamma_1, \ldots, \gamma_m; a_1, \ldots, a_n, a_0)$ is non-empty .
    \item It is a manifold of dimension 
    \[\mathrm{dim_\R}\overline{\mathcal M}_{m;n,1}(\gamma_1, \ldots, \gamma_m; a_1, \ldots, a_n, a_0) = (2m+n-2)+ \deg o_{a_0}-\sum_{i=1}^m\deg o_{\gamma_i}-\sum_{j=1}^n\deg o_{a_j}\]
  \end{enumerate}
  \end{lemma}
  \begin{proof}
    See \cite{A10}. For a compactness result, one need to assure that the energy of pseudo-holomorphic curves are a priori bounded in $M$. This estimate is carefully done therein. Transversality result is a standard application of Sard-Smale argument. The dimension formula is also a standard application of Atiyah-Singer index theorem on a linearized Fredholm operator. 
  \end{proof}
  When $\mathcal M_{m;n,1}(\gamma_1, \ldots, \gamma_m; a_1, \ldots, a_n, a_0)$ has dimension zero so that it is rigid, then a map $u:S\to M$ in that moduli space is isolated. An orientation of the moduli space provides a canonical isomorphism 
  \[Q_u : \bigotimes_{i=1}^m o_{\gamma_i} \otimes \bigotimes _{j=1}^n o_{a_j} \to o_{a_0}.\]
  We sum up $Q_u$ for all $u\in \mathcal M(\gamma_1, \ldots, \gamma_m; a_1, \ldots, a_n, a_0)$ and all $a_0$ and define
  \[\textbf{F}_{m;n,1}(\gamma_1, \ldots, \gamma_m; a_1, \ldots, a_n)=\sum_{\mathrm{dim_\R}\mathcal M(\gamma_1, \ldots, \gamma_m; a_1, \ldots, a_n, a_0)=0}\sum_{u\in\mathcal M(\gamma_1, \ldots, \gamma_m; a_1, \ldots, a_n, a_0)} Q_u\left(\bigotimes_{i=1}^m o_{\gamma_i} \otimes \bigotimes _{j=1}^n o_{a_j}\right)\]
  We define $\OL{\mathcal M} (\gamma_1, \ldots, \gamma_m, \gamma_0)$ and $\textbf{F}_{m,1;0}$ in a similar way.

 Let us first recall the our setup and explain the detailed construction.

\subsection{Wrapped Fukaya category and symplectic cohomology}

Symplectic cohomology $SH^\bullet(M)$ is a version of Hamiltonian Floer cohomology for Liouville domain introduced by Cieliebak, Floer and Hofer \cite{CFH95} and Viterbo \cite{Vi99}. For Lagrangian submanifolds in $M$ that are either compact or cylindrical at infinity, a wrapped Fukaya $\AI$-category $\WF(M)$ was defined by Abouzaid-Seidel \cite{AS}.
We use the version with quadratic Hamiltonian of Abouzaid \cite{Ab12}   which we recall briefly here.

For two Lagrangian submanifolds $L_0, L_1 \in \mathcal W$, a \textit{wrapped Floer cochain complex} is a vector space 
  \[CW^\bullet(L_0, L_1;H)=\bigoplus_{a\in \chi(L_0, L_1;H)} o_a\]
  It is graded by the degree $\deg o_a$. We will use the notation $a$ instead of $o_a$ for generators, and $CW^\bullet(L_0,L_1)$ instead of 
  $CW^\bullet(L_0,L_1;H)$  if it cause no confusion. 
  \begin{defn}
  A wrapped Fukaya category $\WF(M)$ consists of a set of objects $\mathcal W$ with the space of morphisms $CW^\bullet(L_0, L_1)$ for $L_i\in \mathcal W$ equipped with an $A_\infty$ structure 
        \begin{align*}
        m_k: CW^\bullet(L_0, L_1) \otimes &\cdots \otimes CW^\bullet(L_{k-1}, L_k) \to CW^\bullet(L_0, L_k)\\
        m_k(a_1, \ldots, a_k)&= (-1)^{\diamond_k} \textbf{F}_{0;k,1}(a_1, \ldots, a_k);\\
        \diamond_k&=\sum_1^k i \cdot \deg a_i
        \end{align*} 
  \end{defn}
The Liouville flow $\psi^\rho$ for time $\mathrm{log}(\rho)$ defines a canonical isomorphism 
  \[CW^\bullet(L_0, L_1; H, J_t) \simeq CW^\bullet\left(\psi^\rho L_0, \psi^\rho L_1; \frac{H\circ \psi^\rho}{\rho}, (\psi^\rho)^*J_t \right)\]
  Also, $H\circ \psi^\rho =\rho^2 r^2$ at the cylindrical end. Therefore $\frac{H\circ \psi^\rho}{\rho^2} \in \mathcal H(M)$. 
The proof of $A_\infty$ relation follows from the degeneration patterns of pseudo-holomorphic discs which correspond to a codimension $1$ boundary strata of Gromov bordification $\overline{\mathcal M}_{0;k,1}(a_1, \ldots, a_k,a_0)$. In particular, we have $m_1^2=0$.
Its $m_1$-cohomology $HW^\bullet(L_0, L_1)$ is called the wrapped Floer cohomology.

 We are also interested in the symplectic cohomology of Liouville manifold. 
  \begin{defn}
  A \textit{symplectic cochain complex} is a $\Z$-graded cochain complex
  \[CH^\bullet(M;H_{S^1})=\bigoplus_{\gamma \in \mathcal O(M;H_{S^1})} o_\gamma \]
  graded by the degree $\deg o_\gamma$. We will use the notation $\gamma$ instead of $o_\gamma$ for generators if it cause no confusion. A differential of this complex is 
  \[\der_{CH} (o_{\gamma_1}) = (-1)^{\deg o_{\gamma_1}}\textbf{F}_{1,1;0}(\gamma_1).\]
  Recall that $\textbf{F}_{1,1;0}(\gamma_1)$ is given by a counting of a zero-dimensional component of a moduli space of pseudo-holomorphic cylinders $\mathcal M_{1,1;0}(\gamma_1, \gamma_0)$. Its cohomology is denoted as $SH^\bullet(M)$.
    \end{defn}
 A ring structure on its cohomology is induced from 
  $$  CH^\bullet(M)^{\otimes 2}  \to CH^\bullet(M): 
    (\gamma_1, \gamma_2)  \mapsto (-1)^{\deg \gamma_1} \textbf{F}_{2,1;0}(\gamma_1, \gamma_2) $$
    Let us recall the definition of a closed-open map to Hochschild cochain complex of wrapped Fukaya category.
    \begin{defn}\label{closed-open map}
    A closed-open map is a map
    \begin{align*}
    \mathrm{CO}=\{\textrm{CO}_l\} _{l\geq0}: CH^\bullet(M) &\to CC^\bullet\left(\WF(M), \WF(M)\right)\\
    \textrm{CO}_l(\gamma)(a_1, \ldots, a_l)&:= (-1)^{\diamond_l} \textbf{F}_{1;l,1}(\gamma; a_1, \ldots, a_l, a_0)
     \end{align*}
  \end{defn}
  A degeneration pattern of a moduli space $\overline {\mathcal M}_{1;n,1}(\gamma;a_1, \ldots, a_n, a_0)$ proves that $\textrm{CO}$ is a cochain map.

\section{Popsicle moduli space and its charts}\label{compactification}
In this section, we prove the following theorem.
 \begin{thm}\label{thm:corner}
The compactified moduli space  of popsicles $\overline{P}_{n, \widetilde F, \phi}$ is a manifold with corners.
\end{thm}

For the compactification, we need to consider a mixture of two disjoint degenerations; one is when an underlying disc component breaks into several pieces, and  the other is when several sprinkles collide. 
 The first part can be covered by the result of \cite{AS}, and we will focus on the second type of degeneration.  To prove the theorem, we will describe a coordinate chart of a given stable popsicle with an alignment data $(\Phi, \Psi)$. Let us first recall the rational strip-like/cylindrical ends from \cite{Se18}. 

	\begin{defn}
	A choice of strip-like ends $(\epsilon_0^-; \epsilon_1^+, \ldots, \epsilon_k^+)$ is called rational if
	\begin{enumerate}
	\item $\epsilon_0^-$ extends to a map $\epsilon_0^-: Z \to \overline \Sigma$ such that $\epsilon_0^-(-\infty)=z_0$.
	\item $\epsilon_k^+$ extends to a map $\epsilon_k^+: Z \to \overline \Sigma$ such that $\epsilon_k^+(-\infty)=z_0$ and $\epsilon_k^+(+\infty)=z_k$. 
	\end{enumerate}
	Likewise, a choice of cylindrical ends $(\eta_0^-; \eta_1^+, \ldots, \eta_k^+)$ is called rational if 
	\begin{enumerate}
	\item $\eta_0^-$ extends to a map $\eta_0^-: \hat \C \to \overline {\Sigma}^+$ such that $\eta_0^-(0)=w_0$.
	\item $\eta_k^+$ extends to a map $\eta_k^+: \hat \C \to \overline {\Sigma}^+$ such that $\eta_k^+(0)=w_0$ and $\epsilon_k^+(\infty)=w_k$. 
	\end{enumerate}
	If $\Sigma\simeq Z$ is semi-stable with two boundary marking,  two choices of strip-like ends are equivalent if they are differ by a translation of $Z$. In a similar way, if $\Sigma^+ \simeq \C^*$ is semi-stable with two interior markings,  two choices of cylindrical ends are equivalent if they are differ by a dialation of $\C^*$.
	\end{defn}

A rational end  at the negative end  can be viewed as a choice of a particular planar models of popsicles:  $\epsilon_0^-$ determines an isomorphism $\overline \Sigma \simeq \mathbb H$ which sends $z_0$ to $\infty$, and the image of $\epsilon_0^-$ is a complement of a unit half disc $\{\vert z \vert \geq1\}$. The rest of its boundary markings can be identified with points $x_k \in \mathbb R$, and their rational ends $\epsilon_k^+$ determine semicircle centered at $x_k$ of radius $\rho_k$. Note that equivalent strip-like ends have the same image of  $\epsilon$-maps, hence radius is well-defined.
 
Similarly, a rational end $\eta_0^-$ determines an isomorphism $\overline \Sigma^+ \simeq \hat{\mathbb C}$ which sends $w_0$ to $\infty$, and the image of $\eta_0^-$ is a complement of a unit disc $\{\vert w\vert \geq1\}$. Then, the rest of the markings $w_k^+$ can be identified with a complex number $\{z_k\}$ and the rational end $\eta_k^+$ is determined by a disc with radius $\xi_k$ centered at each $z_k$.

Next, we discuss the alignment of ends.

\begin{defn}
\label{aligned ends}
Suppose a popsicle disc $\Sigma$ and a popsicle sphere $\Sigma^+$ are aligned to each other. A choice of strip-like/cylindrical ends for $\Sigma$ and $\Sigma^+$ is said to be \textbf{aligned} if the following condition holds:
	\begin{itemize}
	\item Strip-like ends of $\Sigma$ and cylindrical ends of $\Sigma^+$ are rational.
	\item If you compare two planar models of $\Sigma$ and $\Sigma^+$ associated to $\epsilon_0$ and $\eta_0$, the $x$-coordinates for corresponding popsicle lines are exactly the same.
	\item If $z^+\in \Sigma$ be a sprinkle on a disc which lies on a popsicle line coming out of $z_k \in \mathbb R$, then the image of a cylindrical end for $z^+$ is a disc of radius $\rho_k$. Here $\rho_k$ is a radius of a strip-like end for $z_k$. 
	\item If $w\in \Sigma^+$ is a sprinkle on a sphere which lies on a line corresponds to $z_k \in \mathbb R$, then the image of a rational cylindrical end for $w$  is disc of radius $\rho_k$. Here $\rho_k$ is a radius of a strip-like end for $z_k$. 
 \end{itemize}
 If every ends on a broken popsicle satisfy these conditions, we say that choices of ends are \textbf{aligned}. See Figure \ref{fig:aligned ends}.
 \end{defn}
 
 	\begin{figure}[h]
	\centering
	\def\svgwidth{10cm}
	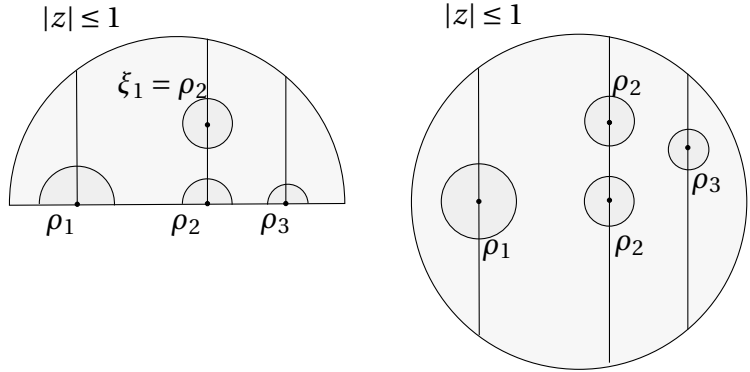
	\caption{Example of aligned pairs with the choice of aligned ends}
	\label{fig:aligned ends}
	\end{figure}

One can choose the ends of a popsicle smoothly and consistently along the moduli space while the choices are all aligned. From now on, we assume that all choice of ends are always aligned. 
 
Let us describe a corner chart for a point in $P_{n, \widetilde F, \phi, \star} \subset \overline P_{n, \widetilde F, \phi}$. 
As usual, a corner chart is described in terms of gluing parameters, but alignment of conformal structure forces us to align the gluing parameters as well.
We will describe how to align gluing parameters according to the alignment data.

For a broken popsicle, we will select a  collection of gluing parameters of outgoing marked points, one for each aligned disc/spheres pair. 

Describing a gluing procedure requires a little more work. We will first define, what we call, an \textbf{auxiliary model}, by adding suitable semi-stable discs and spheres in the following way.

	\begin{enumerate}
	\item (adding semi-stable discs) For a given data $\Psi_e: \Phi^{-1}(e) \to \{1, \ldots, m\}$, we add a string of $m$ many semi-stable discs at the nodal point corresponding to the edge $e$. If $\Psi_e(w) = k$,  the corresponding sphere component of $w$ is aligned with $k$-th semi-stable disc component.
	
	\item (adding semi-stable spheres) Let $w\in \mathrm{Vertex}(T_{v, i, j})$ be a vertex connected to $v \in \mathrm{Vertex}(T)$. If $\Phi(w)$ is adjacent to $v$ (even after adding semi-stable discs), then we do nothing. If not, consider a unique path from $v$ to $\Phi(w)$. Then we add semi-stable popsicle spheres (at the output of the component $w$) which are aligned to  disc components along this path. Here, each added semi-stable popsicle sphere have one input and one output, and aligned to the corresponding disc component. 
	
	Similarly, let $w_1<w_2$ are two adjacent vertices of $T_{v,i,j}$. If $\Phi(w_1)$ and $\Phi(w_2)$ are not adjacent, take a unique path from $\Phi(w_1)$ and $\Phi(w_2)$, and we add corresponding semi-stable popsicle spheres for disc components along this path. Here, each sphere has one input and one output, and aligned to the corresponding disc component.  
	
	We apply the same procedure to leaves. Let $l$ be a leaf of $T_{v,i,j}$ attached to a vertex $w$ with induced color $m$. Then we add a semi-stable sphere for each disc component along a unique path from $\Phi(w)$ to the leaf $l_m$. Again, each sphere component is aligned to a corresponding disc component. 
	\item (choose aligned ends) We extend the choice of ends to added discs and spheres so that the overall choice is still aligned. Since we only add semi-stable components, we can copy-and-paste original choice of ends without modification. 
	\end{enumerate}
\begin{remark}
We record a minimal amount of information in the alignment data. Auxiliary model recovers the full alignment from the prescribed ones by adding suitable semi-stable discs and spheres. They are just tracking devices for aligned components and compatible choices of ends so that the gluing process is well-defined. 
For example, in the gluing model of the broken popsicle described in Figure \ref{fig:popdeg2}, two disc neighborhoods of interior marked points as well as the upper-half of the annulus (between radius $r_1r_2$ and $r_2$) corresponds to the added semi-stable spheres and discs. The broken popsicle in Figure  \ref{fig:popdeg1} does not require additional semi-stable components.
\end{remark}
We call the resulting surface with decoration the auxiliary model for a broken popsicle of type $\star$. Every sphere component in the auxiliary model is aligned to a  disc component. Moreover, if several spheres are aligned to a single disc, at least one of the disc or spheres are stable. Therefore, we can move back and forth between stable broken popsicles and its auxiliary model by simply forgetting (and recovering) semi-stable components.

The gluing process takes place in two stages. At first, we replace a broken popsicle to its auxiliary model. Second, we turn on gluing parameters and obtain a new auxiliary model. Finally, we contract semi-stable components to get a stable broken popsicle.

The process is well-defined in the following sense. In the auxiliary model, suppose that we are gluing  $\Sigma_1$ and $\Sigma_2$
with a gluing parameter $\epsilon$.  If $\Sigma_1'$ is aligned to $\Sigma_1$ and $\Sigma_2'$ is aligned to $\Sigma_2$, then our  procedure glues two pairs $(\Sigma_1, \Sigma_2)$ and $(\Sigma_1', \Sigma_2')$ simultaneously. Since our choice of ends are compatible with respect to the alignment, they are still aligned after gluing.

  Rational ends help us to put a canonical popsicle structure on a gluing of two. Let $S_\gamma$ be a gluing of $S_1 \in P_{n_1, F_1}$ and $S_2 \in P_{n_2, F_2}$ along $z_{0,1} \in S_1$ and $z_{i,2}\in S_2$. Then positive boundary markings of $S_\gamma$ are $(n_1+n_2-1)$ points on $\partial \mathbb H$:
  \[z_k \leftrightarrow x_{k, \gamma} = \left \{ 
   \begin{array}{ll}
   x_{k,1} & k<i ,\\ 
   x_{i, 1}+ (\rho_{1, i}\times \gamma)x_{k-i+1,2} & i\leq k\leq i+n_2-1 ,\\ 
   x_{k-n_2+1,1} & i+n_2 \leq k.
   \end{array} 
   \right .\]
  Moreover, sprinkles are given as: 
  \[z_l^+ \leftrightarrow x_{\phi(f), \gamma}+y_{f, \gamma} = \left \{ 
  \begin{array}{ll} 
  x_{\phi_1(f),1}+i y_f & f \in F_1, \hskip 0.2cm \phi_1(f) \neq i ,\\ 
  \left( x_{i, 1}+ (\rho_{1, i}\times \gamma)x_{\phi_1(f)-i+1,2} \right) + i y_f & f\in F_1  \hskip 0.2cm \phi_1(f)=i, \\ 
  \left( x_{i, 1}+ (\rho_{1, i}\times \gamma)x_{\phi_2(f),2} \right) + i y_{\phi_2(f)}  & f\in F_2.
  \end{array} 
  \right .\]

Let us describe more details of the gluing, and how a combinatorial type of general popsicle changes afterward. Every finite edge $e\in \mathrm{Edge}(\widetilde T)$ there is the unique vertex $v$ such that $e=e_{v,0}$. We denote a gluing parameter associated to $e$ by $\epsilon_v$. Recall that if $v_1$ and $v_2$ are aligned to each other, then $\epsilon_{v_1}$ and $\epsilon_{v_2}$ must be turned on simultaneously. Therefore, we get two types of gluing parameters. 
	\begin{enumerate}
	\item for each $v \in \mathrm{Vertex}(T)$, we get a single parameter from a tuple $(\epsilon_v, \epsilon_{w_1}, \ldots, \epsilon_{w_{n_v}})$ where  $\{w_1, \ldots, w_{n_v}\}=\Phi^{-1}(v)$. We simply denote it by $\epsilon_v$. 
	\item for each $e\in \mathrm{Edge}(T)$ such that $\Phi^{-1}(e)\neq \emptyset$, we get $m_e$-many parameters, one for each $1\leq k\leq m_e$,  from each tuple $(\epsilon_{w_1}, \ldots, \epsilon_{w_{n_{e, k}}})$ where $\{w_1, \ldots, w_{n_{e,k}}\} = \Psi_e^{-1}(k)$. We simply denote them by $\epsilon_{e, k}$.
	\end{enumerate}
	
Now let us see what happens on a combinatorial type when we turn on a single gluing parameter $\epsilon$. 

	\begin{enumerate}
	\item[(Case 1)]:  $\epsilon=\epsilon_v$ for some $v\in \mathrm{Vertex}(T)$ and $\Phi^{-1}(e_{v,0})=\emptyset$. 
	Let us say the edge $e_{v,0}$ goes from $u$ to $v$. Because of the second condition, popsicle discs $\Sigma_u$ and $\Sigma_v$ are still adjacent in the auxillary model. 
	They are glued together if we turn on $\epsilon_v$. It means that the edge $e_{v,0}$ is contracted and two vertices $u$ and $v$ become a single vertex $v'$ after gluing. 	
	Now we have to see what happen to the vertices which are aligned to $v$. 
	\begin{enumerate}
	\item  Suppose you see two vertices $w_1<w_2 \in \mathrm{Vertex}(\widetilde T\setminus T)$ connected by an edge $e'$ such that $\Phi(w_1)=u$ and $\Phi(w_2) = v$. Two components $\Sigma^+_{w_1}$ and $\Sigma^+_{w_2}$ will be glued together with
	the same gluing parameter $\epsilon_v$.
	 For the trees, the edge $e'$ is also contracted and two vertices $w_i$ are replaced with a single vertex $w'$. Now, we align $w'$ to $v'$ for an obvious reason.
	\item Suppose $\Phi(w)=v$ , but there is no adjacent $w'$ such that $\Phi(w')=u$. There will be no change in combinatorial type except $\Phi(w)$ becomes $v'$ after gluing.  
	\item Suppose $\Phi(w)=u$, but there is no adjacent $w'$ such that $\Phi(w')=v$. There will be no change in combinatorial type except $\Phi(w)$ becomes $v'$ after gluing, similar as above. 
	\item Suppose you see $w \in \mathrm{Vertex}(T_{u, i, j})$ such that $\Phi(w)=v$ and the incoming edge $e_{0, w}$ is actually the root of $T_{u,i,j}$ connected to $u$. It means $\Sigma^+_w$ and $\Sigma_v$ are both adjacent to $\Sigma_u$, and they will be glued together to a single disc. Therefore the vertices $u, v, w$ and the edge $e$ and $e_{0, w}$ are contracted to a single vertex $v'$. Also, a tree $T_{u, i, j}$ becomes $T_{u, i, j}\setminus \{w, e_{0, w}\}$, which is a disjoint union of $(\mathrm{val}(w)-1)$-many trees. A decomposition of $F_{u, i, j}$ and its $F$-label changes accordingly.
	\end{enumerate}
	\item[(Case 2)]: $\epsilon = \epsilon_v$ for some $v \in \mathrm{Vertex}(T)$ and $\Phi^{-1}(e_{v,0})\neq \emptyset$. 
	Again, let us say that $e_{v,0}$ goes from $u$ to $v$. 
	The second condition implies that there will be a string of $m_e$-many semi-stable components in between $\Sigma_u$ and $\Sigma_v$. 
       Let $w_1$ be a stable sphere component that is aligned to the node $e_{v,0}$ and $\Psi_{e_{v,0}}(w_1) = m_e$. 
    	In the auxiliary model, we should have another sphere component (either stable or semi-stable) that is aligned with $v$, and we denote it by $w_2$.
	
	Turning on $\epsilon_v$ will glue $\Sigma_v$ and the $m_e$-th semi-stable  disc components (call the resulting component by $v'$), and the same parameter will also glue two sphere components $\Sigma^+_{w_1}$ and $\Sigma_{w_2}^+$ as well (call the resulting component by $w_1'$). For the resulting broken popsicle, the tree $T$ remains the same, and now $w_1'$ should be
	aligned to $v'$. Note that $m_e$ decrease by one.
%
%
	\item[(Case 3)]:  $\epsilon = \epsilon_{e,k}$ for some $e\in\mathrm{Edge}(T)$ and $1\leq k \leq m_e$. 
		 $\epsilon_{e,k}$ is a gluing parameter for the $k$-th and $(k+1)$-th semi-stable disc components (say $u_1$ and $u_2$) attached at the node $e$.
		 We can identify it with the corresponding gluing parameters for the aligned sphere components: 
		 Consider any two (semi-stable or stable) sphere components in the auxiliary model, say $w_1$ and $w_2$ that are aligned to $u_1$ and $u_2$ respectively. 
		 Turning on  $\epsilon_{e,k}$, we glue $u_1$ and $u_2$  (call the resulting component $u'$) and with the same parameter, we glue $w_1$ and $w_2$  (call the resulting component $w'$).
		 Then, $w'$ is now aligned to $u'$.
	 	 This process will reduce $m_e$ by one, and we need to reassign the function $\Psi_e$ accordingly: for all the vertices $w''$ such that $\Phi(w'')=e$ and $\Psi_e(w'')\geq k+1$, we lower the value $\Psi_e(w'')$ by one. This is because the function $\Psi_e$ measures the relative order of the sphere bubbles.	 
	     In this way, the rooted ribbon tree $T$ for the resulting glued popsicle  remains the same and only the associated alignment data is modified.	
	
	\end{enumerate}	
	
%

In any cases, a gluing procedure decreases either the number $\mathrm{Vertex}(T)$, or decreases a certain $m_e$ by one. In fact, the following lemma can be obtained from the above description.
\begin{lemma}
$\overline P_{n, \widetilde F, \phi}$ is a manifold with corners of dimension $n-2+\vert \widetilde F \vert$, and 
	\begin{equation}
	\label{codimension formula}
	\mathrm{codim}(P_{n, \widetilde F, \star})= \vert \mathrm{Vertex} (T) \vert -1 + \sum_e m_e 
	\end{equation}
\end{lemma}

\bibliographystyle{amsalpha}
\bibliography{FukayaSing}

\end{document}